\documentclass[10pt]{article}
\usepackage{epsf}
\usepackage{amsmath}

\allowdisplaybreaks

\usepackage[showframe=false]{geometry}
\usepackage{changepage}

\usepackage{epsfig}
\usepackage{amssymb}

\usepackage{amsthm}
\usepackage{setspace}
\usepackage{cite}

\usepackage{algorithmic}  
\usepackage{algorithm}

\usepackage{shadow}
\usepackage{fancybox}
\usepackage{fancyhdr}

\usepackage{color}
\usepackage[usenames,dvipsnames,svgnames,table]{xcolor}
\newcommand{\bl}[1]{\textcolor{blue}{#1}}
\newcommand{\red}[1]{\textcolor{red}{#1}}
\newcommand{\gr}[1]{\textcolor{green}{#1}}

\definecolor{mypurple}{rgb}{.4,.0,.5}

\usepackage[hyphens]{url}

\usepackage[colorlinks=true,
            linkcolor=black,
            urlcolor=blue,
            citecolor=purple]{hyperref}

\usepackage{breakurl}

\def\w{{\bf w}}

\def\y{{\bf y}}

\def\x{{\bf x}}

\def\x{{\mathbf x}}

\def\w{{\bf w}}

\def\x{{\bf x}}
\def\y{{\bf y}}

\def\h{{\bf h}}

\def\be{\begin{equation}}
\def\ee{\end{equation}}
\def\ba{\left[\begin{array}}
\def\ea{\end{array}\right]}

\def\w{{\bf w}}

\def\x{{\bf x}}
\def\y{{\bf y}}

\def\1{{\bf 1}}

\def\g{{\bf g}}
\def\0{{\bf 0}}

\def\erfinv{\mbox{erfinv}}
\def\erf{\mbox{erf}}
\def\erfc{\mbox{erfc}}

\def\erfinv{\mbox{erfinv}}

\def\Sw{S_w}

\def\hw{\bar{\h}}

\def\Sw{S_w}

\def\mR{{\mathbb R}}

\def\psiint{\Psi_{int}}
\def\psiext{\Psi_{ext}}
\def\psicom{\Psi_{com}}
\def\psinet{\Psi_{net}}

\def\psiintnon{\Psi_{int}^{+}}
\def\psiextnon{\Psi_{ext}^{+}}
\def\psicomnon{\Psi_{com}^{+}}
\def\psinetnon{\Psi_{net}^{+}}

\def\lp{\left (}
\def\rp{\right )}

\newtheorem{theorem}{Theorem}

\begin{document}

\begin{singlespace}

\title {Random linear systems with sparse solutions -- asymptotics and large deviations 
}
\author{
\textsc{Mihailo Stojnic \footnote
{e-mail: {\tt flatoyer@gmail.com}} }}
\date{}
\maketitle

\centerline{{\bf Abstract}} \vspace*{0.1in}

In this paper we revisit random linear under-determined systems with sparse solutions. We consider $\ell_1$ optimization heuristic known to work very well when used to solve these systems. A collection of fundamental results that relate to its performance analysis in a statistical scenario is presented. We start things off by recalling on now classical phase transition (PT) results that we derived in \cite{StojnicCSetam09,StojnicUpper10}. As these represent the so-called breaking point characterizations, we now complement them by analyzing the behavior in a zone around the breaking points in a sense typically used in the study of the large deviation properties (LDP) in the classical probability theory. After providing a conceptual solution to these problems we attack them on a ``hardcore" mathematical level attempting/hoping to be able to obtain explicit solutions as elegant as those we obtained in \cite{StojnicCSetam09,StojnicUpper10} (this time around though, the final characterizations were to be expected to be way more involved than in \cite{StojnicCSetam09,StojnicUpper10}; simply, the ultimate goals are set much higher and their achieving would provide a much richer collection of information about the $\ell_1$'s behavior). Perhaps surprisingly, the final LDP $\ell_1$ characterizations that we obtain happen to match the elegance of the corresponding PT ones from \cite{StojnicCSetam09,StojnicUpper10}. Moreover, as we have done in \cite{StojnicEquiv10}, here we also present a corresponding LDP set of results that can be obtained through an alternative high-dimensional geometry approach. Finally, we also prove that the two types of characterizations, obtained through two substantially different mathematical approaches, match as one would hope that they do.

\vspace*{0.25in} \noindent {\bf Index Terms: Linear systems of equations; sparse solutions;
$\ell_1$-heuristic; large deviations}.

\end{singlespace}

\section{Introduction}
\label{sec:back}

In this paper we will study some of the fundamental properties of random systems of linear equations. These problems are well known and have been the subject of extensive mathematical studies over last several decades. Consequently, quite a few of their nice mathematical features have been discovered and explained. Discussing all of these we leave for a review type of paper and instead focus on those of interest for the line of studying that we will pursue here. Along the same lines, we will introduce many mathematical objects/facts that we will rely on without too much detailing essentially assuming a high degree of familiarity with a somewhat lengthy line of work initiated in \cite{StojnicCSetamBlock09,StojnicUpperBlock10,StojnicICASSP09block,StojnicJSTSP09} and continued in large collection of our papers that followed in their footsteps. As usual, we will try to maintain as much consistency with \cite{StojnicCSetamBlock09,StojnicUpperBlock10,StojnicICASSP09block,StojnicJSTSP09} as possible hoping that even less experienced readers will enjoy a smooth transition from these earlier works and the topics that we will discuss here.

To put everything on the right mathematical track, we start with the standard description of the linear systems. Let $A$ be an $m\times n$ ($m\leq n$) dimensional matrix (we may often call $A$ throughout the paper the system matrix). Further, assume that its entries are real numbers and let $\tilde{\x}$ be an $n$ dimensional vector that also has real entries (for short we say $A\in \mR^{m\times n}$ and $\tilde{\x}\in \mR^{n}$). Additionally, we will call $\tilde{\x}$ $k$-sparse if it has no more than $k$ nonzero entries. Then the standard linear system is formed through the product of $A$ and $\tilde{\x}$. Let this product be $\y$ and we write
\begin{equation}
\y=A\tilde{\x}. \label{eq:defy}
\end{equation}
Now, the standard linear system problem is to in fact determine $\tilde{\x}$ if $A$ and $\y$ in (\ref{eq:defy}) are given. There is one thing that one should emphasize though. Namely, by the formation of $\y$ it is clear that such an $\tilde{\x}$ exists. What is way less clear is that such an $\tilde{\x}$ may not be unique, i.e. there may be more than one $\tilde{\x}$ for which (\ref{eq:defy}) holds for given $A$ and $\y$. This will not be possible if $m\leq n$ and $A$ is full rank. On the other hand, if $A$ is indeed full rank but $m<n$ (in such a case we of course call the systems under-determined) this indeed becomes possible. Precisely this under-determined full-rank system will be one of the main topics of this paper. Still, we will through structuring $\tilde{\x}$ in a way ensure that the solution is almost always unique. The type of the structure that we will assume is the called sparsity of $\tilde{\x}$.  Mathematically, we will look at a structured variant of (\ref{eq:defy}) and ask for the $k$ sparse solution of
\begin{equation}
A\x=\y, \label{eq:system}
\end{equation}
knowing of course (based on (\ref{eq:defy})) that such a solution exists. It is an easy algebraic exercise to show that for $k<m/2$ the solution of (\ref{eq:system}) is in fact unique; moreover, we will also additionally assume that there is no $\x$ that satisfies (\ref{eq:system}) that is less than $k$ sparse. Assuming then all of this, we will often instead consider then the following problem
\begin{eqnarray}
\mbox{min} & & \|\x\|_{0}\nonumber \\
\mbox{subject to} & & A\x=\y, \label{eq:l0}
\end{eqnarray}
where $\|\x\|_{0}$ is what is typically called $\ell_0$ (quasi) norm of vector $\x$. We will simply view $\|\x\|_{0}$ as the number of the nonzero entries of $\x$.


Finding the sparsest $\x$ in (\ref{eq:l0}) (which we will technically call solving (\ref{eq:l0})) is of course typically considered as a not very easy task. There may be many reasons for that from the numerical linear algebra point of view. However, from the point of view that we will adopt the main reason is the numerical complexity of solving (\ref{eq:l0}). Clearly, if one is just interested in solving (\ref{eq:l0}) (i.e. if one doesn't really care how numerically complex such a solving is) then an exhaustive search of all subsets of $k$ columns of $A$ would solve the problem (simply extracting such subsets and solving the resulting over-determined systems would do it; we will always throughout the paper assume $m\geq k$). However, if one assumes the so-called linear regime (which eventually we will in this paper) where $k=\beta n$ and $m=\alpha n$ and $n$ is large and $\alpha$ and $\beta$ are constant independent of $n$ then there is an exponentially large number (in $n$ of course) of $k$ column subsets of $A$. Here, such a complexity will be considered as too high and will focus instead on algorithms/heuristics of polynomial complexity. Even with such a restriction there are quite a few fairly successful algorithms developed over last several decades (see, e.g. \cite{JATGomp,NeVe07,DTDSomp,NT08,DaiMil08,DonMalMon09}) that one could utilize. As the most important and as, mathematically speaking, the best currently known, we view the following $\ell_1$-optimization relaxation of (\ref{eq:l0})
\begin{eqnarray}
\mbox{min} & & \|\x\|_{1}\nonumber \\
\mbox{subject to} & & A\x=\y. \label{eq:l1}
\end{eqnarray}
Of course, initially the reason for its importance/popularity is its polynomial complexity and the fact that it effectively amounts to solving a linear program - task typically considered among the easiest in the theory of the (continuous) optimization algorithms. Furthermore, the implementation of (\ref{eq:l1}) itself is fairly universal as it requires no other knowledge beyond $A$ and $\y$ and it can be used/run (possibly with higher or lower success rate) with basically any full-rank matrix $A$.

Complementary to this are of course the excellent performance characteristics of (\ref{eq:l1}). Performance characterizations started in \cite{CRT,DOnoho06CS} and perfected in \cite{DonohoPol,DonohoUnsigned,StojnicCSetam09,StojnicUpper10} mathematically solidified the importance of (\ref{eq:l1}) in studying the linear under-determined systems with structured solutions. Motivated by the success we initially achieved in \cite{StojnicCSetam09,StojnicUpper10} in this paper we will substantially deepen our understanding of the performance characterization of (\ref{eq:l1}). Namely, \cite{DonohoPol,DonohoUnsigned,StojnicCSetam09,StojnicUpper10} uncovered that the (\ref{eq:l1}) exhibits the so-called phase-transition phenomenon when utilized in statistical contexts. Moreover, both sets of results, \cite{DonohoPol,DonohoUnsigned} and \cite{StojnicCSetam09,StojnicUpper10}, in addition to uncovering the existence of the phase transition phenomenon precisely characterized the so-called ``breaking points" where these phase transitions happen (essentially the highest possible $\beta$ for which the solution of (\ref{eq:l1}) with overwhelming probability matches the sparsest solution of (\ref{eq:l0}) for a fixed $\alpha$; under overwhelming probability we will in this papers consider probability over statistics of $A$ that is no more than a number exponentially decaying in $n$ away from $1$). Here, we will make a substantial progress in studying further the phase transitions. We will essentially connect them to the so-called \emph{large deviations property/principle} (LDP) from the classical probability theory and provide their explicit characterizations when viewed through such a prism. We will do so for two types of structured $\x$, namely for $k$ sparse $\x$ that we will sometimes refer to as the regular/general $k$ sparse $\x$ and for the so-called a priori known to be nonnegative k sparse $\x$ that we will often refer to as the positive/nonnegative $\x$. Furthermore, we will do so through two seemingly different approaches, the novel, more modern one which is purely probabilistic and is fully developed by us and the more classical one that is again fully developed by us but also uses as starting blocks some known facts from the high-dimensional geometry.

We will split the presentation into several parts. We will start things off by recalling on the basics of the phase transitions and on the known results that relate to them in the context of interest in this paper. We will then connect them to the LDP and then study the LDPs in a great detail (first for the general $\x$ and then for the positive $\x$) through a purely probabilistic approach mentioned above and initiated in a line of work that we started with \cite{StojnicCSetam09,StojnicUpper10}. In the later sections of the paper we will then switch to the high-dimensional geometry aspects of these problems and prove that through them one can obtained exactly the same characterizations. The main emphasis will be on the elegance of the final results that in our view matches the corresponding one that we have achieved initially in \cite{StojnicCSetam09,StojnicUpper10} and in a large set of results that we created later on.

\section{Phase transitions}
\label{sec:phasetrans}

We start by recalling on the phase transition (PT) phenomena that occur in statistical studies of many random structures. There is of course a long history of studying these phenomena in various aspects of optimization theory and algorithms. Instead of discussing all of them in detail we will focus on what is known about them in the context that is of interest here (as will soon be clear below, even explaining that in full mathematical detail will take some time and effort). As these phenomena are introduced in a bit more subtle way and in several different scenarios of interest in studying (\ref{eq:l1}) we will now make the above informal definition a bit more precise; however, we do emphasize as earlier, that although we will make a substantial effort to make all the definitions self-contained and fully precise we may on occasion deviate from this and rely on a familiarity with some of the known concepts (if such a scenario presents itself, we recommend that the reader consults our earlier works that contain all the necessary details).

As a measure of the above mentioned performance excellence of (\ref{eq:l1}) one typically takes (and we will do the same in this paper) the highest possible $\beta$ for which the solution of (\ref{eq:l1}) matches the sparsest solution of (\ref{eq:l0}) for a fixed $\alpha$. Along the same lines, for an algorithm that exhibits the so-called phase transition phenomenon, for any given constant $\alpha\leq 1$ there is a maximum
allowable value of $\beta$ such that for \emph{any} given $k$-sparse $\x$ in (\ref{eq:system}) the solution that the algorithm produces
is with overwhelming probability exactly that given $k$-sparse $\x$. This value of
$\beta$ is typically referred to as the \emph{strong threshold} (see
\cite{DonohoPol,StojnicCSetam09}) and we also say that the algorithm exhibits the \emph{strong} phase transition, i.e. the \emph{strong} PT. Informally speaking, the threshold values are essentially the breaking points where the algorithms (in our case here (\ref{eq:l1})) exhibit the phase transition phenomenon. In a more mathematical language, the phase transition phenomenon essentially means that if the problem dimensions are such that the pair $(\alpha,\beta)$ is below the so called phase transition curve (i.e. the PT curve) then the algorithm (here (\ref{eq:l1})) solves (in a probabilistic sense) the problem (here (\ref{eq:system}) or (\ref{eq:l0})); otherwise it fails. A full asymptotic performance characterization of an algorithm that exhibits the phase transition phenomenon assumes determining this phase transition curve.

When viewed from a practical point of view, the above requirement may sometimes be a bit restrictive. Instead, one may choose to characterize performance in a bit less restrictive way hoping to capture a bit more typical performance. A way to do so that we found as a fairly useful relaxes the \emph{any} requirement in the following way: for any given constant
$\alpha\leq 1$ and \emph{any} given $\x$ with a given fixed location and a given fixed set of signs
there will be a maximum allowable value of $\beta$ such that
(\ref{eq:l1}) finds that given $\x$ in (\ref{eq:system}) with overwhelming
probability. We will refer to this maximum allowable value of
$\beta$ as the \emph{weak threshold} and will denote it by $\beta_{w}$ (see, e.g. \cite{StojnicICASSP09,StojnicCSetam09}). Correspondingly, we also say that the algorithm exhibits the \emph{weak} phase transition (i.e. the \emph{weak} PT) and we call the resulting curve the \emph{weak} phase transition curve. The rationale behind the introduction of the weak PT is that if one needs (\ref{eq:system}) solved, it typically wants it solved for an $\x$ or for a set of $\x$ but quite likely not for every single $\x$. Whenever this happens to be the case the above weak phase transition curve is highly likely to be a more suitable and useful type of performance characterization. One can then proceed further along these lines and define various other types of phase transitions (or thresholds for $\beta$) depending on the scenarios where the algorithms are used and what kind of performance one is interested in. Since our main concern in this paper will be studying of the weak PT, we stop short of going into further details regarding other types of phase transitions and instead just mention in passing that another interesting and popular concept is the so-called \emph{sectional} phase transition (i.e. the \emph{sectional} PT) introduced in \cite{DonohoPol} and studied into the greatest of the details in \cite{DonohoPol,StojnicCSetam09,StojnicUpper10,StojnicUpperSec13,StojnicLiftStrSec13}. As a side note, we also add/emphasize that studying different types of phase transitions may often pose different challenges and more often than not studying some of them may be much harder/easier than studying others. This is precisely what happens with the $\ell_1$'s PTs, where for example studying the strong and sectional PTs turned out to be much, much harder than studying the weak ones and the results that we established in those directions in \cite{StojnicLiftStrSec13} have been standing for a while now as the benchmarks very hard to approach or beat.

There is of course a large body of work that deals with various aspects of the PTs that we introduced above. We will here just briefly single out the works that in our view stand as the most important and relevant to what we will showcase in the later sections of the paper. Initial performance characterizations done in \cite{CRT,DOnoho06CS} determined in a statistical scenario that for any $\alpha$ there is $\beta$ such that the solutions of (\ref{eq:l0}) and (\ref{eq:l1}) coincide (from this point on, we will under the solution of (\ref{eq:l0}) consider its sparsest solution). While we believe that \cite{CRT,DOnoho06CS} take a special place in the history of studying (\ref{eq:l1}) (in particular for their ability to generate a substantial portion of the interest in linear systems over the last decade) they fell a bit short in fully uncovering and characterizing the PT phenomenon. This was eventually done in \cite{DonohoPol,DonohoUnsigned,StojnicCSetam09,StojnicUpper10}. \cite{DonohoPol,DonohoUnsigned} connected the (\ref{eq:l1})'s PT properties to the studying of neighbourly polytopes in high-dimensional geometry and in return utilized a powerful machinery developed overthere to fully characterize the (\ref{eq:l1})'s PT. On the other hand in our own series of work \cite{StojnicCSetam09,StojnicUpper10}, we developed a novel purely probabilistic approach that also turned out to be very powerful generic probability tool. The results that we obtained in \cite{StojnicCSetam09,StojnicUpper10} of course fully characterized the ultimate (\ref{eq:l1})'s PTs as well. Of particular importance though, we view the simplicity of the concepts that we have developed in \cite{StojnicCSetam09,StojnicUpper10} and the ultimate elegance that we were able to achieve by using them in characterizing the PTs. We below recall on a theorem that essentially summarizes the results obtained in \cite{StojnicCSetam09,StojnicUpper10} and effectively establishes for any $0<\alpha\leq 1$ the exact value of $\beta_w$ for which (\ref{eq:l1}) finds the $k$-sparse $\x$ from (\ref{eq:system}).


\begin{theorem}(\cite{StojnicCSetam09,StojnicUpper10} Exact $\ell_1$'s weak threshold/PT)
Let $A$ be an $m\times n$ matrix in (\ref{eq:system})
with i.i.d. standard normal components. Let
the unknown $\x$ in (\ref{eq:system}) be $k$-sparse. Further, let the location and signs of nonzero elements of $\x$ be arbitrarily chosen but fixed.
Let $k,m,n$ be large
and let $\alpha_w=\frac{m}{n}$ and $\beta_w=\frac{k}{n}$ be constants
independent of $m$ and $n$. Let $\erfinv$ be the inverse of the standard error function associated with zero-mean unit variance Gaussian random variable.  Further, let $\alpha_w$ and $\beta_w$ satisfy the following \textbf{fundamental characterization of the $\ell_1$'s PT}

\begin{center}
\shadowbox{$
\xi_{\alpha_{w}}(\beta_w)\triangleq\psi_{\beta_w}(\alpha_{w})\triangleq
(1-\beta_w)\frac{\sqrt{\frac{2}{\pi}}e^{-\lp\erfinv\lp\frac{1-\alpha_w}{1-\beta_w}\rp\rp^2}}{\alpha_w\sqrt{2}\erfinv \lp\frac{1-\alpha_w}{1-\beta_w}\rp}=1.
$}
-\vspace{-.5in}\begin{equation}
\label{eq:thmweaktheta2}
\end{equation}
\end{center}

Then:
\begin{enumerate}
\item If $\alpha>\alpha_w$ then with overwhelming probability the solution of (\ref{eq:l1}) is the $k$-sparse $\x$ from (\ref{eq:system}).
\item If $\alpha<\alpha_w$ then with overwhelming probability there will be a $k$-sparse $\x$ (from a set of $\x$'s with fixed locations and signs of nonzero components) that satisfies (\ref{eq:system}) and is \textbf{not} the solution of (\ref{eq:l1}).
    \end{enumerate}
\label{thm:thmweakthr}
\end{theorem}
\begin{proof}
The first part was established in \cite{StojnicCSetam09} and the second one was established in \cite{StojnicUpper10}. An alternative way of establishing the same set of results was also presented in \cite{StojnicEquiv10}. Of course, similar results were obtained in \cite{DonohoPol,DonohoUnsigned}. Moreover, a different class of algorithms based on message passing introduced in \cite{DonMalMon09} was proven in \cite{BayMon10} to possess the above phase transition as well.
\end{proof}
\subsection{Properties of $\xi_\alpha(\beta)$ and $\psi_\beta(\alpha)$}
\label{sec:propxi}

In this subsection we will briefly look at a couple of key properties of functions $\xi_\alpha(\beta)$ and $\psi_\beta(\alpha)$ from Theorem \ref{thm:thmweakthr}. We do mention right here at the beginning that these are by now well known and fairly straightforward but for the mathematical exactness and completeness we find it convenient to have them neatly presented so that we may eventually recall on them in a more easier fashion.

\subsubsection{$\xi_\alpha(\beta)$}
\label{sec:propxi1}

The key observation regarding $\xi_\alpha(\beta)$ is that for any fixed $\alpha\in (0,1)$ there is a unique $\beta$ such that $\xi_\alpha(\beta)=1$. This essentially ensures that (\ref{eq:thmweaktheta2}) is an unambiguous PT characterization. To confirm that this is indeed true we make the following observations:

\underline{\emph{1) For any fixed $\alpha\in (0,1)$, $\xi_\alpha(\beta)-1$ is a decreasing function of $\beta$ on interval $[0,\alpha)$.
}}

To see this we proceed in the following straightforward way
\begin{eqnarray}\label{eq:propxi1}
  \frac{d(\xi_\alpha(\beta)-1)}{d\beta} & = & \frac{d\lp(1-\beta)\frac{\sqrt{\frac{2}{\pi}}e^{-\lp\erfinv\lp\frac{1-\alpha}{1-\beta}\rp\rp^2}}{\alpha\sqrt{2}\erfinv \lp\frac{1-\alpha}{1-\beta}\rp}-1\rp}{d\beta}\nonumber\\
  & = & \sqrt{\frac{2}{\pi}}\frac{-\frac{\sqrt{\pi} (1-\alpha)}{(1-\beta) \erfinv((1-\alpha)/(1-\beta))^2}-\frac{2 e^{-\lp \erfinv\lp\frac{1-\alpha}{1-\beta}\rp\rp^2}}{\erfinv((1-\alpha)/(1-\beta))}-\frac{2\sqrt{\pi} (1-\alpha)}{1-\beta}}{2 \sqrt{2} \alpha}\nonumber \\
  &< & 0.
\end{eqnarray}

\underline{\emph{2) For any fixed $\alpha\in (0,1)$, $\lim_{\beta\rightarrow \alpha}\xi_\alpha(\beta)-1=-1$.}}

This easily follows after one observes that
\begin{equation}\label{eq:propxi2}
  \lim_{\beta\rightarrow \alpha}\lp\erfinv\lp\frac{1-\alpha}{1-\beta}\rp\rp=\infty
\end{equation}

\underline{\emph{3) For any fixed $\alpha\in (0,1)$, $\xi_\alpha(0)-1>0$.}}

To show $\xi_\alpha(0)-1>0$ it is of course enough to show $\xi_\alpha(0)>1$. To that end we have
\begin{equation}\label{eq:propxi31}
\xi_\alpha(0)=\frac{\sqrt{\frac{2}{\pi}}e^{-(\erfinv(1-\alpha))^2}}{\alpha\sqrt{2}\erfinv (1-\alpha)}.
\end{equation}
Let
\begin{equation}\label{eq:propxi32}
z_\alpha=\erfinv (1-\alpha) \quad \mbox{and} \quad \alpha=1-z_\alpha.
\end{equation}
Then combining (\ref{eq:propxi31}) and (\ref{eq:propxi32}) we obtain
\begin{equation}\label{eq:propxi33}
\xi_\alpha(0)  =  \frac{\sqrt{\frac{2}{\pi}}e^{-(\erfinv(1-\alpha))^2}}{\alpha\sqrt{2}\erfinv (1-\alpha)} =  \frac{\sqrt{\frac{2}{\pi}}e^{-z_\alpha^2}}{\sqrt{2}z_\alpha(1-\erf(z_\alpha))}=
\frac{\sqrt{\frac{2}{\pi}}e^{-z_\alpha^2}}{\sqrt{2}z_\alpha\erfc(z_\alpha)}=
\frac{\sqrt{\frac{2}{\pi}}e^{-z_\alpha^2}}{\sqrt{2}z_\alpha(2Q(\sqrt{2}z_\alpha))},
\end{equation}
where $Q(\cdot)$ is the $Q$-function associated with the standard normal variables. Now we recall on the following well known inequalities that $Q(\cdot)$ satisfies:
\begin{equation}\label{eq:propxi34}
  \frac{x}{1+x^2}\frac{e^{-\frac{x^2}{2}}}{\sqrt{2\pi}}< Q(x)<   \frac{1}{x}\frac{e^{-\frac{x^2}{2}}}{\sqrt{2\pi}}.
\end{equation}
An easy combination of (\ref{eq:propxi33}) and (\ref{eq:propxi34}) then gives
\begin{equation}\label{eq:propxi35}
\xi_\alpha(0) =
\frac{\sqrt{\frac{2}{\pi}}e^{-z_\alpha^2}}{\sqrt{2}z_\alpha(2Q(\sqrt{2}z_\alpha))}>
\frac{\sqrt{\frac{2}{\pi}}e^{-z_\alpha^2}}{2\sqrt{2}z_\alpha}\sqrt{2}z_\alpha\sqrt{2\pi}e^{\frac{(\sqrt{2}z_\alpha)^2}{2}}=1.
\end{equation}
A combination of the above three observations ensures that for any fixed $\alpha\in (0,1)$ there is a unique $\beta$ such that $\xi_\alpha(\beta)=1$, which as mentioned above essentially means that (\ref{eq:thmweaktheta2}) is an unambiguous PT characterization. For the completeness, in Figure \ref{fig:propxi} we present a few numerical results related to the behavior of $\xi_\alpha(\beta)$ that indeed confirm the above calculations.
\begin{figure}[htb]
\begin{minipage}[b]{.5\linewidth}
\centering
\centerline{\epsfig{figure=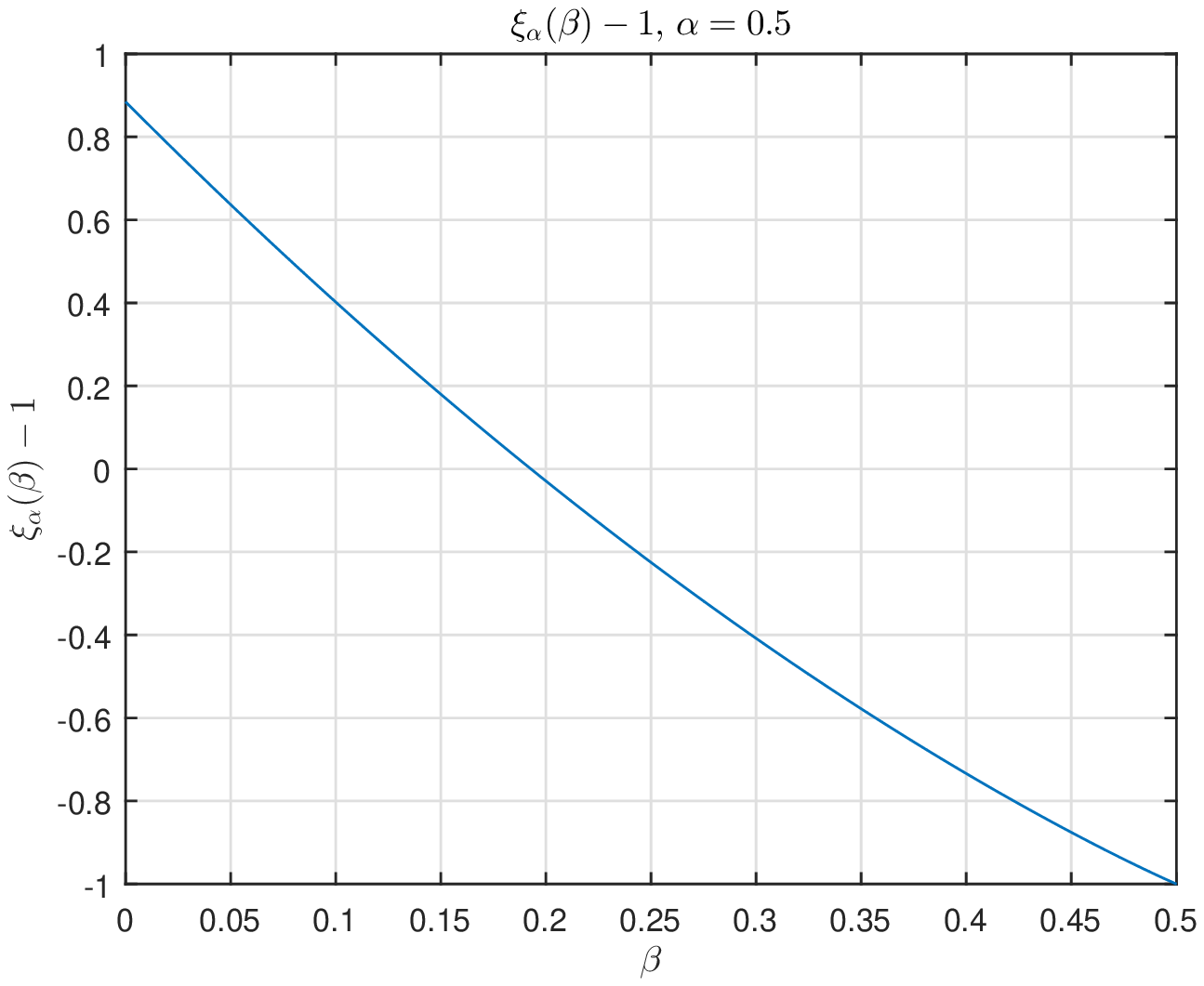,width=9cm,height=7cm}}
\end{minipage}
\begin{minipage}[b]{.5\linewidth}
\centering
\centerline{\epsfig{figure=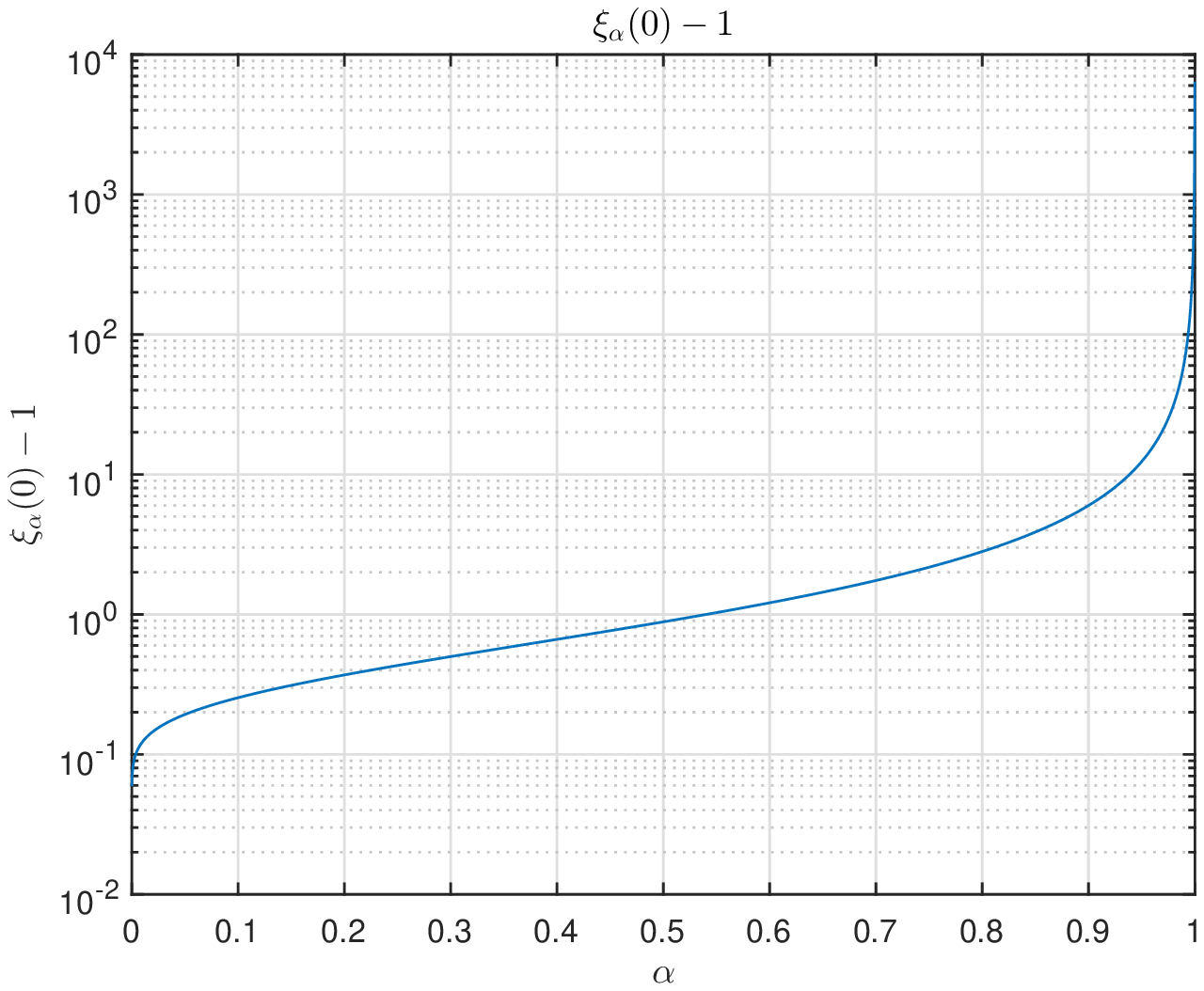,width=9cm,height=7cm}}
\end{minipage}
\caption{Properties of $\xi_\alpha(\beta)$: $\frac{d(\xi_\alpha(\beta)-1)}{d\beta}$ as a function of $\beta$ ($\alpha=0.5$) -- left; $\xi_\alpha(0)$ as a function of $\alpha$ -- right}
\label{fig:propxi}
\end{figure}

\subsubsection{$\psi_\beta(\alpha)$}
\label{sec:proppsixi1}

We now look at $\psi_\beta(\alpha)$. One then notes that one of the key observations that was true for $\xi_\alpha(\beta)$ remains true when it comes to $\psi_\beta(\alpha)$ as well. Namely, for any fixed $\beta\in (0,1)$ there is a unique $\alpha$ such that $\psi_\beta(\alpha)=1$. This essentially ensures that the $\ell_1$'s fundamental PT from the above theorem is also unambiguous when viewed as a function of $\alpha$. To confirm that this is indeed true we proceed in a fashion similar to the one from Section \ref{sec:propxi1} and make the following observations:

\underline{\emph{1) For any fixed $\beta\in (0,1)$, $\psi_\beta(\alpha)-1$ is an increasing function of $\alpha$ on interval $(\beta,1]$.
}}

To see this we proceed by computing the derivative
\begin{eqnarray}\label{eq:proppsixi1}
  \frac{d(\psi_\beta(\alpha)-1)}{d\alpha} & = & \frac{d\lp(1-\beta)\frac{\sqrt{\frac{2}{\pi}}e^{-\lp\erfinv\lp\frac{1-\alpha}{1-\beta}\rp\rp^2}}{\alpha\sqrt{2}\erfinv (\frac{1-\alpha}{1-\beta})}-1\rp}{d\alpha}\nonumber\\
   & = & \frac{2 (\beta-1) e^{-\left (\erfinv\left (\frac{1-\alpha}{1-\beta}\right )\right )^2} \left (\erfinv\left (\frac{1-\alpha}{1-\beta}\right )\right )
  +\sqrt{\pi} \alpha \left (2 \left (\erfinv\left (\frac{1-\alpha}{1-\beta}\right )\right )^2+1\right )}{2 \alpha^2 \sqrt{\pi} \left (\erfinv\left (\frac{1-\alpha}{1-\beta}\right )\right )^2}.
\end{eqnarray}
Let
\begin{equation}\label{eq:proppsixi2}
  q=\erfinv\left (\frac{1-\alpha}{1-\beta}\right ).
\end{equation}
Then
\begin{equation}\label{eq:proppsixi3}
  \frac{2 (\beta-1) e^{-\left (\erfinv\left (\frac{1-\alpha}{1-\beta}\right )\right )^2} \left (\erfinv\left (\frac{1-\alpha}{1-\beta}\right )\right )
  +\sqrt{\pi} \alpha \left (2 \left (\erfinv\left (\frac{1-\alpha}{1-\beta}\right )\right )^2+1\right )}{2 \alpha^2 \sqrt{\pi} \left (\erfinv\left (\frac{1-\alpha}{1-\beta}\right )\right )^2}   =  \frac{2 \frac{\alpha-1}{\erf(q)} e^{-q^2} q
  +\sqrt{\pi} \alpha \left (2q^2+1\right )}{2 \alpha^2 \sqrt{\pi} q^2}.
\end{equation}
Since $\alpha\geq 1-\erf(q)=\erfc(q)$ we also have
\begin{eqnarray}\label{eq:proppsixi4}
2 \frac{\alpha-1}{\erf(q)} e^{-q^2} q
  +\sqrt{\pi} \alpha \left (2q^2+1\right )
  & =&  -2 e^{-q^2} q
  +\sqrt{\pi} \alpha \erfc(q)\left (2q^2+1\right )\nonumber \\
  & > & -2 e^{-q^2}q
  +\frac{2 e^{-q^2}\left (2q^2+1\right )}{\left (q+\sqrt{q^2+2}\right )}\nonumber \\
  & = & \frac{2 e^{-q^2}}{\left (q+\sqrt{q^2+2}\right )}(1+q^2-q\sqrt{q^2+2})\nonumber \\
  & >& 0,
\end{eqnarray}
where the first inequality follows as an application of the following well known inequalities for $\erfc(\cdot)$.
\begin{equation}
\frac{2}{\sqrt{\pi}}\frac{e^{-y^2}}{y+\sqrt{y^2+2}}< \erfc(y)\leq \frac{2}{\sqrt{\pi}}\frac{e^{-y^2}}{y+\sqrt{y^2+\frac{4}{\pi}}}.
 \label{eq:hdg10a}
\end{equation}
Connecting (\ref{eq:proppsixi1}), (\ref{eq:proppsixi2}), (\ref{eq:proppsixi3}), and (\ref{eq:proppsixi4}) we finally have
\begin{equation}\label{proppsoxi5}
    \frac{d(\psi_\beta(\alpha)-1)}{d\alpha}>0,
\end{equation}
and the function $(\psi_\beta(\alpha)-1)$ is indeed increasing on $(\beta,1]$.

\underline{\emph{2) For any fixed $\beta\in (0,1)$, $\lim_{\alpha\rightarrow \beta}\psi_\beta(\alpha)-1=-1$.}}

This easily follows after one observes that
\begin{equation}\label{eq:proppsixi6}
  \lim_{\alpha\rightarrow \beta}\lp\erfinv\lp\frac{1-\alpha}{1-\beta}\rp\rp=\infty
\end{equation}

\underline{\emph{3) For any fixed $\beta\in (0,1)$, $\lim_{\alpha\rightarrow 1}\psi_\beta(\alpha)-1=\infty >0$.}}

This also easily follows after one observes that
\begin{equation}\label{eq:proppsixi7}
  \lim_{\alpha\rightarrow 1}\lp\erfinv\lp\frac{1-\alpha}{1-\beta}\rp\rp=0
\end{equation}
Combining the above three observations one can ensure that for any fixed $\beta\in (0,1)$ there is a unique $\alpha$ such that $\psi_\beta(\alpha)=1$, which reconfirms that the $\ell_1$'s fundamental PT characterization is unambiguous. For the completeness, in Figure \ref{fig:proppsi} we present a few numerical results related to the behavior of $\psi_\beta(\alpha)$ that are indeed in agreement with the above calculations.
\begin{figure}[htb]
\begin{minipage}[b]{.5\linewidth}
\centering
\centerline{\epsfig{figure=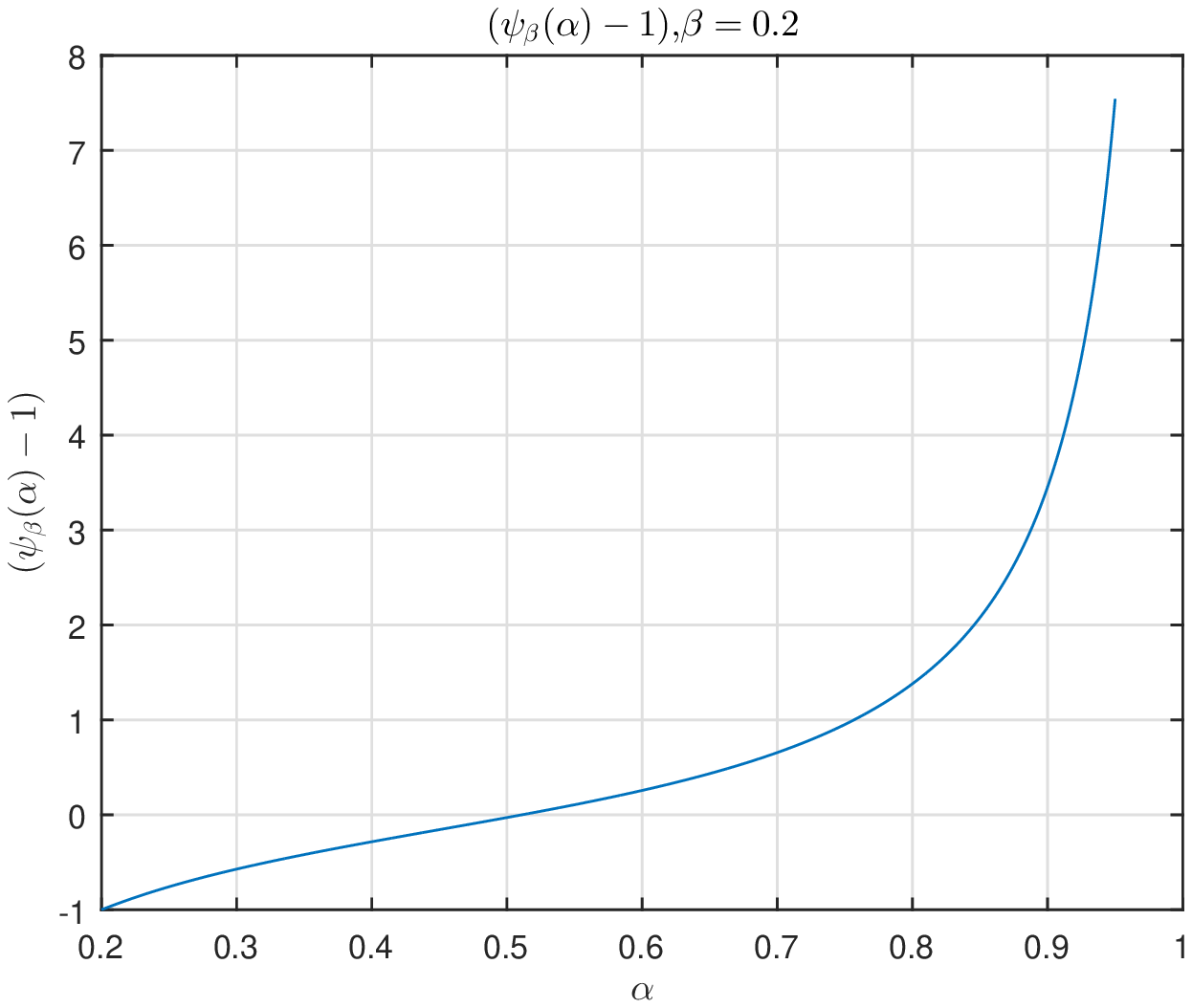,width=9cm,height=7cm}}
\end{minipage}
\begin{minipage}[b]{.5\linewidth}
\centering
\centerline{\epsfig{figure=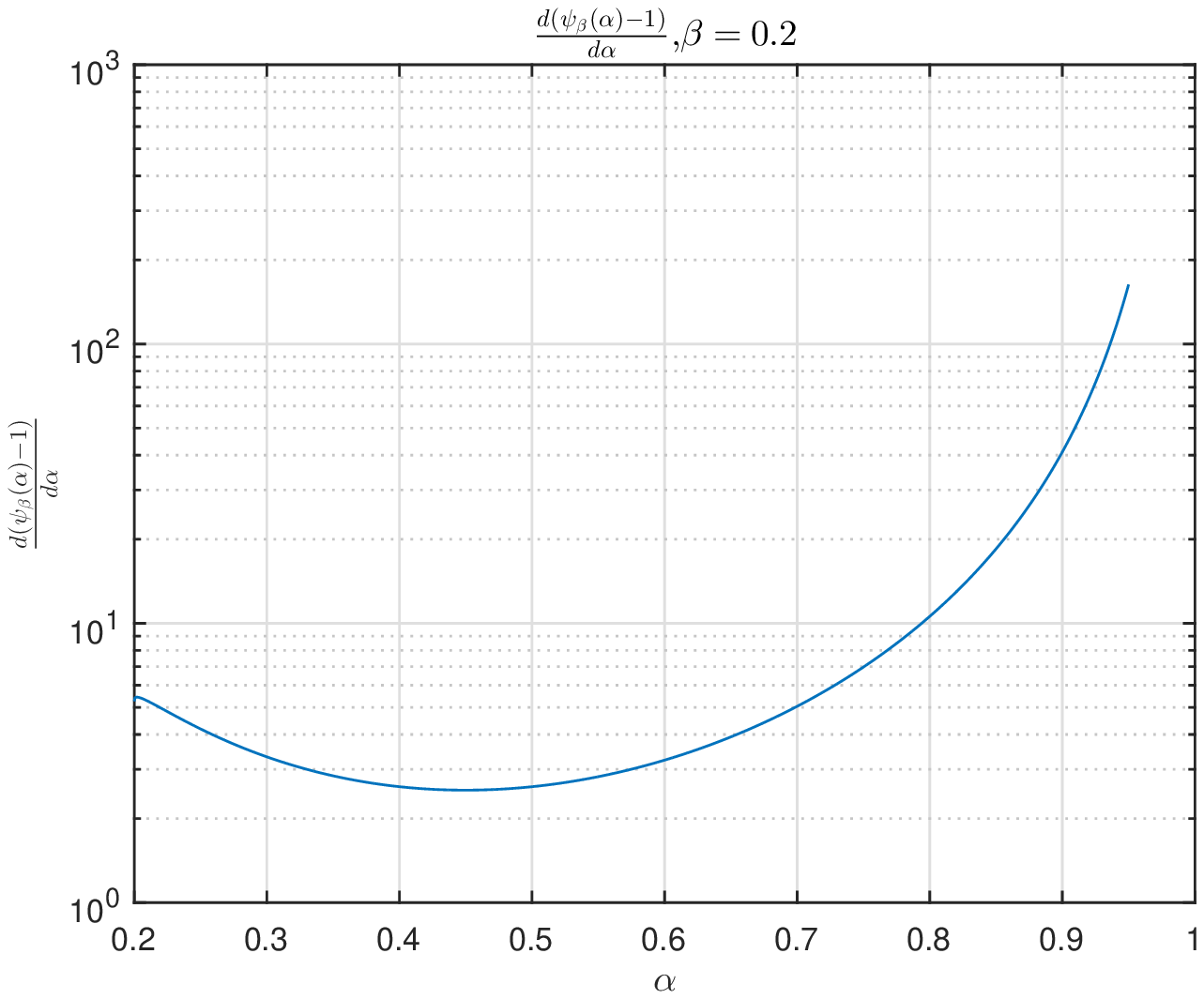,width=9cm,height=7cm}}
\end{minipage}
\caption{Properties of $\psi_\beta(\alpha)$: $\psi_\beta(\alpha)-1$ as a function of $\alpha$ ($\beta=0.2$) -- left; $\frac{d(\psi_\beta(\alpha)-1)}{d\alpha}$ as a function of $\alpha$ ($\beta=0.2$) -- right}
\label{fig:proppsi}
\end{figure}

Finally, to give a little bit of a flavor as to what is actually proven in Theorem \ref{thm:thmweakthr} we in Figure \ref{fig:weakl1PT} show the theoretical PT curve that one can obtain based on (\ref{eq:thmweaktheta2}) as well as how it fits the corresponding one obtained through a high-dimensional geometry approach in \cite{DonohoPol,DonohoUnsigned}.
\begin{figure}[htb]
\centering
\centerline{\epsfig{figure=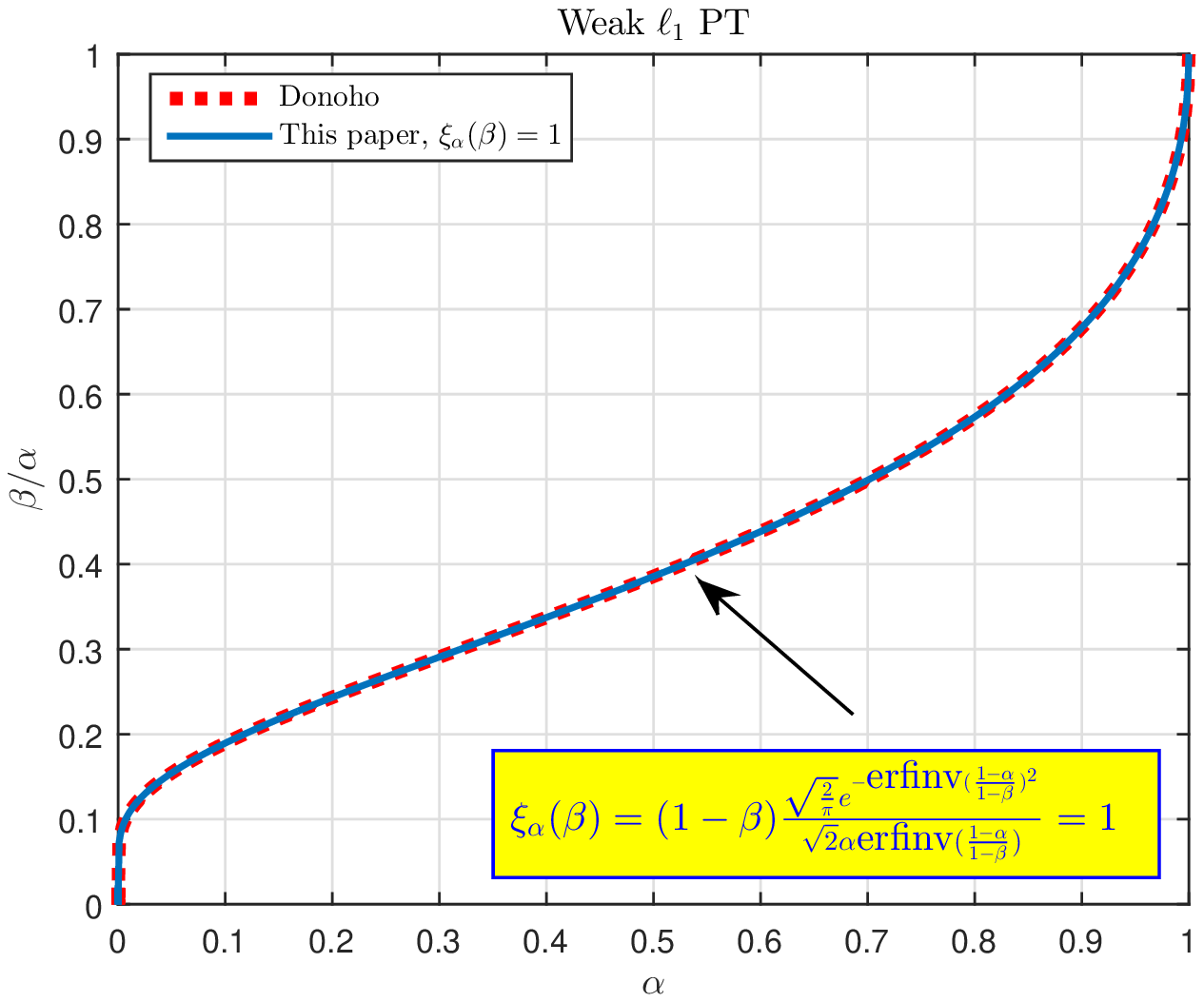,width=11.5cm,height=8cm}}
\caption{$\ell_1$'s weak PT; $\{(\alpha,\beta)|\xi_{\alpha}(\beta)=1\}$}
\label{fig:weakl1PT}
\end{figure}


\section{Large deviations}
\label{sec:ldp}

The results that we presented in the previous section deal with the so-called phase-transition phenomenon and along those lines they determine the so-called breaking points of success of $\ell_1$. These are determined in an asymptotic sense, assuming large system dimensions. Furthermore, they are typically framed through the so-called ``overwhelming" probabilities (which tend to zero as systems dimensions grow large). In this section we will raise the bar a bit higher and try to determine these overwhelming probabilities in a bit more explicit way, essentially the one that goes a bit beyond this standard formulation that states that they tend to zero as the system dimensions grow large. Effectively, we will determine the so-called \emph{rate} at which they tend to zero. As expected, these rates will change as the ratios of systems dimensions change. To fully characterize all the rates, we will essentially determine them for any point in $(\alpha,\beta)$ plane (also, and as is probably expected, to ensure that the results make sense we will assume $\{(\alpha,\beta)| \beta<\alpha,\alpha\in(0,1)\}$). To achieve these characterizations we will first introduce a somewhat novel concept that effectively resembles what is in probability theory known as the large deviation principle/property. Before doing any of that we will start things off by recalling on a couple of results that we established in \cite{StojnicCSetam09,StojnicICASSP09}.

For the simplicity and without loss of generality we will assume that the elements $\x_{1},\x_{2},\dots,\x_{n-k}$ of $\x$ are equal to zero and that the elements $\x_{n-k+1},\x_{n-k+2},\dots,\x_n$ have fixed signs, say they all are positive (one can observe that this is of course in an agreement with the requirement that the weak phase transition imposes). The following was proved in \cite{StojnicCSetam09,StojnicICASSP09} and is one of the key features that enabled us to run the entire machinery developed overthere.
\begin{theorem}(\cite{StojnicCSetam09,StojnicICASSP09} Nonzero elements of $\x$ have fixed signs and location)
Assume that an $m\times n$ measurement matrix $A$ is given. Let $\x$
be a $k$ sparse vector. Also let $\x_1=\x_2=\dots=\x_{n-k}=0$. Let the signs of $\x_{n-k+1},\x_{n-k+2},\dots,\x_n$ be fixed, say all positive. Further, assume that $\y=A\x$ and that $\w$ is
a $n\times 1$ vector. Then (\ref{eq:l1}) will
produce the solution of (\ref{eq:system}) if
\begin{equation}
(\forall \w\in \mR^{n} | A\w=0) \quad  -\sum_{i=n-k+1}^n \w_i<\sum_{i=1}^{n-k}|\w_{i}|.
\label{eq:thmeqgenweak}
\end{equation}\label{thm:thmgenweak}
\end{theorem}
To facilitate the exposition we set
\begin{equation}
\Sw\triangleq\{\w\in S^{n-1}| \quad -\sum_{i=n-k+1}^n \w_i<\sum_{i=1}^{n-k}|\w_{i}|\}.\label{eq:defSwpr}
\end{equation}

\subsection{Upper tail}
\label{sec:uppertail}

As mentioned above our goal will be to determine at which rate the probabilities that $\ell_1$ succeeds/fails go to zero as system dimensions grow. As our goal will also be to do so for any point in $\{(\alpha,\beta)| \beta<\alpha,\alpha\in(0,1)\}$ we will split the presentation depending on where exactly in this set our point of interest is. Namely, we will first consider points $(\alpha,\beta)$ such that $\alpha\geq \alpha_w$ where $\alpha_w$ is such that $\psi_\beta(\alpha_w)=\xi_{\alpha_w}(\beta)=1$. These points will establish what we will refer to as the upper tail. The remaining ones will be discussed in the following section and they will establish what we will refer to as the lower tail.

As in Theorem \ref{thm:thmweakthr}, we will assume that the elements of $A$ are i.i.d. standard normals and will be interested in the following probability
\begin{equation}
P_{err}\triangleq P(\min_{\w\in S_w}\|A\w\|_2\leq 0)=P(\max_{\w\in S_w}\min_{\|\y\|_2=1}(\y^T A\w )\geq 0).
\label{eq:ldpprob}
\end{equation}
It is relatively easy to see that $P_{err}$ is the so-called probability of error/failure, i.e. the probability that (\ref{eq:l1}) fails to produce the solution of (\ref{eq:system}). A simple application of the Chernoff bound then gives
\begin{equation}
P_{err}=P(\max_{\w\in S_w}\min_{\|\y\|_2=1}(\y^T A\w )\geq 0)\leq E e^{c_3\max_{\w\in S_w}\min_{\|\y\|_2=1}(\y^T A\w )}=E \max_{\w\in S_w}\min_{\|\y\|_2=1}e^{(-c_3\y^T A\w )},
\label{eq:ldpprob1}
\end{equation}
where we assume $c_3\geq 0$ and note that $A$ and $-A$ have the same distribution. Following the machinery of \cite{StojnicBlockasymldpfinn15,StojnicLiftStrSec13} we then have
\begin{equation}
P_{err}\leq e^{-\frac{c_3^2}{2}}E\max_{\w\in S_w}\min_{\|\y\|_2=1}e^{-c_3(\y^T A\w+g)}\leq
e^{-\frac{c_3^2}{2}}Ee^{-c_3\|\g\|_2}Ee^{c_3w(\h,S_w)},
\label{eq:ldpprob3}
\end{equation}
where
\begin{equation}
w(\h,\Sw)\triangleq\max_{\w\in \Sw} (\h^T\w), \label{eq:widthdefSw}
\end{equation}
and the elements of $\h$ are i.i.d. standard normals. To characterize the right hand side of (\ref{eq:ldpprob3}) we focus on $w(\h,S_w)$. In \cite{StojnicCSetam09,StojnicCSetamBlock09,StojnicBlockasymldpfinn15,StojnicLiftStrSec13} we developed a super powerful mechanism that enables a very elegant and useful representation for $w(\h,S_w)$. We will of course skip the details of these presentations and instead just present the final, neat results that we will utilize here.

We start by setting
\begin{equation}
\hw\triangleq(|\h_1|,|\h_2|,\dots,|\h_{n-k}|,
-\h_{n-k+1},-\h_{n-k+2},\dots,-\h_{n})^T.\label{eq:defhweak}
\end{equation}
Then one can characterize $w(\h,\Sw)$ in (\ref{eq:widthdefSw}) in the following way
\begin{eqnarray}
w(\h,\Sw) = \max_{\bar{\y}\in \mR^{n}} & &  \sum_{i=1}^{n} \hw_i \bar{\y}_i\nonumber \\
\mbox{subject to} &  & \bar{\y}_i\geq 0, 0\leq i\leq n-k\nonumber \\
& & \sum_{i=n-k+1}^n\bar{\y}_i\geq \sum_{i=1}^{n-k} \bar{\y}_i \nonumber \\
& & \sum_{i=1}^{n}\bar{\y}_i^2\leq 1\label{eq:workww2}
\end{eqnarray}
where $\hw_i$ is the $i$-th element of $\hw$ and $\bar{\y}_i$ is the $i$-th element of $\bar{\y}$. Solving (\ref{eq:workww2}) as was done in \cite{StojnicCSetam09,StojnicCSetamBlock09,StojnicBlockasymldpfinn15,StojnicLiftStrSec13} one obtains
\begin{eqnarray}
w(\h,\Sw) = -\max_{\nu\geq 0,\gamma\geq 0}\min_{\bar{\y}} & & \sum_{i=1}^{n} -\hw_i \bar{\y}_i+\nu\sum_{i=1}^{n-k}\bar{\y}_i
-\nu\sum_{i=n-k+1}^{n}\bar{\y}_i+\gamma\sum_{i=1}^{n}\bar{\y}_i^2-\gamma\nonumber \\
\mbox{subject to} & & \bar{\y}_i\geq 0, 0\leq i\leq n-k.\label{eq:ldpwhSw0}
\end{eqnarray}
Finally, after solving the inner minimization we have
\begin{eqnarray}
w(\h,\Sw) & = & \min_{\nu\geq0,\gamma\geq 0} \frac{\sum_{i=1}^{n-k}\max(\hw_i-\nu,0)^2+\sum_{i=n-k+1}^{n}(\hw_i+\nu)^2}{4\gamma}+\gamma\nonumber \\
& = & \min_{\nu\geq0}\sqrt{\sum_{i=1}^{n-k}\max(\hw_i-\nu,0)^2+\sum_{i=n-k+1}^{n}(\hw_i+\nu)^2}.\label{eq:ldpwhSw}
\end{eqnarray}
The above results provide a way to upper bound $P_{err}$. We summarize them in the following theorem.
\begin{theorem}
Let $A$ be an $m\times n$ matrix in (\ref{eq:system})
with i.i.d. standard normal components. Let
the unknown $\x$ in (\ref{eq:system}) be $k$-sparse and let the location and the signs of nonzero elements of $\x$ be arbitrarily chosen but fixed. Let $P_{err}$ be the probability that the solution of (\ref{eq:l1}) is not the $k$-sparse solution of (\ref{eq:system}). Then
\begin{equation}
P_{err}\leq \min_{c_3\geq 0}e^{-\frac{c_3^2}{2}}e^{-c_3\|\g\|_2}Ee^{c_3w(\h,S_w)}
=\min_{c_3\geq 0}\left (e^{-\frac{c_3^2}{2}}\frac{1}{\sqrt{2\pi}^m}\int_{\g}e^{-\sum_{i=1}^{m}\g_i^2/2-c_3\|\g\|_2}d\g \min_{\nu\geq 0,\gamma\geq\frac{c_3}{2}} w_1^{n-k}w_2^{k}e^{c_3\gamma}\right ),
\label{eq:ldpthm1perrub1}
\end{equation}
where
\begin{eqnarray}
w_1 &=& \frac{1}{\sqrt{2\pi}}\int_{\bar{h}}e^{-\bar{h}^2/2}e^{c_3\max(|\bar{h}|-\nu,0)^2/4/\gamma}d\bar{h}
  =\frac{e^{\frac{c_3\nu^2/4/\gamma}{1-c_3/2/\gamma}}}{\sqrt{1-c_3/2/\gamma}}\erfc\left (\frac{\nu}{\sqrt{2}\sqrt{1-c_3/2/\gamma}}\right )+\erf\left (\frac{\nu}{\sqrt{2}}\right )\nonumber \\
w_2 &=& \frac{1}{\sqrt{2\pi}}\int_{\bar{h}}e^{-\bar{h}^2/2}e^{c_3(\bar{h}+\nu)^2/4/\gamma}d\bar{h}
  =\frac{e^{\frac{c_3\nu^2/4/\gamma}{1-c_3/2/\gamma}}}{\sqrt{1-c_3/2/\gamma}}.\label{eq:ldpthm1perrub2}
\end{eqnarray}\label{thm:ldp1}
\end{theorem}
\begin{proof}
Follows from the above considerations and ultimately through the mechanisms developed in \cite{StojnicCSetam09,StojnicCSetamBlock09,StojnicBlockasymldpfinn15,StojnicLiftStrSec13}.
\end{proof}
Although the bound given in the above theorem is not exactly the type of quantity that we are interested in below we would like to point out that it is valid for any integers $m$, $k$, and $n$ (provided $k\leq m\leq n$ so that the results make sense). Our main concern below though is the asymptotic regime, basically the same one as in Theorem \ref{thm:thmweakthr}. In such a scenario the rate at which $P_{err}$ decays is of particular importance. Such a rate one can define as
\begin{equation}\label{eq:ldpasymp1}
  I_{err}\triangleq\lim_{n\rightarrow\infty}\frac{\log{P_{err}}}{n}.
\end{equation}
This now clearly resembles the so-called large deviation property/principle (LDP) with $I_{err}$ emulating the so-called LDP's rate (indicator) function. Based on Theorem \ref{thm:ldp1} we have the following LDP type of theorem.
\begin{theorem}
Assume the setup of Theorem \ref{thm:ldp1}. Further, let integers $m$, $k$, and $n$ be large ($k\leq m\leq n$) such that $\beta=\frac{k}{n}$ and $\alpha=\frac{m}{n}$ are constants independent of $n$. Assume that a pair $(\alpha,\beta)$  is given. Also, assume the following scaling: $c_3\rightarrow c_3\sqrt{n}$ and $\gamma\rightarrow\gamma\sqrt{n}$. Then
\begin{equation}
I_{err}(\alpha,\beta)\triangleq\lim_{n\rightarrow\infty}\frac{\log{P_{err}}}{n}
\leq \min_{c_3\geq 0}\left (-\frac{(c_3)^2}{2}+I_{sph}+\min_{\nu\geq 0,\gamma\geq 0} ((1-\beta)\log{w_1}+\beta\log{w_2}+c_3\gamma)\right )\triangleq I_{err,u}^{(ub)}(\alpha,\beta),
\label{eq:ldpthm2Ierrub1}
\end{equation}
where
\begin{eqnarray}
I_{sph} &=& \widehat{\gamma}c_3-\frac{\alpha }{2}\log\left (1-\frac{c_3}{2\widehat{\gamma}}\right )\nonumber \\
  \widehat{\gamma} &=& \frac{c_3-\sqrt{(c_3)^2+4\alpha}}{4}\nonumber \\
w_1 &=& \frac{1}{\sqrt{2\pi}}\int_{\bar{h}}e^{-\bar{h}^2/2}e^{c_3\max(|\bar{h}|-\nu,0)^2/4/\gamma}d\bar{h}
  =\frac{e^{\frac{c_3\nu^2/4/\gamma}{1-c_3/2/\gamma}}}{\sqrt{1-c_3/2/\gamma}}\erfc\left (\frac{\nu}{\sqrt{2}\sqrt{1-c_3/2/\gamma}}\right )+\erf\left (\frac{\nu}{\sqrt{2}}\right )\nonumber \\
  w_2 &=& \frac{1}{\sqrt{2\pi}}\int_{\bar{h}}e^{-\bar{h}^2/2}e^{c_3(\bar{h}+\nu)^2/4/\gamma}d\bar{h}
  =\frac{e^{\frac{c_3\nu^2/4/\gamma}{1-c_3/2/\gamma}}}{\sqrt{1-c_3/2/\gamma}}.\label{eq:ldpthm2perrub2}
\end{eqnarray}\label{thm:ldp2}
\end{theorem}
\begin{proof} Follows from Theorem \ref{thm:ldp1} and by noting that in \cite{StojnicMoreSophHopBnds10} we established
\begin{equation}
I_{sph}=\lim_{n\rightarrow\infty}\frac{1}{n}\log(Ee^{-c_3\sqrt{n}\|\g\|_2})
=\widehat{\gamma}c_3-\frac{\alpha }{2}\log\left (1-\frac{c_3}{2\widehat{\gamma}}\right ),\label{eq:gamaiden2lift}
\end{equation}
where
\begin{equation}
\widehat{\gamma}=\frac{c_3-\sqrt{(c_3)^2+4\alpha }}{4}.\label{eq:gamaiden3lift}
\end{equation}
\end{proof}

One can now numerically solve the above optimization problem and obtain the estimates for the rate of $P_{err}$'s decay. However, our goal here will be much more than that. We will raise the bar to an ultimate level and in the following subsection we will present a closed form solution to the above optimization problem. We do emphasize that the solution that we will present will look fairly elegant and consequently one may be tempted to believe that it was fairly straightforward to achieve it. Such a statement could not be further from the truth. Not only did it turn out to be quite a challenge to achieve an elegant presentation of the final solution but it was also fairly hard to provide any type of closed form solution. This is along the lines of what happened when we created the fundamental $\ell_1$ PT from Theorem \ref{thm:thmweakthr} in \cite{StojnicCSetam09}. The final result looked incredibly elegant, however before we discovered it, achieving any form of PT characterization even remotely close to the true one had been considered quite a success for decades.

\subsection{A detailed analysis of $I_{err,u}^{(ub)}$}
\label{sec:analysisIerr}

We start by introducing the following
\begin{equation}\label{eq:detanalIerr1}
  A_{0}\triangleq\sqrt{1-\frac{c_3}{2\gamma}}.
\end{equation}
One is then left with the following problem
\begin{equation}\label{eq:detanalIeer2}
I_{err,u}^{(ub)}(\alpha,\beta)\triangleq \min_{c_3\geq 0,\nu\geq 0,A_0\leq 1}\zeta_{\alpha,\beta}(c_3,\nu,A_0)
\end{equation}
where
\begin{eqnarray}
\zeta_{\alpha,\beta}(c_3,\nu,A_0)&=&\left (-\frac{c_3^2}{2}+I_{sph}+(1-\beta)\log{w_1}+\beta\log{w_2}+\frac{c_3^2}{2(1-A_0^2)}\right )\nonumber \\
I_{sph} &=& \widehat{\gamma}c_3-\frac{\alpha }{2}\log\left (1-\frac{c_3}{2\widehat{\gamma}}\right )\nonumber \\
  \widehat{\gamma} &=& \frac{c_3-\sqrt{(c_3)^2+4\alpha}}{4}\nonumber \\
w_1 &=&
  \frac{e^{\frac{(1-A_0^2)\nu^2}{2A_0^2}}}{A_0}\erfc\left (\frac{\nu}{\sqrt{2}A_0}\right )+\erf\left (\frac{\nu}{\sqrt{2}}\right )\nonumber \\
  w_2 &=&
  \frac{e^{\frac{(1-A_0^2)\nu^2}{2A_0^2}}}{A_0}.\label{eq:detanalIeer3}
\end{eqnarray}
We will now proceed by computing the derivatives of $\zeta_{\alpha,\beta}(c_3,\nu,A_0)$ with respect to $c_3$, $\nu$, and $A_0$.

\subsubsection{Handling the derivatives of $\zeta_{\alpha,\beta}(c_3,\nu,A_0)$}
\label{sec:derivativeszeta}

We start with the derivative with respect to $\nu$.
\begin{eqnarray}
\frac{d\zeta_{\alpha,\beta}(c_3,\nu,A_0)}{d\nu}&=&\frac{d}{d\nu}\left (-\frac{c_3^2}{2}+I_{sph}+(1-\beta)\log{w_1}+\beta\log{w_2}+\frac{c_3^2}{2(1-A_0^2)}\right )\nonumber \\
&=& \frac{\beta(1-A_0^2)\nu}{A_0^2}+\frac{1-\beta}{w_1}\left (\frac{(1-A_0^2)\nu}{A_0^2}\frac{e^{\frac{(1-A_0^2)\nu^2}{2A_0^2}}}{A_0}\erfc\left (\frac{\nu}{\sqrt{2}A_0}\right )-\frac{e^{\frac{(1-A_0^2)\nu^2}{2A_0^2}}}{A_0^2}\frac{2e^{-\frac{\nu^2}{2A_0^2}}}{\sqrt{2}\sqrt{\pi}}\right)\nonumber \\
&&+\frac{1-\beta}{w_1}\frac{2e^{-\frac{\nu^2}{2}}}{\sqrt{2}\sqrt{\pi}}\nonumber \\
&=& \frac{\beta(1-A_0^2)\nu}{A_0^2}+\frac{1-\beta}{w_1}\left (\frac{(1-A_0^2)\nu}{A_0^2}\frac{e^{\frac{(1-A_0^2)\nu^2}{2A_0^2}}}{A_0}\erfc\left (\frac{\nu}{\sqrt{2}A_0}\right )-\frac{1-A_0^2}{A_0^2}\frac{2e^{-\frac{\nu^2}{2}}}{\sqrt{2}\sqrt{\pi}}\right)\nonumber \\
&=& \frac{1}{w_1}\left (\frac{\beta(1-A_0^2)\nu}{A_0^2}\erf\left (\frac{\nu}{\sqrt{2}}\right )+\frac{(1-A_0^2)\nu}{A_0^3}\frac{\erfc\left (\frac{\nu}{\sqrt{2}A_0}\right )}{e^{-\frac{(1-A_0^2)\nu^2}{2A_0^2}}}-(1-\beta)\left (\frac{1-A_0^2}{A_0^2}\right )\frac{2e^{-\frac{\nu^2}{2}}}{\sqrt{2}\sqrt{\pi}}\right)\nonumber \\
&=& \frac{1-A_0^2}{w_1A_0^2}\left (\beta\nu\erf\left (\frac{\nu}{\sqrt{2}}\right )+\frac{\nu}{A_0}\frac{\erfc\left (\frac{\nu}{\sqrt{2}A_0}\right )}{e^{-\frac{(1-A_0^2)\nu^2}{2A_0^2}}}-(1-\beta)\sqrt{\frac{2}{\pi}}e^{-\frac{\nu^2}{2}}\right).\nonumber \\
.\label{eq:detanalIeer4}
\end{eqnarray}
We will also compute the derivative with respect to $c_3$. We start with
\begin{equation}
\frac{d\zeta_{\alpha,\beta}(c_3,\nu,A_0)}{dc_3}=-c_3+\frac{c_3}{1-A_0^2}+\frac{dI_{sph}}{dc_3}.
\label{eq:detanalIeer5}
\end{equation}
Then we also have
\begin{equation}
\frac{dI_{sph}}{dc_3}=\frac{d\widehat{\gamma}}{dc_3}c_3+\widehat{\gamma}-\frac{\alpha}{2(1-c_3/2/\widehat{\gamma})}
(-1/2/\widehat{\gamma}+1/2/\widehat{\gamma}^2\frac{d\widehat{\gamma}}{dc_3}),
\label{eq:detanalIeer6}
\end{equation}
where
\begin{equation}
\frac{d\widehat{\gamma}}{dc_3}=\frac{1}{4}-\frac{c_3}{4\sqrt{c_3^2+4\alpha}}.
\label{eq:detanalIeer7}
\end{equation}
Combining (\ref{eq:detanalIeer6}) and (\ref{eq:detanalIeer7}) we find
\begin{equation}
\frac{dI_{sph}}{dc_3}= \frac{c_3-\sqrt{(c_3)^2+4\alpha}}{2}.
\label{eq:detanalIeer8}
\end{equation}
Combining further (\ref{eq:detanalIeer5}) and (\ref{eq:detanalIeer8}) we finally obtain for the derivative wiht respect to $c_3$
\begin{equation}
\frac{d\zeta_{\alpha,\beta}(c_3,\nu,A_0)}{dc_3}=-c_3+\frac{c_3}{1-A_0^2}+\frac{c_3-\sqrt{(c_3)^2+4\alpha}}{2}.
\label{eq:detanalIeer9}
\end{equation}
Finally we compute the derivative with respect to $A_0$ as well. To that end we have
\begin{eqnarray}
\frac{d\zeta_{\alpha,\beta}(c_3,\nu,A_0)}{dA_0}&=&\frac{d}{dA_0}\left (-\frac{c_3^2}{2}+I_{sph}+(1-\beta)\log{w_1}+\beta\log{w_2}+\frac{c_3^2}{2(1-A_0^2)}\right )\nonumber \\
&=& (1-\beta)\frac{d\log{w_1}}{dA_0}+\beta\left (\nu^2\frac{d}{dA_0}\left (\frac{1-A_0^2}{2A_0^2}\right )-\frac{1}{A_0}\right )+\frac{c_3^2A_0}{(1-A_0^2)^2}\nonumber \\
&=& (1-\beta)\frac{d\log{w_1}}{dA_0}-\frac{\beta\nu^2}{A_0^3}-\frac{\beta A_0^2}{A_0^3}+\frac{c_3^2A_0}{(1-A_0^2)^2},\nonumber \\
\label{eq:detanalIeer10}
\end{eqnarray}
and
\begin{eqnarray}
\frac{d\log{w_1}}{dA_0}=\frac{d\log{(\frac{1}{A_0}e^{\frac{\nu^2}{2A_0^2}}\erfc(\frac{\nu}{\sqrt{2}A_0})+e^{\frac{\nu^2}{2}}\erf(\frac{\nu}{\sqrt{2}}))}}{dA_0}=
-\frac{e^{\frac{\nu^2}{2A_0^2}}(A_0^2+\nu^2)\erfc(\frac{\nu}{\sqrt{2}A_0})-\sqrt{\frac{2}{\pi}}A_0\nu}
{A_0^3(e^{\frac{\nu^2}{2A_0^2}}\erfc(\frac{\nu}{\sqrt{2}A_0})+A_0e^{\frac{\nu^2}{2}}\erf(\frac{\nu}{\sqrt{2}}))}.\nonumber \\
\label{eq:detanalIeer11}
\end{eqnarray}
We will below select certain values for $c_3$, $\nu$, and $A_0$ and check if one has all of the above derivatives equal to zeros for such values. Before doing that we will first facilitate the procedure a bit by recognizing from (\ref{eq:detanalIeer9}) that setting the derivative with respect to $c_3$ to zero implies that
\begin{equation}\label{eq:detanalIeer12}
  c_3=\frac{(1-A_0^2)\sqrt{\alpha}}{A_0}.
\end{equation}
Using (\ref{eq:detanalIeer12}) one then transforms (\ref{eq:detanalIeer10}) to obtain
\begin{equation}
\frac{d\zeta_{\alpha,\beta}(c_3,\nu,A_0)}{dA_0}=
 (1-\beta)\frac{d\log{w_1}}{dA_0}-\frac{\beta\nu^2}{A_0^3}-\frac{\beta A_0^2}{A_0^3}+\frac{\alpha A_0^2}{A_0^3}= (1-\beta)\frac{d\log{w_1}}{dA_0}+\frac{(\alpha-\beta) A_0^2-\beta\nu^2}{A_0^3}.
\label{eq:detanalIeer13}
\end{equation}

\subsubsection{Selecting the values for $c_3$, $\nu$, and $A_0$}
\label{sec:valuesc3nuA0}

We now recall that what we are looking for is essentially a collection of values for $c_3$, $\nu$, and $A_0$ so to compute $I_{err,u}^{(ub)}$ in (\ref{eq:detanalIeer2}) and ultimately in Theorem \ref{thm:ldp2}. What we are given though is a pair $(\alpha,\beta)$ (we recall that the regime we consider is asymptotic and $\alpha=\frac{m}{n}$ and $\beta=\frac{k}{n}$; $m$ is the number of equations in (\ref{eq:system}), $n$ is the length of $\x$, and $k$ is its sparsity). We also recall that in this section we are cosidering the so called upper tail i.e. the scenario where $\alpha$ is larger than the breaking point $\alpha_w$ obtained from $\ell_1$'s PT, i.e. from $\psi_\beta(\alpha_w)=\xi_{\alpha_w}(\beta)=1$.

Now, we will select $\nu$ and $A_0$ in the following way. Let $\beta_w$ be the solution of the fundamental $\ell_1$ PT obtained for our given $\alpha$, i.e. let $\beta_w$ be such that
\begin{equation}\label{eq:selvalc3nuA01}
  \xi_{\alpha}(\beta_w)=(1-\beta_w)\frac{\sqrt{\frac{2}{\pi}}e^{-(\erfinv(\frac{1-\alpha}{1-\beta_w}))^2}}{\alpha\sqrt{2}\erfinv (\frac{1-\alpha}{1-\beta_w})}=1.
\end{equation}
One should quickly note that unless $\alpha=\alpha_w$, $\beta_w$ is different from the given $\beta$ (i.e. different from the $\beta$ given in the pair $(\alpha,\beta)$). Now we set
\begin{equation}\label{eq:selvalc3nuA02}
  \nu=\sqrt{2}\erfinv \left (\frac{1-\alpha}{1-\beta_w}\right ).
\end{equation}
Further let $\beta_0$ be such that
\begin{equation}\label{eq:selvalc3nuA03}
  \frac{\alpha-\beta}{\alpha-\beta_0}\xi_{\alpha}(\beta_0)=\frac{\alpha-\beta}{\alpha-\beta_0}(1-\beta_0)\frac{\sqrt{\frac{2}{\pi}}e^{-(\erfinv(\frac{1-\alpha}{1-\beta_0}))^2}}{\alpha\sqrt{2}\erfinv (\frac{1-\alpha}{1-\beta_0})}=1.
\end{equation}
Then we set
\begin{equation}\label{eq:selvalc3nuA04}
  A_0=\frac{\erfinv \left (\frac{1-\alpha}{1-\beta_w}\right )}{\erfinv \left (\frac{1-\alpha}{1-\beta_0}\right )}=\frac{\nu}{\sqrt{2}\erfinv \left (\frac{1-\alpha}{1-\beta_0}\right )}.
\end{equation}
Finally combining (\ref{eq:detanalIeer12}) and (\ref{eq:selvalc3nuA04}) we obtain the value for $c_3$
\begin{equation}\label{eq:selvalc3nuA05}
  c_3=\frac{(1-A_0^2)\sqrt{\alpha}}{A_0}=\frac{\left (\erfinv \left (\frac{1-\alpha}{1-\beta_0}\right )\right )^2- \left (\erfinv\left (\frac{1-\alpha}{1-\beta_w}\right )\right )^2}{\erfinv \left (\frac{1-\alpha}{1-\beta_0}\right )\erfinv \left (\frac{1-\alpha}{1-\beta_w}\right )}\sqrt{\alpha}.
\end{equation}

\subsubsection{Rechecking the derivatives}
\label{sec:recheckder}

In this subsection we take the above selected values for $c_3$, $\nu$, and $A_0$ and recheck if they indeed ensure that the derivatives are equal to zero. That basically amounts to checking if the expressions on the right hand sides of (\ref{eq:detanalIeer4}), (\ref{eq:detanalIeer9}), and (\ref{eq:detanalIeer11}) are equal to zero if $\nu$, $A_0$, and $c_3$ are as in (\ref{eq:selvalc3nuA02}), (\ref{eq:selvalc3nuA04}), and (\ref{eq:selvalc3nuA05}), respectively and $\beta_w$ and $\beta_0$ are as in (\ref{eq:selvalc3nuA01}) and (\ref{eq:selvalc3nuA03}), respectively.

We first observe that (\ref{eq:detanalIeer9}) is trivially satisfied by the way how we chose $c_3$ and move to check (\ref{eq:detanalIeer4}). Combining (\ref{eq:detanalIeer4}), (\ref{eq:selvalc3nuA02}), and (\ref{eq:selvalc3nuA04}) we obtain
\begin{eqnarray}
\frac{d\zeta_{\alpha,\beta}(c_3,\nu,A_0)}{d\nu}&=& \frac{1-A_0^2}{w_1A_0^2}\left (\beta\nu\erf\left (\frac{\nu}{\sqrt{2}}\right )+\frac{\nu}{A_0}\frac{\erfc\left (\frac{\nu}{\sqrt{2}A_0}\right )}{e^{-\frac{(1-A_0^2)\nu^2}{2A_0^2}}}-(1-\beta)\sqrt{\frac{2}{\pi}}e^{-\frac{\nu^2}{2}}\right)\nonumber \\
&=&\frac{1-A_0^2}{w_1A_0^2}\left (\beta\sqrt{2}\erfinv \left (\frac{1-\alpha}{1-\beta_w}\right )\left (\frac{1-\alpha}{1-\beta_w}\right )
+\frac{\nu}{A_0}\frac{\erfc\left (\frac{\nu}{\sqrt{2}A_0}\right )}{e^{-\frac{(1-A_0^2)\nu^2}{2A_0^2}}}
-\frac{(1-\beta)\sqrt{2}\alpha\erfinv \left (\frac{1-\alpha}{1-\beta_w}\right )}{1-\beta_w}\right)\nonumber \\
&=&\frac{1-A_0^2}{w_1A_0^2}\left (
\frac{\nu}{A_0}\frac{\erfc\left (\frac{\nu}{\sqrt{2}A_0}\right )}{e^{-\frac{(1-A_0^2)\nu^2}{2A_0^2}}}
-\frac{(\alpha-\beta)\sqrt{2}\erfinv \left (\frac{1-\alpha}{1-\beta_w}\right )}{1-\beta_w}\right)\nonumber \\
&=&\frac{1-A_0^2}{w_1A_0^2e^{-\frac{(1-A_0^2)\nu^2}{2A_0^2}}}\left (
\frac{\nu}{A_0}\erfc\left (\frac{\nu}{\sqrt{2}A_0}\right )
-\frac{(\alpha-\beta)\sqrt{2}\erfinv \left (\frac{1-\alpha}{1-\beta_w}\right )}{1-\beta_w}e^{-\frac{(1-A_0^2)\nu^2}{2A_0^2}}\right)\nonumber \\
&=&\frac{1-A_0^2}{w_1A_0^2e^{-\frac{(1-A_0^2)\nu^2}{2A_0^2}}}\left (
\frac{\nu}{A_0}\erfc\left (\frac{\nu}{\sqrt{2}A_0}\right )
-\frac{(\alpha-\beta)\sqrt{\frac{2}{\pi}}e^{-\frac{\nu^2}{2}}}{\alpha}e^{-\frac{(1-A_0^2)\nu^2}{2A_0^2}}\right)\nonumber \\
&=&\frac{1-A_0^2}{\alpha w_1A_0^2e^{-\frac{(1-A_0^2)\nu^2}{2A_0^2}}}\left (
\frac{\alpha\nu}{A_0}\erfc\left (\frac{\nu}{\sqrt{2}A_0}\right )
-(\alpha-\beta)\sqrt{\frac{2}{\pi}}e^{-\frac{\nu^2}{2A_0^2}}\right)\nonumber \\
&=&\frac{1-A_0^2}{\alpha w_1A_0^2e^{-\frac{(1-A_0^2)\nu^2}{2A_0^2}}}\left (
\alpha\sqrt{2}\erfinv\left (\frac{1-\alpha}{1-\beta_0}\right )\left (\frac{\alpha-\beta_0}{1-\beta_0}\right )
-(\alpha-\beta)\sqrt{\frac{2}{\pi}}e^{-\left (\erfinv\left (\frac{1-\alpha}{1-\beta_0}\right )\right )^2}\right)\nonumber \\
& = & 0,\label{eq:recheck1}
\end{eqnarray}
where the second and fifth equality hold because of (\ref{eq:selvalc3nuA01}) and (\ref{eq:selvalc3nuA02}), the seventh equality holds because of (\ref{eq:selvalc3nuA04}), the last one because of (\ref{eq:selvalc3nuA03}), and the remaining ones follow through simple algebraic transformations. The above then confirms that the choice we made in (\ref{eq:selvalc3nuA01})-(\ref{eq:selvalc3nuA05}) indeed ensures that $\frac{d\zeta_{\alpha,\beta}(c_3,\nu,A_0)}{d\nu}=0$.

Now, to check the derivative with respect to $A_0$ we look at a combination of (\ref{eq:detanalIeer11}) and (\ref{eq:detanalIeer13}) to find
\begin{eqnarray}
\frac{d\zeta_{\alpha,\beta}(c_3,\nu,A_0)}{dA_0} & = &  (1-\beta)\frac{d\log{w_1}}{dA_0}+\frac{(\alpha-\beta) A_0^2-\beta\nu^2}{A_0^3}\nonumber \\
& = & -(1-\beta)\frac{e^{\frac{\nu^2}{2A_0^2}}(A_0^2+\nu^2)\erfc(\frac{\nu}{\sqrt{2}A_0})-\sqrt{\frac{2}{\pi}}A_0\nu}
{A_0^3(e^{\frac{\nu^2}{2A_0^2}}\erfc(\frac{\nu}{\sqrt{2}A_0})+A_0e^{\frac{\nu^2}{2}}\erf(\frac{\nu}{\sqrt{2}}))}+\frac{(\alpha-\beta) A_0^2-\beta\nu^2}{A_0^3}\nonumber \\
& = & -(1-\beta)\frac{(A_0^2+\nu^2)\frac{\alpha-\beta}{\alpha}\sqrt{\frac{2}{\pi}}\frac{A_0}{\nu}-\sqrt{\frac{2}{\pi}}A_0\nu}
{A_0^3(\frac{\alpha-\beta}{\alpha}\sqrt{\frac{2}{\pi}}\frac{A_0}{\nu}+A_0e^{\frac{\nu^2}{2}}\erf(\frac{\nu}{\sqrt{2}}))}+\frac{(\alpha-\beta) A_0^2-\beta\nu^2}{A_0^3}\nonumber \\
& = & -(1-\beta)\frac{(A_0^2+\nu^2)\frac{\alpha-\beta}{\alpha}\sqrt{\frac{2}{\pi}}\frac{A_0}{\nu}-\sqrt{\frac{2}{\pi}}A_0\nu}
{A_0^3(\frac{\alpha-\beta}{\alpha}\sqrt{\frac{2}{\pi}}\frac{A_0}{\nu}+\frac{1-\alpha}{\alpha}\sqrt{\frac{2}{\pi}}\frac{A_0}{\nu})}+\frac{(\alpha-\beta) A_0^2-\beta\nu^2}{A_0^3}\nonumber \\
& = & -\frac{(A_0^2+\nu^2)(\alpha-\beta)-\alpha\nu^2}
{A_0^3}
+\frac{(\alpha-\beta) A_0^2-\beta\nu^2}{A_0^3}\nonumber \\
& = & 0,\label{eq:recheck2}
\end{eqnarray}
where the third equality holds through a combination of equalities six and eight in (\ref{eq:recheck1}), the fourth equality follows because of (\ref{eq:selvalc3nuA01}) and (\ref{eq:selvalc3nuA02}), and the remaining ones follow through basic algebraic transformations. (\ref{eq:recheck2}) then confirms that the choice we made in (\ref{eq:selvalc3nuA01})-(\ref{eq:selvalc3nuA05}) indeed ensures that $\frac{d\zeta_{\alpha,\beta}(c_3,\nu,A_0)}{dA_0}=0$.

Having all the derivatives equal to zero ensures that the selection (\ref{eq:selvalc3nuA01})-(\ref{eq:selvalc3nuA05}) at the very least determines a stationary point of the underlying optimization problem. One can then proceed and check the second derivatives to ensure that the this selection in fact is the global optimum. That can be done both analytically and numerically. Analytical computations are more involved and we refrain from presenting them as they don't bring any novel ideas. Instead we will in the following sections prove that the choice (\ref{eq:selvalc3nuA01})-(\ref{eq:selvalc3nuA05}) is not only precisely the one that solves the optimization in (\ref{eq:detanalIeer2}) but also precisely the one that determines $I_{err}$. In the following subsection we will compute the value of $\zeta_{\alpha,\beta}(c_3,\nu,A_0)$ that one gets if $c_3$, $\nu$, and $A_0$ are as in (\ref{eq:selvalc3nuA01})-(\ref{eq:selvalc3nuA05}). As stated above and as will turn out later on, this  value of $\zeta_{\alpha,\beta}(c_3,\nu,A_0)$ will be precisely the $I_{err}$ from (\ref{eq:ldpasymp1}) and Theorem \ref{thm:ldp2}.

\subsubsection{Computing $\zeta_{\alpha,\beta}(c_3,\nu,A_0)$}
\label{sec:computingzeta}

Before computing $\zeta_{\alpha,\beta}(c_3,\nu,A_0)$ we will first compute all other quantities in (\ref{eq:detanalIeer3}), namely, $\widehat{\gamma}$, $I_{sph}$, $w_1$, and $w_2$. We start with $\widehat{\gamma}$ and to that end write
\begin{equation}\label{eq:computingzeta1}
  \widehat{\gamma}= \frac{c_3-\sqrt{(c_3)^2+4\alpha}}{4}=\frac{\frac{(1-A_0^2)\sqrt{\alpha}}{A_0}
-\sqrt{(\frac{(1-A_0^2)\sqrt{\alpha}}{A_0}
)^2+4\alpha}}{4}= \frac{\frac{(1-A_0^2)\sqrt{\alpha}}{A_0}
-\frac{(1-A_0^2)\sqrt{\alpha}}{A_0}}{4}=-\frac{A_0\sqrt{\alpha}}{2},
\end{equation}
where of course utilized $c_3$ from (\ref{eq:selvalc3nuA05}). In a similar fashion we then have for $I_{sph}$
\begin{equation}\label{eq:computingzeta2}
I_{sph} = \widehat{\gamma}c_3-\frac{\alpha }{2}\log\left (1-\frac{c_3}{2\widehat{\gamma}}\right )
=-\frac{(1-A_0^2)\alpha}{2}+\alpha\log(A_0).
\end{equation}
For $w_1$ we have
\begin{equation}\label{eq:computingzeta3}
w_1 =
  \frac{e^{\frac{(1-A_0^2)\nu^2}{2A_0^2}}}{A_0}\erfc\left (\frac{\nu}{\sqrt{2}A_0}\right )+\erf\left (\frac{\nu}{\sqrt{2}}\right )
  =\frac{\alpha-\beta}{\alpha\nu}\sqrt{\frac{2}{\pi}}e^{-\frac{\nu^2}{2}}+\erf\left (\frac{\nu}{\sqrt{2}}\right )
  =\frac{\alpha-\beta}{1-\beta_w}+\frac{1-\alpha}{1-\beta_w}=\frac{1-\beta}{1-\beta_w},
\end{equation}
where the second equality follows by a combination of the sixth and the eight equality in (\ref{eq:recheck1}) while the third equality follows by a combination of (\ref{eq:selvalc3nuA01}) and (\ref{eq:selvalc3nuA02}). Finally utilizing (\ref{eq:selvalc3nuA04}) and (\ref{eq:computingzeta3}) we have for $w_2$
\begin{equation}\label{eq:computingzeta4}
  w_2 =
  \frac{e^{\frac{(1-A_0^2)\nu^2}{2A_0^2}}}{A_0}=\frac{\alpha-\beta}{1-\beta_w}\frac{1}{\erfc\left (\frac{\nu}{\sqrt{2}A_0}\right )}
  =\frac{\alpha-\beta}{1-\beta_w}\frac{1}{1-\erf\left (\frac{\nu}{\sqrt{2}A_0}\right )}
  =\frac{\alpha-\beta}{1-\beta_w}\frac{1}{1-\frac{1-\alpha}{1-\beta_0}}
  =\frac{\alpha-\beta}{\alpha-\beta_0}\frac{1-\beta_0}{1-\beta_w}.
\end{equation}
A combination of (\ref{eq:computingzeta2}), (\ref{eq:computingzeta3}), and (\ref{eq:computingzeta4}) then gives
\begin{eqnarray}
\zeta_{\alpha,\beta}(c_3,\nu,A_0)&=&\left (-\frac{c_3^2}{2}+I_{sph}+(1-\beta)\log{w_1}+\beta\log{w_2}+\frac{c_3^2}{2(1-A_0^2)}\right )\nonumber \\
&=&\left (\frac{c_3^2A_0^2}{2(1-A_0^2)}-\frac{(1-A_0^2)\alpha}{2}+\alpha\log(A_0)+(1-\beta)\log{w_1}+\beta\log{w_2}\right )\nonumber \\
&=& \alpha\log(A_0)+(1-\beta)\log{w_1}+\beta\log{w_2}\nonumber \\
&=& \alpha\log\left (\frac{\erfinv\left (\frac{1-\alpha}{1-\beta_w}\right )}{\erfinv\left (\frac{1-\alpha}{1-\beta_0}\right )}\right )+(1-\beta)\log{\left (\frac{1-\beta}{1-\beta_w}\right )}+\beta\log{\left (\frac{\alpha-\beta}{\alpha-\beta_0}\frac{1-\beta_0}{1-\beta_w}\right )}.
\end{eqnarray}
We summarize the above results in the following theorem.
\begin{theorem}
Assume the setup of Theorem \ref{thm:ldp2} and assume that a pair $(\alpha,\beta)$  is given. Let $\alpha>\alpha_w$ where $\alpha_w$ is such that $\psi_\beta(\alpha_w)=\xi_{\alpha_w}(\beta)=1$. Also let $\beta_w$ satisfy the following \textbf{fundamental characterization of the $\ell_1$'s PT:}

\begin{equation}
\xi_{\alpha}(\beta_w)\triangleq
(1-\beta_w)\frac{\sqrt{\frac{2}{\pi}}e^{-\lp\erfinv\lp\frac{1-\alpha}{1-\beta_w}\rp\rp^2}}{\alpha_w\sqrt{2}\erfinv \lp\frac{1-\alpha}{1-\beta_w}\rp}=1.
\label{eq:thmldp3l1PT}
\end{equation}

\noindent Further let $\beta_0$ satisfy the following \textbf{fundamental characterization of the $\ell_1$'s LDP:}

\begin{equation}
\frac{\alpha-\beta}{\alpha-\beta_0}\xi_{\alpha}(\beta_0)=\frac{\alpha-\beta}{\alpha-\beta_0}
(1-\beta_0)\frac{\sqrt{\frac{2}{\pi}}e^{-\lp\erfinv\lp\frac{1-\alpha}{1-\beta_0}\rp\rp^2}}{\alpha\sqrt{2}\erfinv \lp\frac{1-\alpha}{1-\beta_0}\rp}=1.
\label{eq:thmldp3l1LDP}
\end{equation}

\noindent Then
\begin{equation}
I_{err}(\alpha,\beta)\triangleq\lim_{n\rightarrow\infty}\frac{\log{P_{err}}}{n}
\leq
\alpha\log\left (\frac{\erfinv\left (\frac{1-\alpha}{1-\beta_w}\right )}{\erfinv\left (\frac{1-\alpha}{1-\beta_0}\right )}\right )+(1-\beta)\log\left (\frac{1-\beta}{1-\beta_w}\right ) +\beta\log\left (\frac{(\alpha-\beta)(1-\beta_0)}{(\alpha-\beta_0)(1-\beta_w)}\right )\triangleq I_{ldp}(\alpha,\beta).
\label{eq:ldpthm3Ierrub1}
\end{equation}
Moreover, the following choice for $\nu$, $c_3$, and $\gamma$ in the optimization problem in Theorem \ref{thm:ldp2} achieves the right hand side of (\ref{eq:ldpthm3Ierrub1})
\begin{eqnarray}
  \nu &  = & \sqrt{2}\erfinv \left (\frac{1-\alpha}{1-\beta_w}\right )\nonumber \\
    A_0&=&\frac{\erfinv \left (\frac{1-\alpha}{1-\beta_w}\right )}{\erfinv \left (\frac{1-\alpha}{1-\beta_0}\right )}=\frac{\nu}{\sqrt{2}\erfinv \left (\frac{1-\alpha}{1-\beta_0}\right )}\nonumber \\
      c_3& = &\frac{(1-A_0^2)\sqrt{\alpha}}{A_0}=\frac{\left (\erfinv \left (\frac{1-\alpha}{1-\beta_0}\right )\right )^2- \left (\erfinv\left (\frac{1-\alpha}{1-\beta_w}\right )\right )^2}{\erfinv \left (\frac{1-\alpha}{1-\beta_0}\right )\erfinv \left (\frac{1-\alpha}{1-\beta_w}\right )}\sqrt{\alpha}\nonumber \\
      \gamma&=&\frac{c_3}{2(1-A_0^2)}=\frac{\sqrt{\alpha}}{2A_0}=\frac{\sqrt{\alpha}\erfinv \left (\frac{1-\alpha}{1-\beta_0}\right )}{2\erfinv \left (\frac{1-\alpha}{1-\beta_w}\right )}.
\label{eq:ldpthm3perrub2}
\end{eqnarray}\label{thm:ldp3}
\end{theorem}
\begin{proof} Follows from the above discussion.
\end{proof}
We will present the results that one can obtain based on the above theorem later on when we complement them with the lower tail estimate in the following subsection and eventually connect them to the high-dimensional geometry. At that time we will also find it useful to look at a couple of properties of the function $\frac{\alpha-\beta}{\alpha-\beta_0}\xi_{\alpha}(\beta_0)$.

\subsection{Lower tail}
\label{sec:lowertail}

%

In this section we focus on the lower tail type of large deviations. In the previous sections we focused on the scenario where the number of equations in the original linear system is larger than the critical breaking point. In such a scenario the probability of error is expected to be small and the LDP that we introduced and discussed essentially describes that. Contrary to that, in the scenario that we will focus on in this section, the number of equations will be assumed to be smaller than the what the $\ell_1$'s PT predicts for the threshold (breaking point). Consequently, one would expect that the probability of $\ell_1$ working correctly will be small. The LDP that we will discuss below will mathematically describe and confirm that. To that end we introduce a quantity complementary to the $P_{err}$ considered in the previous sections. We denote it by $P_{cor}$ and set
\begin{equation}
P_{cor} \triangleq P(\min_{A\w=0,\|\w\|_2\leq 1}\sum_{i=n-k+1}^n \w_i+\sum_{i=1}^{n-k}|\w_{i}|\geq 0).
\label{eq:ldpproblower}
\end{equation}
Clearly, $P_{cor}=1-P_{err}$ is the probability that (\ref{eq:l1}) does produce the solution of (\ref{eq:system}), i.e. correctly solves the original linear system. Utilizing the machinery of \cite{StojnicBlockasymldpfinn15} we then have
\begin{equation}
P_{cor}
\leq
P(\|\g\|_2 -w(\h,S_w)-t_1\geq 0)/P(g\geq t_1)
\leq  \min_{t_1}\min_{c_3\geq 0}
Ee^{c_3\|\g\|_2}Ee^{-c_3w(\h,S_w)}e^{-c_3t_1}/P(g\geq t_1).
\label{eq:ldpprob3lower}
\end{equation}
The above bound is valid for any integers $m$, $k$, and $n$ (of course, as earlier, assuming $k\leq m\leq n$ so that the results make sense) and one can then establish theorems/results similar to the ones in the previous section. Instead of repeating all of them we will just focus on the LDP ones.
To that end we define a quantity analogous to $I_{err}$, namely $I_{cor}$, in the following way
\begin{equation}\label{eq:ldpasymp1lower}
  I_{cor}\triangleq\lim_{n\rightarrow\infty}\frac{\log{P_{cor}}}{n},
\end{equation}
and below establish the lower tail analogue to Theorem \ref{thm:ldp2}.
\begin{theorem}
Assume the setup of Theorem \ref{thm:ldp2}. Then
\begin{equation}
I_{cor}\triangleq\lim_{n\rightarrow\infty}\frac{\log{P_{cor}}}{n}
\leq \min_{c_3\geq 0}\left (-\frac{c_3^2}{2}+I_{sph}^++\max_{\nu\geq 0,\gamma^{(s)}\geq 0} ((1-\beta_w)\log{w_1}+\beta_w\log{w_2}+\beta_w\log{w_3}-c_3\gamma)\right )\triangleq I_{cor,l}^{(ub)},
\label{eq:ldpthm2Icorub1}
\end{equation}
where
\begin{eqnarray}
I_{sph}^+ &=& \widehat{\gamma_+}c_3-\frac{\alpha d}{2}\log\left (1-\frac{c_3}{2\widehat{\gamma_+}}\right )\nonumber \\
  \widehat{\gamma_+} &=& \frac{2c_3+\sqrt{4c_3^2+16\alpha }}{8}\nonumber \\
w_1 &=& \frac{1}{\sqrt{2\pi}}\int_{\bar{h}}e^{-\bar{h}^2/2}e^{-c_3\max(|\bar{h}|-\nu,0)^2/4/\gamma}d\bar{h}
  =\frac{e^{\frac{-c_3\nu^2/4/\gamma}{1+c_3/2/\gamma}}}{\sqrt{1+c_3/2/\gamma}}\erfc\left (\frac{\nu}{\sqrt{2}\sqrt{1+c_3/2/\gamma}}\right )+\erf\left (\frac{\nu}{\sqrt{2}}\right )\nonumber \\
  w_2 &=& \frac{1}{\sqrt{2\pi}}\int_{\bar{h}}e^{-\bar{h}^2/2}e^{-c_3(\bar{h}+\nu)^2/4/\gamma}d\bar{h}
  =\frac{e^{\frac{-c_3\nu^2/4/\gamma}{1+c_3/2/\gamma}}}{\sqrt{1+c_3/2/\gamma}}.\label{eq:ldpthm2perrub2lower}
\end{eqnarray}\label{thm:ldp2lower}
\end{theorem}
\begin{proof} Follows from the above considerations, what was presented in \cite{StojnicBlockasymldpfinn15}, and by noting that in \cite{StojnicMoreSophHopBnds10} we established
\begin{equation}
I_{sph}^+=\lim_{n\rightarrow\infty}\frac{1}{n}\log(Ee^{c_3\sqrt{n}\|\g\|_2})
=\widehat{\gamma_+} c_3-\frac{\alpha}{2}\log\left (1-\frac{c_3}{2\widehat{\gamma_+}}\right ),\label{eq:gamaiden2liftlowe}
\end{equation}
where
\begin{equation}
\widehat{\gamma_+}=\frac{2c_3+\sqrt{4c_3^2+16\alpha}}{8}.\label{eq:gamaiden3liftlower}
\end{equation}
\end{proof}
One can then proceed as in Section \ref{sec:analysisIerr}, albeit in a much faster fashion, as we will see below.

\subsection{A detailed analysis of $I_{cor,l}^{(ub)}$}
\label{sec:analysisIcor}

After noting the change $c_3\rightarrow -c_3$ we again have as in Section \ref{sec:analysisIerr}
\begin{equation}\label{eq:detanalIcor1}
  A_{0}\triangleq\sqrt{1-\frac{c_3}{2\gamma}}.
\end{equation}
and
\begin{equation}\label{eq:detanalIcor2}
I_{cor,l}^{(ub)}(\alpha,\beta)\triangleq \min_{c_3\leq 0}\max_{\nu\geq 0,A_0\leq 1}\zeta_{\alpha,\beta}(c_3,\nu,A_0)
\end{equation}
where
\begin{eqnarray}
\zeta_{\alpha,\beta}(c_3,\nu,A_0)&=&\left (-\frac{c_3^2}{2}+I_{sph}^++(1-\beta)\log{w_1}+\beta\log{w_2}+\frac{c_3^2}{2(1-A_0^2)}\right )\nonumber \\
I_{sph}^+ &=& -\widehat{\gamma^+}c_3-\frac{\alpha }{2}\log\left (1+\frac{c_3}{2\widehat{\gamma^+}}\right )\nonumber \\
  \widehat{\gamma^+} &=& \frac{-c_3+\sqrt{c_3^2+4\alpha}}{4}=-\widehat{\gamma}\nonumber \\
w_1 &=&
  \frac{e^{\frac{(1-A_0^2)\nu^2}{2A_0^2}}}{A_0}\erfc\left (\frac{\nu}{\sqrt{2}A_0}\right )+\erf\left (\frac{\nu}{\sqrt{2}}\right )\nonumber \\
  w_2 &=&
  \frac{e^{\frac{(1-A_0^2)\nu^2}{2A_0^2}}}{A_0}.\label{eq:detanalIcor3}
\end{eqnarray}
One now observes that $\zeta_{\alpha,\beta}(c_3,\nu,A_0)$ defined in (\ref{eq:detanalIcor3}) is exactly the same as the corresponding one in (\ref{eq:detanalIeer3}). That means that one can proceed with computation of all the derivatives as earlier and all the results will be the same. Consequently, the selected values for $c_3$, $\nu$, $\gamma$, and $A_0$ will have the same form. Instead of repeating all these calculations we summarize them in the following theorem, essentially a lower tail analogue of Theorem \ref{thm:ldp3}.
\begin{theorem}
Assume the setup of Theorem \ref{thm:ldp3} and assume that a pair $(\alpha,\beta)$  is given. Differently from Theorem \ref{thm:ldp3}, let $\alpha<\alpha_w$ where $\alpha_w$ is such that $\psi_\beta(\alpha_w)=\xi_{\alpha_w}(\beta)=1$. Also let $\beta_w$ and $\beta_0$ satisfy the \textbf{fundamental $\ell_1$'s PT and LDP characterizations}, respectively as in Theorem \ref{thm:ldp3}. Then choosing $\nu$, $c_3$, and $\gamma$ in the optimization problem in (\ref{eq:ldpthm2Icorub1}) as $\nu$, $-c_3$, and $\gamma$  from Theorem \ref{thm:ldp3} (or equivalently, chooosing $\nu$, $c_3$, and $A_0$ in the optimization problem in (\ref{eq:detanalIcor2}) as $\nu$, $c_3$, and $A_0$  from Theorem \ref{thm:ldp3}) gives
\begin{equation}
\zeta_{\alpha,\beta}(c_3,\nu,A_0)=
\alpha\log\left (\frac{\erfinv\left (\frac{1-\alpha}{1-\beta_w}\right )}{\erfinv\left (\frac{1-\alpha}{1-\beta_0}\right )}\right )+(1-\beta)\log\left (\frac{1-\beta}{1-\beta_w}\right ) +\beta\log\left (\frac{(\alpha-\beta)(1-\beta_0)}{(\alpha-\beta_0)(1-\beta_w)}\right )=I_{ldp}(\alpha,\beta).
\label{eq:ldpthm3Icorub1}
\end{equation}
\label{thm:ldp3lower}
\end{theorem}
\begin{proof} Follows from the considerations leading to Theorem \ref{thm:ldp3}.
\end{proof}
We should also add that in the scenario of interest in Theorem \ref{thm:ldp3lower}, i.e. for $\alpha<\alpha_w$, one has $\beta_w<\beta_0$ which implies $A_0>1$ and ultimately $c_3<0$. Of course, exactly opposite happens in the scenario of interest in Theorem \ref{thm:ldp3}.

\subsection{High-dimensional geometry}
\label{sec:hdg}

In this section we look at an alternative way to characterize the performance of $\ell_1$. It utilizes some of the basic concepts from the high-dimensional integral geometry and we will assume a solid degree of familiarity with these. We will assume that we are given a pair $(\alpha,\beta)$ and that $\beta_w$ and $\beta_0$ are given by the $\ell_1$ fundamental PT and LDP characterizations defined earlier. Also, we assume the upper tail regimes, i.e. $\alpha>\alpha_w$ (where $\alpha_w$ is such that $\psi_\beta(\alpha_w)=\xi_{\alpha_w}(\beta)=1$) and start with the following collection of results established in \cite{hdg}.
\begin{equation}\label{eq:hdg1}
  \Psi_{net}(\alpha,\beta)=I_{err}(\alpha,\beta)\triangleq\lim_{n\rightarrow\infty}\frac{\log{P_{err}}}{n}=\psicom+\psiint-\psiext,
\end{equation}
where
\begin{eqnarray}
\psicom & = & (\alpha-\beta) \log(2)-(\alpha-\beta)\log\left (\frac{\alpha-\beta}{1-\beta}\right )-(1-\alpha)\log\left (\frac{1-\alpha}{1-\beta}\right )\nonumber \\
\psiint & = & \min_{y\geq 0} (\alpha y^2 +(\alpha-\beta)\log(\erfc(y)))- (\alpha-\beta) \log(2)\nonumber\\
\psiext & = & \min_{y\geq 0} (\alpha y^2 -(1-\alpha)\log(\erf(y))). \label{eq:hdg2}
\end{eqnarray}
Let $y_{int}$ and $y_{ext}$ be the solutions of the above optimizations (clearly, $y_{int}$ is the solution of the optimization associated with $\psiint$ and $y_{ext}$ is the solution of the optimization associated with $\psiext$). To determine $y_{int}$ we start by taking the following derivative
\begin{equation}
\frac{d(\alpha y^2 +(\alpha-\beta)\log(\erfc(y)))}{dy}  =  2\alpha y+\frac{\alpha-\beta}{\erfc(y)}\frac{d\erfc(y)}{dy}= 2\alpha y-\frac{\alpha-\beta}{1-\erf(y)}\frac{2e^{-y^2}}{\sqrt{\pi}}. \label{eq:hdg3}
\end{equation}
Choosing
\begin{equation}\label{eq:hdg4}
y_{int}  =  \erfinv\left (\frac{1-\alpha}{1-\beta_0}\right ),
\end{equation}
the derivative in (\ref{eq:hdg3}) becomes
\begin{eqnarray}
\frac{d(\alpha y^2 +(\alpha-\beta)\log(\erfc(y)))}{dy} | _{y=y_{int}}  & = &   2\alpha y_{ext}-\frac{\alpha-\beta}{1-\erf(y_{int})}\frac{2e^{-y_{int}^2}}{\sqrt{\pi}}\nonumber \\
& = &2\alpha \erfinv\left (\frac{1-\alpha}{1-\beta_0}\right )-\frac{\alpha-\beta}{\alpha-\beta_0}(1-\beta_0)\frac{2e^{-\left (\erfinv\left (\frac{1-\alpha}{1-\beta_0}\right )\right )^2}}{\sqrt{\pi}}\nonumber \\
& = &2\alpha \erfinv\left (\frac{1-\alpha}{1-\beta_0}\right )\left (1-\frac{\alpha-\beta}{\alpha-\beta_0}(1-\beta_0)\frac{\sqrt{\frac{2}{\pi}}e^{-\left (\erfinv\left (\frac{1-\alpha}{1-\beta_0}\right )\right )^2}}{\sqrt{2}\alpha \erfinv\left (\frac{1-\alpha}{1-\beta_0}\right )}\right )\nonumber \\
& = &2\alpha \erfinv\left (\frac{1-\alpha}{1-\beta_0}\right )\left (1-\frac{\alpha-\beta}{\alpha-\beta_0}\xi_\alpha(\beta_0)\right )\nonumber \\
& = &0, \label{eq:hdg5}
\end{eqnarray}
where the last equality follows by the fundamental characterization of $\ell_1$'s LDP. Continuing further we obtain for the second derivative
\begin{equation}
\frac{d^2(\alpha y^2 +(\alpha-\beta)\log(\erfc(y)))}{dy^2}   =    2\alpha -\frac{\alpha-\beta}{\erfc(y)^2}\left (\frac{2e^{-y^2}}{\sqrt{\pi}}\right )^2+\frac{\alpha-\beta}{\erfc(y)}\left (\frac{4ye^{-y^2}}{\sqrt{\pi}}\right ). \label{eq:hdg6}
\end{equation}
One also has
\begin{equation}
 -\frac{\alpha-\beta}{\erfc(y)^2}\left (\frac{2e^{-y^2}}{\sqrt{\pi}}\right )^2+\frac{\alpha-\beta}{\erfc(y)}\left (\frac{4ye^{-y^2}}{\sqrt{\pi}}\right )
  =
-\frac{\alpha-\beta}{\erfc(y)}\left (\frac{4ye^{-y^2}}{\sqrt{\pi}}\right )\left (\frac{e^{-y^2}}{\sqrt{\pi}y\erfc(y)}-1 \right )<0, \label{eq:hdg7}
\end{equation}
where the last equality follows by $\alpha>\beta$ and a combination of (\ref{eq:propxi33}), (\ref{eq:propxi34}), and (\ref{eq:propxi35}). Moreover
\begin{equation}
 -\frac{\alpha-\beta}{\erfc(y)^2}\left (\frac{2e^{-y^2}}{\sqrt{\pi}}\right )^2+\frac{\alpha-\beta}{\erfc(y)}\left (\frac{4ye^{-y^2}}{\sqrt{\pi}}\right )
  >
 -\frac{\alpha}{\erfc(y)^2}\left (\frac{2e^{-y^2}}{\sqrt{\pi}}\right )^2+\frac{\alpha}{\erfc(y)}\left (\frac{4ye^{-y^2}}{\sqrt{\pi}}\right ). \label{eq:hdg8}
\end{equation}
Combining further (\ref{eq:hdg6}) and (\ref{eq:hdg8}) we obtain
\begin{eqnarray}
\frac{d^2(\alpha y^2 +(\alpha-\beta)\log(\erfc(y)))}{dy^2}  & = &   2\alpha -\frac{\alpha-\beta}{\erfc(y)^2}\left (\frac{2e^{-y^2}}{\sqrt{\pi}}\right )^2+\frac{\alpha-\beta}{\erfc(y)}\left (\frac{4ye^{-y^2}}{\sqrt{\pi}}\right )\nonumber \\
& > &   2\alpha -\frac{\alpha}{\erfc(y)^2}\left (\frac{2e^{-y^2}}{\sqrt{\pi}}\right )^2+\frac{\alpha}{\erfc(y)}\left (\frac{4ye^{-y^2}}{\sqrt{\pi}}\right )\nonumber \\
& = &   \frac{2\alpha e^{-2y^2}}{\pi\erfc(y)^2}\left (\pi\erfc(y)^2e^{2y^2} -2+2 y \sqrt{\pi}\erfc(y)e^{y^2}\right ). \label{eq:hdg9}
\end{eqnarray}
We also recall on the following inequalities for $\erfc(\cdot)$ introduced earlier in Section \ref{sec:propxi}.
\begin{equation}
\frac{2}{\sqrt{\pi}}\frac{e^{-y^2}}{y+\sqrt{y^2+2}}< \erfc(y)\leq \frac{2}{\sqrt{\pi}}\frac{e^{-y^2}}{y+\sqrt{y^2+\frac{4}{\pi}}}.
 \label{eq:hdg10}
\end{equation}
Using (\ref{eq:hdg10}) we further obtain from (\ref{eq:hdg9})
\begin{eqnarray}
\frac{d^2(\alpha y^2 +(\alpha-\beta)\log(\erfc(y)))}{dy^2}  & > &      \frac{2\alpha e^{-2y^2}}{\pi\erfc(y)^2}\left (\pi\erfc(y)^2e^{2y^2} -2+2 y \sqrt{\pi}\erfc(y)e^{y^2}\right )\nonumber \\
 & > &      \frac{2\alpha e^{-2y^2}}{\pi\erfc(y)^2}\left (\frac{4}{(y+\sqrt{y^2+2})^2} -2+\frac{4y}{y+\sqrt{y^2+2}}\right )\nonumber \\
 & = &      \frac{2\alpha e^{-2y^2}}{\pi\erfc(y)^2}\left (\frac{4-2(y+\sqrt{y^2+2})^2+4y(y+\sqrt{y^2+2})}{(y+\sqrt{y^2+2})^2}\right )\nonumber \\
& = & 0. \label{eq:hdg11}
\end{eqnarray}
The above then means that $(\alpha y^2 +(\alpha-\beta)\log(\erfc(y)))$ is convex and that $y_{int}$ is not only its a local but also its a global optimum (minimum) as well. A combination of (\ref{eq:hdg2}) and (\ref{eq:hdg4}) together with $\ell_1$'s fundamental LDP then finally gives
\begin{eqnarray}\label{eq:hdg12}
\psiint & = &  \alpha y_{int}^2 +(\alpha-\beta)\log(\erfc(y_{int})- (\alpha-\beta) \log(2) \nonumber \\
 & = &   \alpha \left (\erfinv\left (\frac{1-\alpha}{1-\beta_0}\right )
\right )^2 +(\alpha-\beta)\log\left (\frac{\alpha-\beta_0}{1-\beta_0}\right )- (\alpha-\beta) \log(2)\nonumber \\
& = & \alpha \log\left (e^{\left (\erfinv\left (\frac{1-\alpha}{1-\beta_0}\right )
\right )^2}\right ) +(\alpha-\beta)\log\left (\frac{\alpha-\beta_0}{1-\beta_0}\right )- (\alpha-\beta) \log(2)\nonumber \\
& = & \alpha \log\left (\frac{\alpha-\beta}{\alpha-\beta_0}\frac{1-\beta_0}{\alpha\sqrt{2}\erfinv\left (\frac{1-\alpha}{1-\beta_0}\right )}\right ) +(\alpha-\beta)\log\left (\frac{\alpha-\beta_0}{1-\beta_0}\right )- (\alpha-\beta) \log(2)+\alpha\log\lp\sqrt{\frac{2}{\pi}}\rp\nonumber \\
& = & -\alpha \log\left (\sqrt{2}\erfinv\left (\frac{1-\alpha}{1-\beta_0}\right )\right )+\alpha\log\left (\frac{\alpha-\beta}{\alpha}\right )
-\beta\log\left (\frac{\alpha-\beta_0}{1-\beta_0}\right )- (\alpha-\beta) \log(2)+\alpha\log\lp\sqrt{\frac{2}{\pi}}\rp\nonumber \\
\end{eqnarray}
Using the $\ell_1$'s fundamental PT and the definition of $\beta_w$ one can then further utilize the results of \cite{StojnicEquiv10} to obtain
\begin{eqnarray}\label{eq:hdg13}
\psiext  & = &   \min_{y\geq 0} (\alpha y^2 -(1-\alpha)\log(\erf(y)))\nonumber \\
& = &\alpha\left (\erfinv\left (\frac{1-\alpha}{1-\beta_w}\right )\right )^2-(1-\alpha)\log\left (\frac{1-\alpha}{1-\beta_w}\right )\nonumber \\
& = &\alpha\log\left (e^{\left (\erfinv\left (\frac{1-\alpha}{1-\beta_w}\right )\right )^2}\right )-(1-\alpha)\log\left (\frac{1-\alpha}{1-\beta_w}\right )\nonumber \\
& = &\alpha \log\left (\frac{1-\beta_w}{\alpha\sqrt{2}\erfinv\left (\frac{1-\alpha}{1-\beta_w}\right )}\right )-(1-\alpha)\log\left (\frac{1-\alpha}{1-\beta_w}\right )+\alpha\log\lp\sqrt{\frac{2}{\pi}}\rp\nonumber \\
& = &-\alpha \log\left (\sqrt{2}\erfinv\left (\frac{1-\alpha}{1-\beta_w}\right )\right )+\alpha \log\left (\frac{1-\beta_w}{\alpha}\right )-(1-\alpha)\log\left (\frac{1-\alpha}{1-\beta_w}\right )+\alpha\log\lp\sqrt{\frac{2}{\pi}}\rp\nonumber \\
& = &-\alpha \log\left (\sqrt{2}\erfinv\left (\frac{1-\alpha}{1-\beta_w}\right )\right )-\alpha \log(\alpha)-(1-\alpha)\log (1-\alpha)+\log(1-\beta_w)
+\alpha\log\lp\sqrt{\frac{2}{\pi}}\rp.\nonumber \\
\end{eqnarray}
Finally one can combine (\ref{eq:hdg1}), (\ref{eq:hdg12}), and (\ref{eq:hdg13}) to obtain
\begin{eqnarray}
  \Psi_{net}(\alpha,\beta)&=&I_{err}(\alpha,\beta)=\psicom+\psiint-\psiext\nonumber \\
& = & (\alpha-\beta) \log(2)-(\alpha-\beta)\log\left (\frac{\alpha-\beta}{1-\beta}\right )-(1-\alpha)\log\left (\frac{1-\alpha}{1-\beta}\right )
-\alpha \log\left (\sqrt{2}\erfinv\left (\frac{1-\alpha}{1-\beta_0}\right )\right )\nonumber \\
&&+\alpha\log\left (\frac{\alpha-\beta}{\alpha}\right )
-\beta\log\left (\frac{\alpha-\beta_0}{1-\beta_0}\right )- (\alpha-\beta) \log(2)\nonumber\\
&&+\alpha \log\left (\sqrt{2}\erfinv\left (\frac{1-\alpha}{1-\beta_w}\right )\right )+\alpha \log(\alpha)+(1-\alpha)\log (1-\alpha)-\log(1-\beta_w)\nonumber \\
& = & \alpha\log\left (\frac{1-\beta}{\alpha-\beta}\frac{1-\alpha}{1-\beta}\frac{\alpha-\beta}{\alpha}\frac{\alpha}{1-\alpha}\right )+\beta\log\left (\frac{\alpha-\beta}{1-\beta}\frac{1-\beta_0}{\alpha-\beta_0}\right )
+\alpha \log\left (\frac{\erfinv\left (\frac{1-\alpha}{1-\beta_w}\right )}{\erfinv\left (\frac{1-\alpha}{1-\beta_0}\right )}\right )\nonumber \\
&&+\log(1-\beta)-\log(1-\beta_w)\nonumber \\
& = & \beta\log\left (\frac{\alpha-\beta}{\alpha-\beta_0}\frac{1-\beta_0}{1-\beta}\frac{1-\beta_w}{1-\beta_w}\right )
+\alpha \log\left (\frac{\erfinv\left (\frac{1-\alpha}{1-\beta_w}\right )}{\erfinv\left (\frac{1-\alpha}{1-\beta_0}\right )}\right )+
\log\left (\frac{1-\beta}{1-\beta_w}\right )\nonumber \\
& = & \alpha \log\left (\frac{\erfinv\left (\frac{1-\alpha}{1-\beta_w}\right )}{\erfinv\left (\frac{1-\alpha}{1-\beta_0}\right )}\right )+
(1-\beta)\log\left (\frac{1-\beta}{1-\beta_w}\right )+\beta\log\left (\frac{\alpha-\beta}{\alpha-\beta_0}\frac{1-\beta_0}{1-\beta_w}\right )
=I_{ldp}(\alpha,\beta).
\label{eq:hdg14}
\end{eqnarray}
A combination of (\ref{eq:ldpthm3Ierrub1}), (\ref{eq:hdg1}), and (\ref{eq:hdg14}) then gives
\begin{equation}\label{eq:hdg15}
  I_{err}=I_{ldp}(\alpha,\beta),
\end{equation}
and ensures that the choice for $\nu$, $A_0$, $c_3$, and $\gamma$ made in (\ref{eq:ldpthm3perrub2}) is indeed optimal. Moreover, in the lower tail regime ($\alpha<\alpha_w$, where $\alpha_w$ is such that $\psi_\beta(\alpha_w)=\xi_{\alpha_w}(\beta)=1$) considerations from \cite{StojnicEquiv10} ensure that one also has
\begin{equation}\label{eq:hdg1a}
  \Psi_{net}(\alpha,\beta)=I_{cor}(\alpha,\beta)\triangleq\lim_{n\rightarrow\infty}\frac{\log{P_{cor}}}{n}=\psicom+\psiint-\psiext,
\end{equation}
where $\psicom$, $\psiint$, and $\psiext$ are as in (\ref{eq:hdg2}). Finally, we are in position to fully characterize $\ell_1$'s LDP. The following theorem does so.
\begin{theorem}[$\ell_1$'s LDP]
Assume the setup of Theorem \ref{thm:thmweakthr} and assume that a pair $(\alpha,\beta)$ is given. Let $P_{err}$ be the probability that the solutions of (\ref{eq:system}) and (\ref{eq:l1}) coincide and let $P_{cor}$ be the probability that the solutions of (\ref{eq:system}) and (\ref{eq:l1}) do \emph{not} coincide. Let $\alpha_w$ and $\beta_w$ satisfy the \textbf{$\ell_1$'s fundamental PT} characterizations in the following way
\begin{equation}
\psi_{\beta}(\alpha_w)\triangleq
(1-\beta)\frac{\sqrt{\frac{2}{\pi}}e^{-\lp\erfinv\lp\frac{1-\alpha_w}{1-\beta}\rp\rp^2}}{\alpha_w\sqrt{2}\erfinv \lp\frac{1-\alpha_w}{1-\beta}\rp}=1\quad \mbox{and} \quad
\xi_{\alpha}(\beta_w)\triangleq
(1-\beta_w)\frac{\sqrt{\frac{2}{\pi}}e^{-\lp\erfinv\lp\frac{1-\alpha}{1-\beta_w}\rp\rp^2}}{\alpha\sqrt{2}\erfinv \lp\frac{1-\alpha}{1-\beta_w}\rp}=1.\label{eq:thmfinalldpl11}
\end{equation}
Further let $\beta_0$ satisfy the following \textbf{$\ell_1$'s fundamental LDP} characterization
\begin{equation}\label{eq:thmfinalldpl12}
\frac{\alpha-\beta}{\alpha-\beta_0}\xi_{\alpha}(\beta_0)=\frac{\alpha-\beta}{\alpha-\beta_0}
(1-\beta_0)\frac{\sqrt{\frac{2}{\pi}}e^{-(\erfinv(\frac{1-\alpha}{1-\beta_0}))^2}}{\alpha\sqrt{2}\erfinv (\frac{1-\alpha}{1-\beta_0})}=1.
\end{equation}
Finally, let $I_{ldp}(\alpha,\beta)$ be defined through the following \textbf{$\ell_1$'s fundamental LDP rate function} characterization
\begin{equation}
I_{ldp}(\alpha,\beta)\triangleq
\alpha\log\left (\frac{\erfinv\left (\frac{1-\alpha}{1-\beta_w}\right )}{\erfinv\left (\frac{1-\alpha}{1-\beta_0}\right )}\right )+(1-\beta)\log\left (\frac{1-\beta}{1-\beta_w}\right ) +\beta\log\left (\frac{(\alpha-\beta)(1-\beta_0)}{(\alpha-\beta_0)(1-\beta_w)}\right ).
\label{eq:thmfinalldpl13}
\end{equation}
Then if $\alpha>\alpha_w$
\begin{equation}
I_{err}(\alpha,\beta)\triangleq\lim_{n\rightarrow\infty}\frac{\log{P_{err}}}{n}=I_{ldp}(\alpha,\beta).\label{eq:thmfinalldpl14}
\end{equation}
Moreover, if $\alpha<\alpha_w$
\begin{equation}
I_{cor}(\alpha,\beta)\triangleq\lim_{n\rightarrow\infty}\frac{\log{P_{cor}}}{n}=I_{ldp}(\alpha,\beta).\label{eq:thmfinalldpl15}
\end{equation}\label{thm:finalldpl1}
\end{theorem}
\begin{proof} Follows from the above discussion.
\end{proof}
Before we present the results one can obtain based on the above theorem we will establish a few additional properties of function $\frac{\alpha-\beta}{\alpha-\beta_0}\xi_{\alpha}(\beta_0)$ to ensure that everything is on a right mathematical track.

\subsubsection{Properties of $\frac{\alpha-\beta}{\alpha-\beta_0}\xi_{\alpha}(\beta_0)$}
\label{sec:proppsixi}

In this subsection we will try to complement some of the key properties of functions $\xi_\alpha(\beta)$ and $\psi_\beta(\alpha_w)$ from Theorem \ref{thm:thmweakthr} that we introduced in Section \ref{sec:propxi}. We again emphasize that these may be viewed as fairly straightforward but for the mathematical exactness and completeness we find it convenient to have them neatly presented.

As was the case with functions $\xi_\alpha(\beta)$ and $\psi_\beta(\alpha_w)$ in Section \ref{sec:propxi} the key observation regarding $\frac{\alpha-\beta}{\alpha-\beta_0}\xi_{\alpha}(\beta_0)$ is that for any fixed $(\alpha,\beta)\in (0,1)\times(0,\alpha)$ there is a unique $\beta_0$ such that $\frac{\alpha-\beta}{\alpha-\beta_0}\xi_{\alpha}(\beta_0)=1$. This essentially ensures that (\ref{eq:thmfinalldpl12}) is an unambiguous LDP characterization. To confirm that this is indeed true we proceed in a fashion similar to the one showcased in Section \ref{sec:propxi}. Setting as in (\ref{eq:proppsixi2})
\begin{equation}\label{eq:propxi02}
  q=\erfinv\left (\frac{1-\alpha}{1-\beta_0}\right ),
\end{equation}
we have
\begin{equation}\label{eq:propxi03}
  \frac{\alpha-\beta}{\alpha-\beta_0}\xi_{\alpha}(\beta_0)=1 \Leftrightarrow \frac{\alpha-\beta}{\alpha}\frac{1}{\erfc(q)}\frac{\sqrt{\frac{1}{\pi}}e^{-q^2}}{q}=1
  \Leftrightarrow \frac{\sqrt{\frac{1}{\pi}}e^{-q^2}}{q}-\erfc(q)c_{\alpha,\beta}=0, c_{\alpha,\beta}> 1.
\end{equation}
Utilizing (\ref{eq:hdg10a}) we then have for $q> \frac{1}{\sqrt{2c_{\alpha,\beta}(c_{\alpha,\beta}-1)}}$
\begin{equation}\label{eq:propxi04}
  \frac{\sqrt{\frac{1}{\pi}}e^{-q^2}}{q}-\erfc(q)c_{\alpha,\beta}\leq
\frac{\sqrt{\frac{1}{\pi}}e^{-q^2}}{q}-c_{\alpha,\beta}\frac{2}{\sqrt{\pi}}\frac{e^{-q^2}}{q+\sqrt{q^2+2}}<  0.
\end{equation}
Moreover, for $q\leq \frac{1}{\sqrt{2c_{\alpha,\beta}(c_{\alpha,\beta}-1)}}$ (actually even for $q\leq \frac{1}{\sqrt{2(c_{\alpha,\beta}-1)}}$)
\begin{equation}\label{eq:propxi05}
 \frac{d\left ( \frac{\sqrt{\frac{1}{\pi}}e^{-q^2}}{q}-\erfc(q)c_{\alpha,\beta}\right )}{dq}=
 -2\sqrt{\frac{1}{\pi}}e^{-q^2}-\frac{\sqrt{\frac{1}{\pi}}e^{-q^2}}{q^2}+\frac{2}{\sqrt{\pi}}e^{-q^2}c_{\alpha,\beta}
 \leq -2\sqrt{\frac{1}{\pi}}e^{-q^2}(c_{\alpha,\beta}-1)^2<0,
\end{equation}
which means that $\left ( \frac{\sqrt{\frac{1}{\pi}}e^{-q^2}}{q}-\erfc(q)c_{\alpha,\beta}\right )$ is decreasing for $q\leq \frac{1}{\sqrt{2c_{\alpha,\beta}(c_{\alpha,\beta}-1)}}$. Finally, one easily also has
\begin{equation}\label{eq:propxi06}
\lim_{q\rightarrow\0} \left ( \frac{\sqrt{\frac{1}{\pi}}e^{-q^2}}{q}-\erfc(q)c_{\alpha,\beta}\right )=\infty.
\end{equation}
A combination of (\ref{eq:propxi04}), (\ref{eq:propxi05}), and (\ref{eq:propxi06}) then implies that $\left ( \frac{\sqrt{\frac{1}{\pi}}e^{-q^2}}{q}-\erfc(q)c_{\alpha,\beta}\right )$ has a unique solution (moreover, it is in the interval $(0, \frac{1}{\sqrt{2c_{\alpha,\beta}(c_{\alpha,\beta}-1)}})$). This then implies that $\frac{\alpha-\beta}{\alpha-\beta_0}\xi_{\alpha}(\beta_0)=1$ also has a unique solution, i.e. that for any fixed $(\alpha,\beta)\in (0,1)\times (0,\alpha)$ there is a unique $\beta_0$ such that $\frac{\alpha-\beta}{\alpha-\beta_0}\xi_\alpha(\beta_0)=1$, which as mentioned above essentially means that (\ref{eq:thmfinalldpl12}) is an unambiguous LDP characterization. For the completeness, in Figure \ref{fig:propxi0} we present a few numerical results related to the behavior of $\left ( \frac{\sqrt{\frac{1}{\pi}}e^{-q^2}}{q}-\erfc(q)c_{\alpha,\beta}\right )$ (and ultimately of $\left (\frac{\alpha-\beta}{\alpha-\beta_0}\xi_{\alpha}(\beta_0)-1\right )$) that indeed confirm the above calculations.

\begin{figure}[htb]
\centering
\centerline{\epsfig{figure=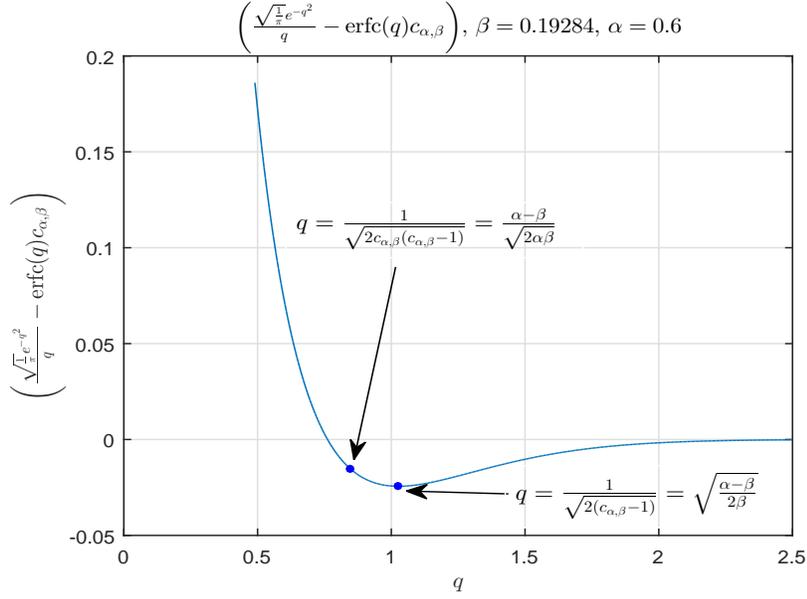,width=11.5cm,height=8cm}}
\caption{Uniqueness of the solution of $\frac{\alpha-\beta}{\alpha-\beta_0}\xi_{\alpha}(\beta_0)=1$ is implied by the properties of $\left ( \frac{\sqrt{\frac{1}{\pi}}e^{-q^2}}{q}-\erfc(q)c_{\alpha,\beta}\right )$}
\label{fig:propxi0}
\end{figure}


\subsection{Theoretical and numerical LDP results}
\label{sec:thnumresuts}

In this section we finally give a little bit of a flavor to what is actually proven in Theorem \ref{thm:finalldpl1}. In Figure \ref{fig:l1regldpIerrub} we show the theoretical LDP rate function curve that one can obtain based on Theorem \ref{thm:finalldpl1}. We complement this figure by Table \ref{tab:Ildptab1} where we show the numerical values for all quantities of interest in Theorems \ref{thm:ldp3} and \ref{thm:finalldpl1}. Finally, in Figure \ref{fig:weakl1LDPthrsim} and Table \ref{tab:Ildptab2} we show how the simulated values compare to the theoretical ones and observe that even for fairly small dimensions (of order $100$) one already approaches the theoretical curves derived of course for a fully (infinite dimensional) asymptotic regime.


\begin{figure}[htb]
\begin{minipage}[b]{.5\linewidth}
\centering
\centerline{\epsfig{figure=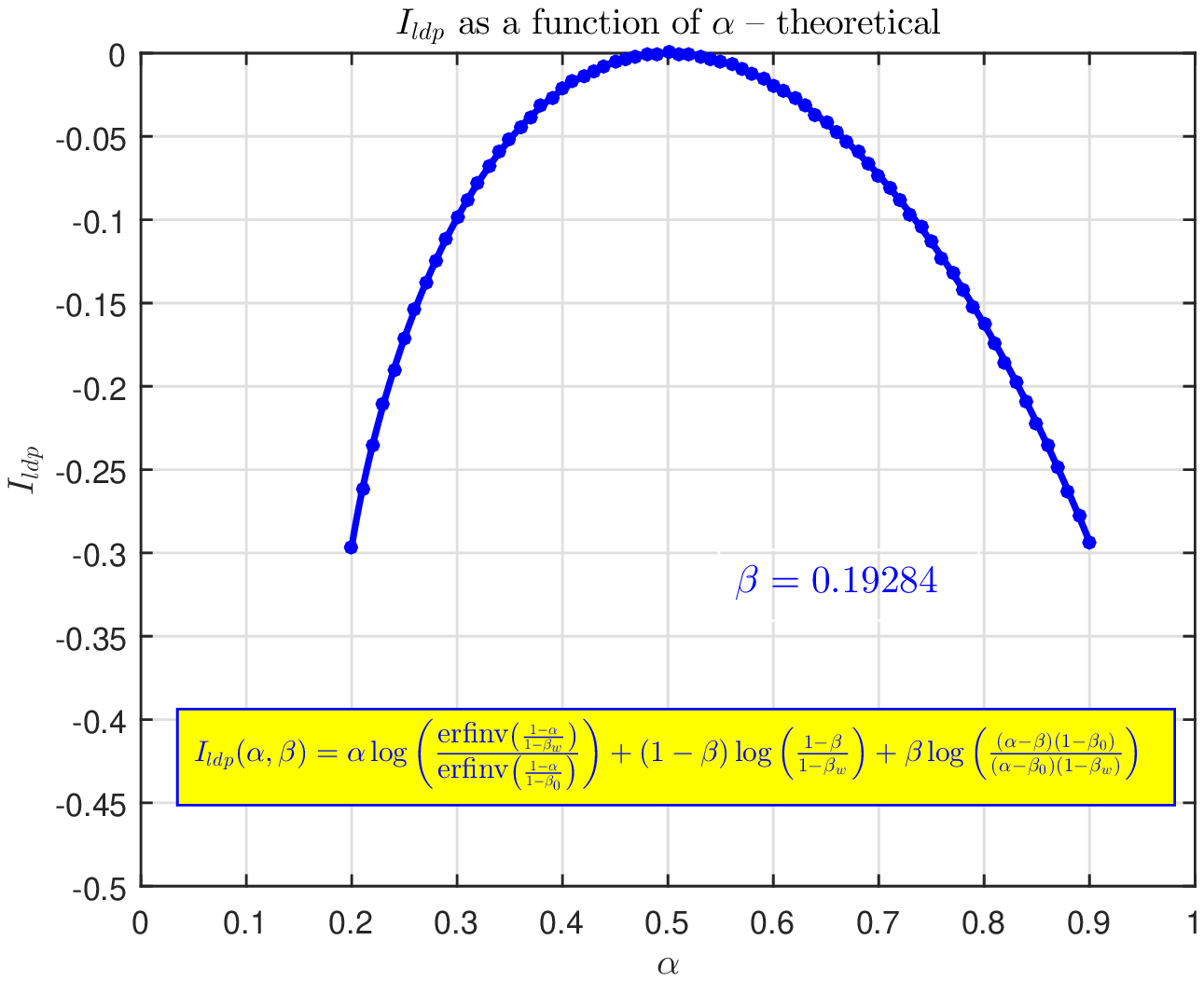,width=9cm,height=7cm}}
\end{minipage}
\begin{minipage}[b]{.5\linewidth}
\centering
\centerline{\epsfig{figure=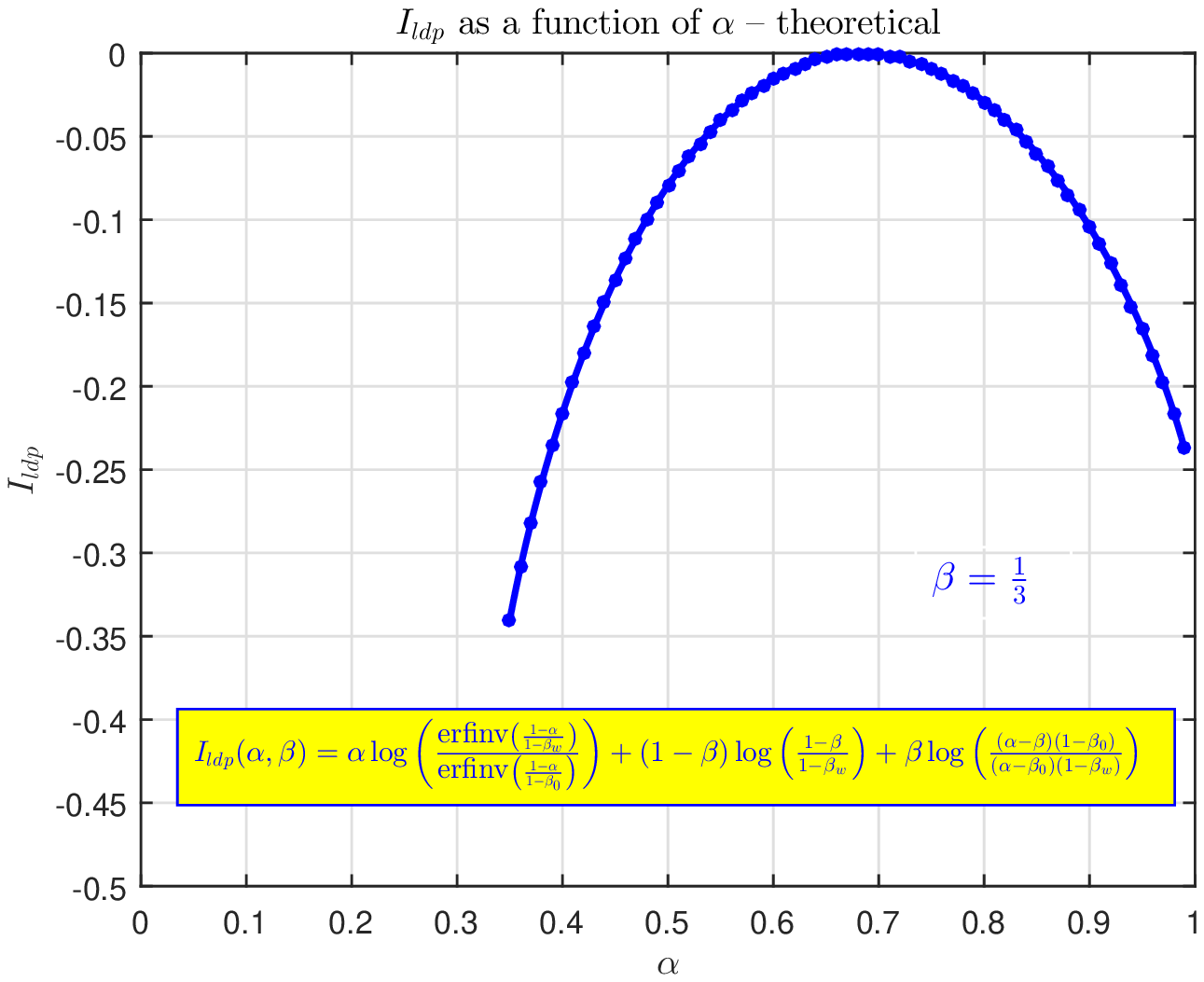,width=9cm,height=7cm}}
\end{minipage}
\caption{$I_{ldp}$ as a function of $\alpha$; left -- $\beta=0.19284$; right -- $\beta=\frac{1}{3}$}
\label{fig:l1regldpIerrub}
\end{figure}

\begin{table}[h]
\caption{A collection of values for $\beta_w$, $\beta_0$, $\nu$, $A_0$, $c_3$, $\gamma$, and $I_{ldp}$ in Theorem \ref{thm:ldp3}; $\beta=0.19284$}\vspace{.1in}
\hspace{-0in}\centering
\begin{tabular}{||c||c|c|c|c|c|c|c||}\hline\hline
$\alpha$ & $ 0.3500 $ & $ 0.4000 $ & $ 0.4500 $ & $ 0.5000 $ & $ 0.5500 $ & $ 0.6000 $ & $ 0.6500 $ \\ \hline\hline
$\beta_w$& $ 0.1099 $ & $ 0.1349 $ & $ 0.1625 $ & $ 0.1928 $ & $ 0.2262 $ & $ 0.2630 $ & $ 0.3038 $ \\ \hline
$\beta_0$& $ -0.6420 $ & $ -0.2424 $ & $ 0.0121 $ & $ 0.1929 $ & $ 0.3313 $ & $ 0.4433 $ & $ 0.5380 $ \\ \hline\hline
$\nu$    & $ 1.1036 $ & $ 1.0228 $ & $ 0.9477 $ & $ 0.8769 $ & $ 0.8091 $ & $ 0.7434 $ & $ 0.6788 $ \\ \hline
$A_0$    & $ 2.1287 $ & $ 1.5786 $ & $ 1.2361 $ & $ 1.0000 $ & $ 0.8256 $ & $ 0.6902 $ & $ 0.5807 $ \\ \hline
$c_3$    & $ -0.9815 $ & $ -0.5978 $ & $ -0.2865 $ & $ 0.0000 $ & $ 0.2859 $ & $ 0.5876 $ & $ 0.9203 $ \\ \hline
$\gamma$ & $ 0.1390 $ & $ 0.2003 $ & $ 0.2713 $ & $ 0.3536 $ & $ 0.4491 $ & $ 0.5611 $ & $ 0.6942 $ \\ \hline\hline
$I_{ldp}$& $ \mathbf{-0.0517} $ & $ \mathbf{-0.0217} $ & $ \mathbf{-0.0052} $ & $ \mathbf{0.0000} $ & $ \mathbf{-0.0049} $ & $ \mathbf{-0.0190} $ & $ \mathbf{-0.0418} $ \\ \hline\hline
\end{tabular}
\label{tab:Ildptab1}
\end{table}

\begin{figure}[htb]
\centering
\centerline{\epsfig{figure=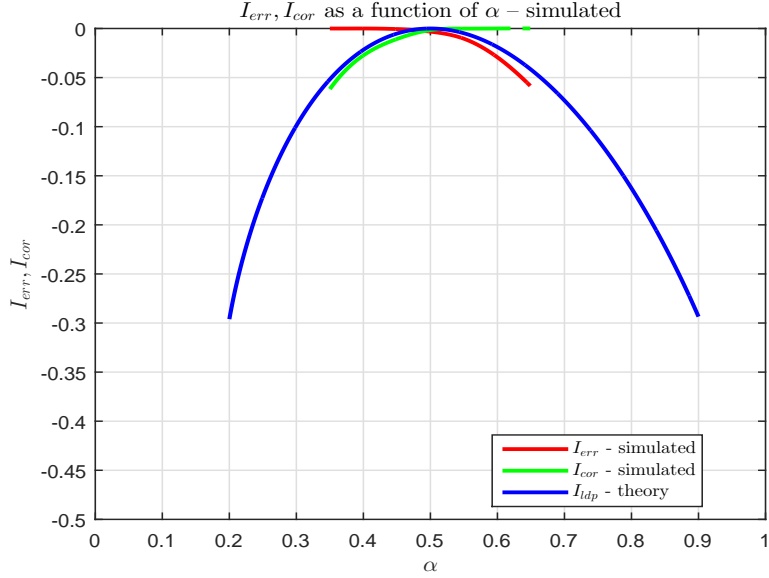,width=11.5cm,height=8cm}}
\caption{$\ell_1$'s weak LDP rate function -- theory and simulation; $\beta=0.19284$}
\label{fig:weakl1LDPthrsim}
\end{figure}

\begin{table}[h]
\caption{$I_{err}$, $I_{err}$ -- simulated; $I_{ldp}$ calculated for $\beta=0.19284$}\vspace{.1in}
\hspace{-0in}\centering
\begin{tabular}{||c||c|c|c|c|c|c|c||}\hline\hline
$\alpha$ & $ 0.35 $ & $ 0.40 $ & $ 0.45 $ & $ 0.50 $ & $ 0.55 $ & $ 0.60 $ & $0.65 $\\ \hline\hline
$k$      & $ 19 $ & $ 38 $ & $ 38 $ & $ 57 $ & $ 57 $ & $ 38 $ & $ 26 $\\ \hline
$m$      & $ 35 $ & $ 80 $ & $ 90 $ & $ 150 $ & $ 165 $ & $ 120 $ & $ 89 $\\ \hline
$n$      & $ 100 $ & $ 200 $ & $ 200 $ & $ 300 $ & $ 300 $ & $ 200 $ & $ 137 $\\ \hline\hline
$I_{err} $ -- simulated & $ -0.0000 $ & $ -0.0000 $ & $ -0.0006 $ & \red{$ \mathbf{-0.0032} $} & \red{$ \mathbf{-0.0104} $} & \red{$ \mathbf{-0.0293} $} & \red{$ \mathbf{-0.0588} $}\\ \hline\hline
$I_{cor}$ -- simulated & \gr{$ \mathbf{-0.0617} $} & \gr{$ \mathbf{-0.0274} $} & \gr{$ \mathbf{-0.0107} $} & \gr{$ \mathbf{-0.0016} $} & $ -0.0001 $ & $ -0.0000 $ & $ -0.0000 $ \\ \hline\hline
$I_{ldp}$ -- theory & \bl{$ \mathbf{-0.0517} $} & \bl{$ \mathbf{-0.0217} $} & \bl{$ \mathbf{-0.0052} $} & \bl{$ \mathbf{0.0000} $} & \bl{$ \mathbf{-0.0049} $} & \bl{$ \mathbf{-0.0190} $} & \bl{$ \mathbf{-0.0418} $} \\ \hline\hline
\end{tabular}
\label{tab:Ildptab2}
\end{table}

\subsection{High-dimensional geometry approach of \cite{DonohoPol,DonohoUnsigned}}
\label{sec:donhdg}

In this section we look at an alternative high-dimensional approach to characterize the performance of $\ell_1$. In \cite{DonohoPol,DonohoUnsigned} Donoho connected the success of the $\ell_1$ optimization when used for finding sparsest $\x$ in (\ref{eq:system}) to the study of neighborly polytopes. He showed that $\ell_1$ succeeds if some of the faces of the crosspolytope are preserved after being projected by the system matrix $A$. He then relied on the vast theory developed in high-dimensional geometry arena to deal with the projected polytopes and their neighborliness. As the underlying medium was random \cite{DonohoPol,DonohoUnsigned} then proceeded by considering an asymptotic regime and analyzing the underlying phase transitions. In \cite{StojnicCSetam09,StojnicUpper10} we designed a purely probabilistic approach and in \cite{StojnicEquiv10} we showed that the $\ell_1$'s phase transitions from Theorem \ref{thm:thmweakthr} and Donoho's result from \cite{DonohoPol,DonohoUnsigned} are in a perfect mathematical agreement.

The analysis presented in \cite{DonohoPol,DonohoUnsigned} can also be used for LDP characterizations. We will assume that we are given a pair $(\alpha,\beta)$ and will immediately write the results for both, upper and lower LDP regimes, i.e. for $\alpha>\alpha_w$ and for $\alpha<\alpha_w$ (where $\alpha_w$ is such that $\psi_\beta(\alpha_w)=\xi_{\alpha_w}(\beta)=1$). When put in the LDP frame of Section \ref{sec:hdg} and ultimately \cite{StojnicEquiv10} results of \cite{DonohoPol,DonohoUnsigned} give
\begin{eqnarray}\label{eq:donhdg1}
  \Psi_{net}^{(D)}(\alpha,\beta) & =&I_{err}(\alpha,\beta)\triangleq\lim_{n\rightarrow\infty}\frac{\log{P_{err}}}{n}=\psicom^{(D)}-\psiint^{(D)}-\psiext^{(D)},\alpha>\alpha_w\nonumber \\
  \Psi_{net}^{(D)}(\alpha,\beta) & =&I_{cor}(\alpha,\beta)\triangleq\lim_{n\rightarrow\infty}\frac{\log{P_{err}}}{n}=\psicom^{(D)}-\psiint^{(D)}-\psiext^{(D)},\alpha<\alpha_w,
\end{eqnarray}
where
\begin{eqnarray}
\psicom^{(D)} & = & (\alpha-\beta) \log(2)-(\alpha-\beta)\log\left (\frac{\alpha-\beta}{1-\beta}\right )-(1-\alpha)\log\left (\frac{1-\alpha}{1-\beta}\right )\nonumber \\
\psiext^{(D)} & = & \min_{y\geq 0} (\alpha y^2 -(1-\alpha)\log(\erf(y))), \label{eq:donhdg2}
\end{eqnarray}
and
\begin{equation}
\psiint^{(D)}=(\alpha-\beta)\left (-\frac{1}{2}\frac{\beta}{\alpha-\beta}s_{\alpha,\beta}^2-\frac{1}{2}\log\left (\frac{2}{\pi}\right )+\log\left (\frac{\alpha s_{\alpha,\beta}}{\alpha-\beta}\right )\right )
+(\alpha-\beta)\log(2),\label{eq:intang1}
\end{equation}
where $s_{\alpha,\beta}\geq 0$ is the solution of
\begin{equation}
\frac{1}{2}\erfc\left (\frac{s}{\sqrt{2}}\right )=\frac{\alpha-\beta}{\alpha}\frac{e^{-\frac{s^2}{2}}}{s\sqrt{2\pi}}.\label{eq:intang3}
\end{equation}
Now if we can show that $\psiint^{(D)}=-\psiint$ then $\psinet^{(D)}=\psinet$ and the approach of \cite{DonohoPol,DonohoUnsigned} indeed matches the approach of Section \ref{sec:hdg}. To that end, we follow into the footsteps of \cite{StojnicEquiv10} and set
\begin{equation}
s_{\alpha,\beta}=\sqrt{2}\erfinv\left (\frac{1-\alpha}{1-\beta_0}\right ),\label{eq:sgamma}
\end{equation}
where as earlier $\beta_0$ is such that $\frac{\alpha-\beta}{\alpha-\beta_0}\xi_{\alpha}(\beta_0)=1$. For such a $s_{\alpha,\beta}$ (\ref{eq:intang3}) becomes
\begin{equation}
\frac{1}{2}\left (\frac{\alpha-\beta_0}{1-\beta_0}\right )=\frac{\alpha-\beta}{\alpha}\frac{e^{-\left (\erfinv\left (\frac{1-\alpha}{1-\beta_0}\right )\right )^2}}{\sqrt{2\pi}\sqrt{2}\erfinv\left (\frac{1-\alpha}{1-\beta_0}\right )},\label{eq:intang3a}
\end{equation}
which is true because of $\frac{\alpha-\beta}{\alpha-\beta_0}\xi_{\alpha}(\beta_0)=1$. Now replacing $s_{\alpha,\beta}$ from (\ref{eq:sgamma}) in (\ref{eq:intang1}) we obtain
\begin{eqnarray}
\psiint^{(D)}&=&(\alpha-\beta)\left (-\frac{1}{2}\frac{\beta}{\alpha-\beta}s_{\alpha,\beta}^2-\frac{1}{2}\log\left (\frac{2}{\pi}\right )+\log\left (\frac{\alpha s_{\alpha,\beta}}{\alpha-\beta}\right )\right )
+(\alpha-\beta)\log(2)\nonumber \\
&=&(\alpha-\beta)\left (-\frac{\beta}{\alpha-\beta}\left (\erfinv\left (\frac{1-\alpha}{1-\beta_0}\right )\right )^2-\frac{1}{2}\log\left (\frac{2}{\pi}\right )+\log\left (\frac{\sqrt{\frac{2}{\pi}}(1-\beta_0) e^{-\left (\erfinv\left (\frac{1-\alpha}{1-\beta_0}\right )\right )^2}}{\alpha-\beta_0}\right )+\log(2)\right )\nonumber \\
& = & -\alpha \left (\erfinv\left (\frac{1-\alpha}{1-\beta_0}\right )
\right )^2 -(\alpha-\beta)\log\left (\frac{\alpha-\beta_0}{1-\beta_0}\right )+(\alpha-\beta) \log(2).
\label{eq:intang1a}
\end{eqnarray}
Connecting (\ref{eq:intang1a}) and the second equality in (\ref{eq:hdg12}) then confirms that indeed $\psiint^{(D)}=-\psiint$ and finally $\psinet^{(D)}=\psinet$.

\section{Phase transitions -- nonnegative vectors}
\label{sec:phasetransnonn}

In this section we look at a class of unknown vectors structured a bit more beyond the standard sparsity. We start with the same systems as in (\ref{eq:system}) but additionally insist that $\y$ was obtained through (\ref{eq:defy}) with $\tilde{x}$ being not only $k$-sparse but also with components that are not negative. We will call such vectors nonnegative. To solve (\ref{eq:system}) knowing that there is an $\x$ that is nonnegative one can employ all the standard methods that can be employed for general $\x$. The only difference would be that one now insists that components of $\x$ are not negative. Here we will focus on such a modification of (\ref{eq:l1}) (which we may often refer to as the \emph{nonnegative} $\ell_1$)
\begin{eqnarray}
\mbox{min} & & \|\x\|_{1}\nonumber \\
\mbox{subject to} & & A\x=\y\nonumber \\
&& \x\geq 0. \label{eq:l1nonn}
\end{eqnarray}
Of course, this small modification does not take away any of the features that make (\ref{eq:l1}) relevant. These are in first place its polynomial complexity and the fact that it is a linear program (a bit more complex though than the one in (\ref{eq:l1}) but still a linear program nonetheless). Its implementation is as universal as the implementation of (\ref{eq:l1}) and knowing $A$ and $\y$ is perfectly sufficient to use it (as earlier, full-rank of matrix $A$ will typically be assumed throughout, though).

As there is not much difference when compared to (\ref{eq:l1}) one would then expect that (\ref{eq:l1nonn}) exhibits excellent performance characteristics as (\ref{eq:l1}) does. Of course, (\ref{eq:l1nonn}) is well known in the theory of linear systems with sparse solutions and its performance has been characterized in various ways. The standard by now are the results of \cite{DonohoPol,DonohoUnsigned,DonohoSigned,DT,StojnicCSetam09,StojnicUpper10} that provided performance characterizations analogous to the ones they did for (\ref{eq:l1}). Namely, \cite{DonohoSigned,DT,StojnicCSetam09,StojnicUpper10} uncovered that (\ref{eq:l1nonn}) also exhibits the so-called phase-transition (PT) phenomenon when utilized in statistical contexts. Moreover, both sets of results, \cite{DT,DonohoSigned} and \cite{StojnicCSetam09,StojnicUpper10}, in addition to uncovering the existence of the phase transition phenomenon precisely characterized the PT curves. Here, we will make a substantial progress in studying further the phase transitions and will fully characterize its LDP (of course similarly to what we did in earlier sections for (\ref{eq:l1})). As in the case of (\ref{eq:l1}), we will do so through two different approaches, the novel, more modern purely probabilistic one and through another one that is based on high-dimensional geometry.

As was the case when we studied general vectors $\x$ in earlier sections, we will split the presentation into several parts. Moreover, to facilitate following we will try as much as possible to parallel the presentation with what we did for general $\x$. Also, as many things will conceptually be similar we will try to avoid repeating many of the calculations and instead will just state the final results. The exceptions will be when we believe that adapting some calculations is not so straightforward; in these scenarios we will then sketch the key arguments again.

We will start things off by recalling on what is known about the phase transitions of (\ref{eq:l1nonn}). We will then connect them to the LDP and then study the LDPs in a great detail (following into the footsteps of what we did in earlier sections for the general $\x$). We will do so first through a purely probabilistic approach and then through considerations of some high-dimensional geometry aspects. As was the case when we studied general vectors $\x$ in earlier sections, here the main emphasis will again be on the elegance of the final results. Before proceeding with the phase transitions of (\ref{eq:l1nonn}) we mention that all the definitions regarding strong and weak PT introduced in Section \ref{sec:phasetrans} remain in place here with obvious modifications to incorporate that now $\x$ is a priori known to be nonnegative. We below recall on a theorem that essentially summarizes the results obtained in \cite{StojnicCSetam09,StojnicUpper10} and effectively establishes for any $0<\alpha\leq 1$ the exact value of $\beta_w$ for which (\ref{eq:l1nonn}) finds the a priori known to be nonnegative $k$-sparse $\x$ from (\ref{eq:system}).


\begin{theorem}(\cite{StojnicCSetam09,StojnicUpper10} Exact \emph{nonnegative} $\ell_1$'s weak threshold/PT)
Let $A$ be an $m\times n$ matrix in (\ref{eq:system})
with i.i.d. standard normal components. Let
the unknown $\x$ in (\ref{eq:system}) be $k$-sparse. Further, let all elements of $\x$ be nonnegative let that be a priori known.
Let $k,m,n$ be large
and let $\alpha_w=\frac{m}{n}$ and $\beta_w=\frac{k}{n}$ be constants
independent of $m$ and $n$. Let $\erfinv$ be the inverse of the standard error function associated with zero-mean unit variance Gaussian random variable.  Further, let $\alpha_w$ and $\beta_w$ satisfy the following \textbf{fundamental characterization of the \emph{nonnegative} $\ell_1$'s PT}

\begin{center}
\shadowbox{$
\xi^+_{\alpha_{w}}(\beta_w)\triangleq\psi^+_{\beta_w}(\alpha_{w})\triangleq
(1-\beta_w)\frac{\sqrt{\frac{1}{2\pi}}e^{-\lp\erfinv\lp 2\frac{1-\alpha_w}{1-\beta_w}-1\rp\rp^2}}{\alpha_w\sqrt{2}\erfinv \lp2\frac{1-\alpha_w}{1-\beta_w}-1\rp}=1.
$}
-\vspace{-.5in}\begin{equation}
\label{eq:thmweaktheta2nonn}
\end{equation}
\end{center}

Then:
\begin{enumerate}
\item If $\alpha>\alpha_w$ then with overwhelming probability the solution of (\ref{eq:l1nonn}) is the a priori known to be nonnegative $k$-sparse $\x$ from (\ref{eq:system}).
\item If $\alpha<\alpha_w$ then with overwhelming probability there will be an a priori known to be nonnegative $k$-sparse $\x$ (from a set of such $\x$'s with fixed locations of nonzero components) that satisfies (\ref{eq:system}) and is \textbf{not} the solution of (\ref{eq:l1nonn}).
    \end{enumerate}
\label{thm:thmweakthrnonn}
\end{theorem}
\begin{proof}
The first part was established in \cite{StojnicCSetam09} and the second one was established in \cite{StojnicUpper10}. An alternative way of establishing the same set of results was also presented in \cite{StojnicEquiv10}. Of course, similar results were obtained in \cite{DT,DonohoSigned}.
\end{proof}

\subsection{Properties of $\xi^+_\alpha(\beta)$ and $\psi^+_\beta(\alpha)$}
\label{sec:propxinonn}

In this subsection we will briefly sketch that for functions $\xi^+_\alpha(\beta)$ and $\psi^+_\beta(\alpha)$ from Theorem \ref{thm:thmweakthrnonn} one has similar properties as for functions $\xi_\alpha(\beta)$ and $\psi_\beta(\alpha)$ from Theorem \ref{thm:thmweakthr}. To ensure parallelism with what we presented in earlier sections when we considered general vectors $\x$, we of course closely follow what was done in Section \ref{sec:propxi}.

\subsubsection{$\xi^+_\alpha(\beta)$}
\label{sec:propxi1nonn}

As was the case for $\xi_\alpha(\beta)$, the key observation regarding $\xi^+_\alpha(\beta)$ is that for any fixed $\alpha\in (0,1)$ there is a unique $\beta$ such that $\xi^+_\alpha(\beta)=1$ which ensures that (\ref{eq:thmweaktheta2nonn}) is an unambiguous PT characterization. Clearly, we below consider only $\beta\in(\max(2\alpha-1,0),\alpha)$ since if $\beta<\max(2\alpha-1,0)$ one easily from (\ref{eq:thmweaktheta2nonn}) has $\xi^+_\alpha(\beta)<0$.

\underline{\emph{1) For any fixed $\alpha\in (0,1)$, $\xi^+_\alpha(\beta)-1$ is a decreasing function of $\beta$ on interval $(\max(2\alpha-1,0),\alpha)$.
}}

To see this we proceed in the following straightforward way
\begin{eqnarray}\label{eq:propxi1nonn}
  \frac{d(\xi^+_\alpha(\beta)-1)}{d\beta} & = & \frac{d((1-\beta)\frac{\sqrt{\frac{1}{2\pi}}e^{-\lp\erfinv\lp2\frac{1-\alpha}{1-\beta}-1\rp\rp^2}}{\alpha\sqrt{2}\erfinv \lp2\frac{1-\alpha}{1-\beta}-1\rp}-1)}{d\beta}\nonumber\\
  & = & \sqrt{\frac{1}{2\pi}}\frac{-\frac{\sqrt{\pi} (1-\alpha)}{(1-\beta) \erfinv(2(1-\alpha)/(1-\beta)-1)^2}-\frac{ e^{-\lp\erfinv\lp2\frac{1-\alpha}{1-\beta}-1\rp\rp^2}}{\erfinv(2(1-\alpha)/(1-\beta)-1)}-\frac{2\sqrt{\pi} (1-\alpha)}{1-\beta}}{2 \sqrt{2} \alpha}\nonumber \\
  &< & 0.
\end{eqnarray}

\underline{\emph{2) For any fixed $\alpha\in (0,1)$, $\lim_{\beta\rightarrow \alpha}\xi^+_\alpha(\beta)-1=-1$.}}

This of course follows easily since
\begin{equation}\label{eq:propxi2nonn}
  \lim_{\beta\rightarrow \alpha}\lp \erfinv\lp2\frac{1-\alpha}{1-\beta}-1\rp\rp=\infty
\end{equation}

\underline{\emph{3) For any fixed $\alpha\in (0,1)$, $\xi^+_\alpha(\max(2\alpha-1,0))-1>0$.}}

We will show $\xi^+_\alpha(\max(2\alpha-1,0))>1$ which of course implies $\xi^+_\alpha(\max(2\alpha-1,0))-1>0$. If $\max(2\alpha-1,0)=0$ we then have through a combination of (\ref{eq:propxi31}), (\ref{eq:propxi32}), (\ref{eq:propxi33}), (\ref{eq:propxi34}, and (\ref{eq:propxi35}) $\xi^+_\alpha(\max(2\alpha-1,0))>1$. On the other hand, if $\max(2\alpha-1,0)=2\alpha-1$ we can write
\begin{equation}\label{eq:propxi31nonn}
\lim_{\beta\rightarrow (2\alpha-1)_+}\xi^+_\alpha(\beta)=\lim_{\beta\rightarrow (2\alpha-1)_+}\lp(1-\beta)\frac{\sqrt{\frac{1}{2\pi}}e^{-\lp\erfinv\lp 2\frac{1-\alpha}{1-\beta}-1\rp\rp^2}}{\alpha\sqrt{2}\erfinv \lp 2\frac{1-\alpha}{1-\beta}-1\rp}-1\rp=\infty.
\end{equation}
A combination of the above three observations ensures that for any fixed $\alpha\in (0,1)$ there is a unique $\beta\in(\max(2\alpha-1,0),\alpha)$ such that $\xi^+_\alpha(\beta)=1$, which as mentioned above essentially means that (\ref{eq:thmweaktheta2nonn}) is an unambiguous PT characterization. For the completeness, in Figure \ref{fig:propxinonn} we present a few numerical results related to the behavior of $\xi^+_\alpha(\beta)$ that indeed confirm the above calculations.
\begin{figure}[htb]
\begin{minipage}[b]{.5\linewidth}
\centering
\centerline{\epsfig{figure=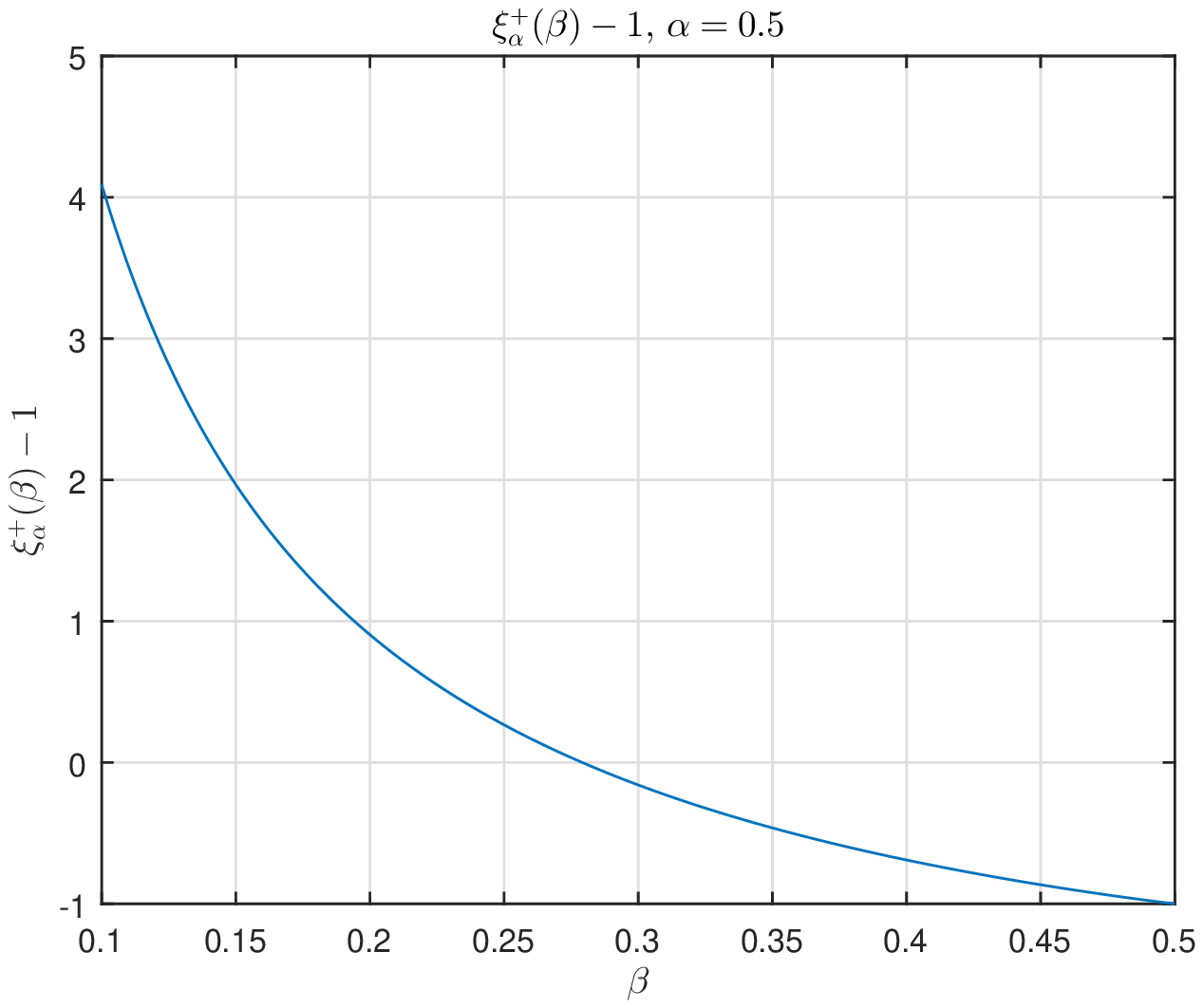,width=9cm,height=7cm}}
\end{minipage}
\begin{minipage}[b]{.5\linewidth}
\centering
\centerline{\epsfig{figure=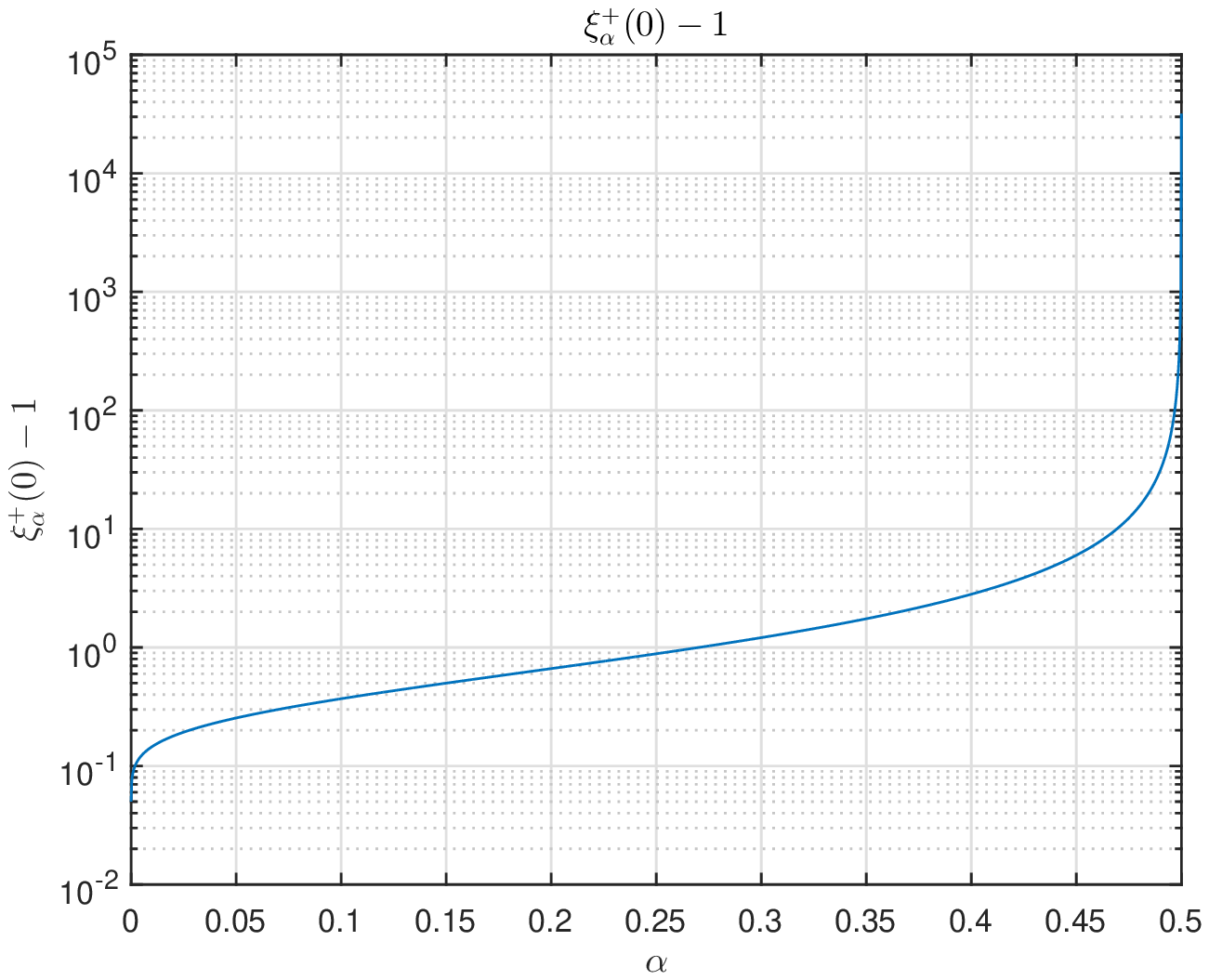,width=9cm,height=7cm}}
\end{minipage}
\caption{Properties of $\xi^+_\alpha(\beta)$: $\frac{d(\xi^+_\alpha(\beta)-1)}{d\beta}$ as a function of $\beta$ ($\alpha=0.5$) -- left; $\xi^+_\alpha(0)$ as a function of $\alpha$ -- right}
\label{fig:propxinonn}
\end{figure}

\subsubsection{$\psi^+_\beta(\alpha)$}
\label{sec:proppsixi1nonn}

We now look at $\psi^+_\beta(\alpha)$. We will show that for any fixed $\beta\in (0,1)$ there is a unique $\alpha$ such that $\psi^+_\beta(\alpha)=1$. This would ensure that the nonnegative $\ell_1$'s fundamental PT from the above theorem is also unambiguous when viewed as a function of $\alpha$. To confirm that this is indeed true we proceed by paralleling what was done in Section \ref{sec:proppsixi1}. Before doing that we also quickly observe that $\alpha\in(\beta,\frac{1+\beta}{2})$ is the interval of interest since if $\alpha>\frac{1+\beta}{2}$ (\ref{eq:thmweaktheta2nonn}) can not be satisfied.

\underline{\emph{1) For any fixed $\beta\in (0,1)$, $\psi^+_\beta(\alpha)-1$ is an increasing function of $\alpha$ on interval $(\beta,\frac{1+\beta}{2})$.
}}

To see this we proceed by computing the derivative
\begin{eqnarray}\label{eq:proppsixi1nonn}
  \frac{d(\psi^+_\beta(\alpha)-1)}{d\alpha} & = & \frac{d\lp(1-\beta)\frac{\sqrt{\frac{1}{2\pi}}e^{-\lp\erfinv\lp2\frac{1-\alpha}{1-\beta}-1\rp\rp^2}}{\alpha\sqrt{2}\erfinv \lp 2\frac{1-\alpha}{1-\beta}-1\rp}-1\rp}{d\alpha}\nonumber\\
   & = & \frac{2 (\beta-1) e^{-\left (\erfinv\left (2\frac{1-\alpha}{1-\beta}-1\right )\right )^2} \left (\erfinv\left (2\frac{1-\alpha}{1-\beta}-1\right )\right )
  +\sqrt{\pi} \alpha \left (2 \left (\erfinv\left (2\frac{1-\alpha}{1-\beta}-1\right )\right )^2+1\right )}{2 \alpha^2 \sqrt{\pi} \left (\erfinv\left (2\frac{1-\alpha}{1-\beta}-1\right )\right )^2}.\nonumber \\
\end{eqnarray}
Let
\begin{equation}\label{eq:proppsixi2nonn}
  q^+=\erfinv\left (2\frac{1-\alpha}{1-\beta}-1\right ).
\end{equation}
Then
\begin{eqnarray}\label{eq:proppsixi3nonn}
  \frac{d(\psi^+_\beta(\alpha)-1)}{d\alpha}
   & = & \frac{2 (\beta-1) e^{-\left (\erfinv\left (2\frac{1-\alpha}{1-\beta}-1\right )\right )^2} \left (\erfinv\left (2\frac{1-\alpha}{1-\beta}-1\right )\right )
  +\sqrt{\pi} \alpha \left (2 \left (\erfinv\left (2\frac{1-\alpha}{1-\beta}-1\right )\right )^2+1\right )}{2 \alpha^2 \sqrt{\pi} \left (\erfinv\left (2\frac{1-\alpha}{1-\beta}-1\right )\right )^2}\nonumber \\
  & = &  \frac{2 \frac{\alpha-1}{\erf(q^+)} e^{-(q^+)^2} q^+
  +\sqrt{\pi} \alpha \left (2(q^+)^2+1\right )}{2 \alpha^2 \sqrt{\pi} (q^+)^2}\nonumber \\
  & >& 0,
\end{eqnarray}
where the last inequality follows through the considerations after (\ref{eq:proppsixi3}). The function $(\psi^+_\beta(\alpha)-1)$ is then indeed increasing on $(\beta,\frac{1+\beta}{2})$.

\underline{\emph{2) For any fixed $\beta\in (0,1)$, $\lim_{\alpha\rightarrow \beta}\psi^+_\beta(\alpha)-1=-1$.}}

This easily follows after one observes that
\begin{equation}\label{eq:proppsixi6nonn}
  \lim_{\alpha\rightarrow \beta}\lp\erfinv\lp 2\frac{1-\alpha}{1-\beta}-1\rp\rp=\infty
\end{equation}

\underline{\emph{3) For any fixed $\beta\in (0,1)$, $\lim_{\alpha\rightarrow \frac{1+\beta}{2}}\psi^+_\beta(\alpha)-1=\infty >0$.}}

This easily follows after one observes that
\begin{equation}\label{eq:proppsixi7nonn}
  \lim_{\alpha\rightarrow \frac{1+\beta}{2}}\lp\erfinv\lp 2\frac{1-\alpha}{1-\beta}-1\rp\rp=0
\end{equation}
Combining the above three observations one can ensure that for any fixed $\beta\in (0,1)$ there is a unique $\alpha$ such that $\psi^+_\beta(\alpha)=1$, which reconfirms that the nonnegative $\ell_1$'s fundamental PT characterization is unambiguous. For the completeness, in Figure \ref{fig:proppsinonn} we present a few numerical results related to the behavior of $\psi^+_\beta(\alpha)$ that are indeed in agreement with the above calculations.
\begin{figure}[htb]
\begin{minipage}[b]{.5\linewidth}
\centering
\centerline{\epsfig{figure=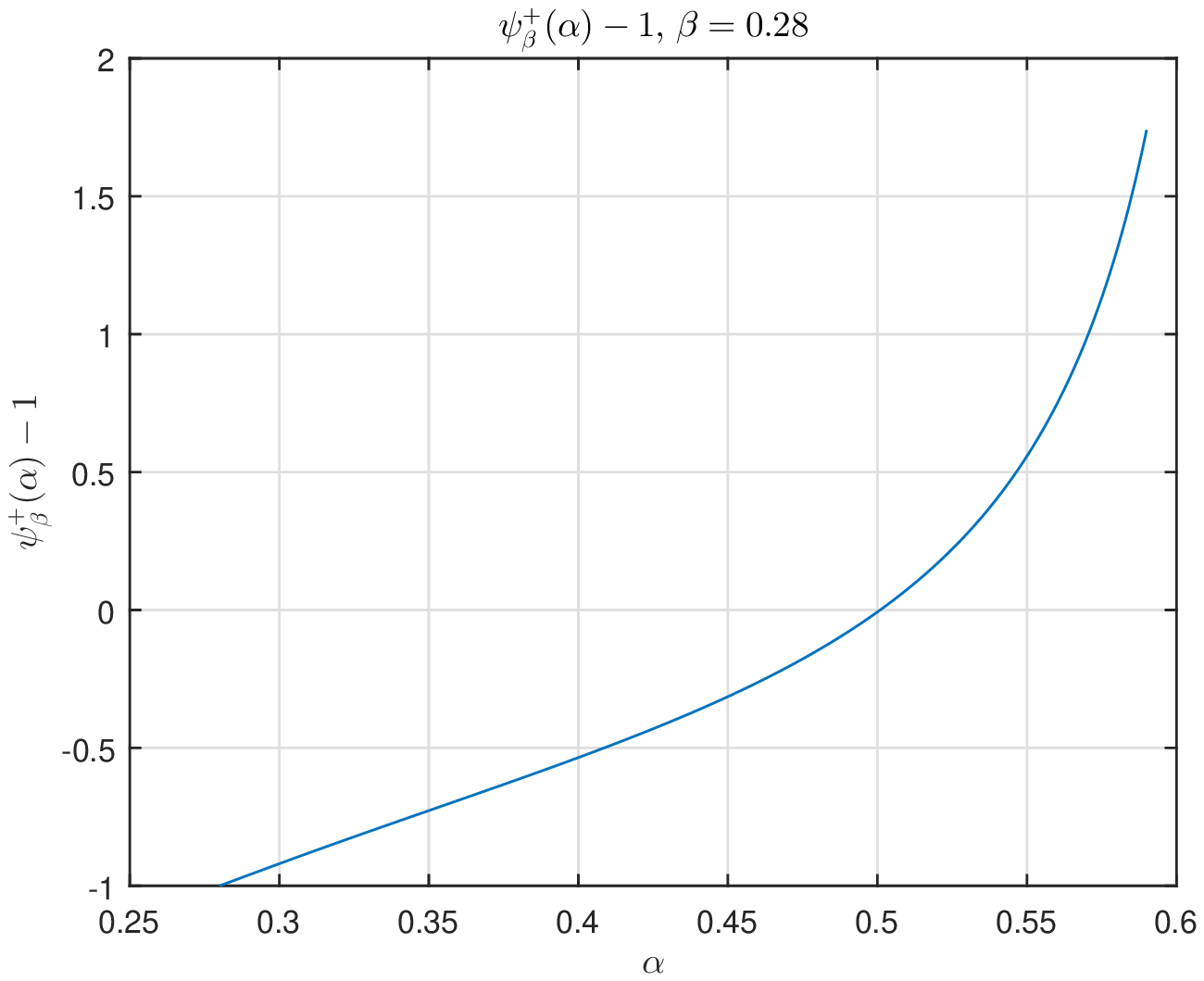,width=9cm,height=7cm}}
\end{minipage}
\begin{minipage}[b]{.5\linewidth}
\centering
\centerline{\epsfig{figure=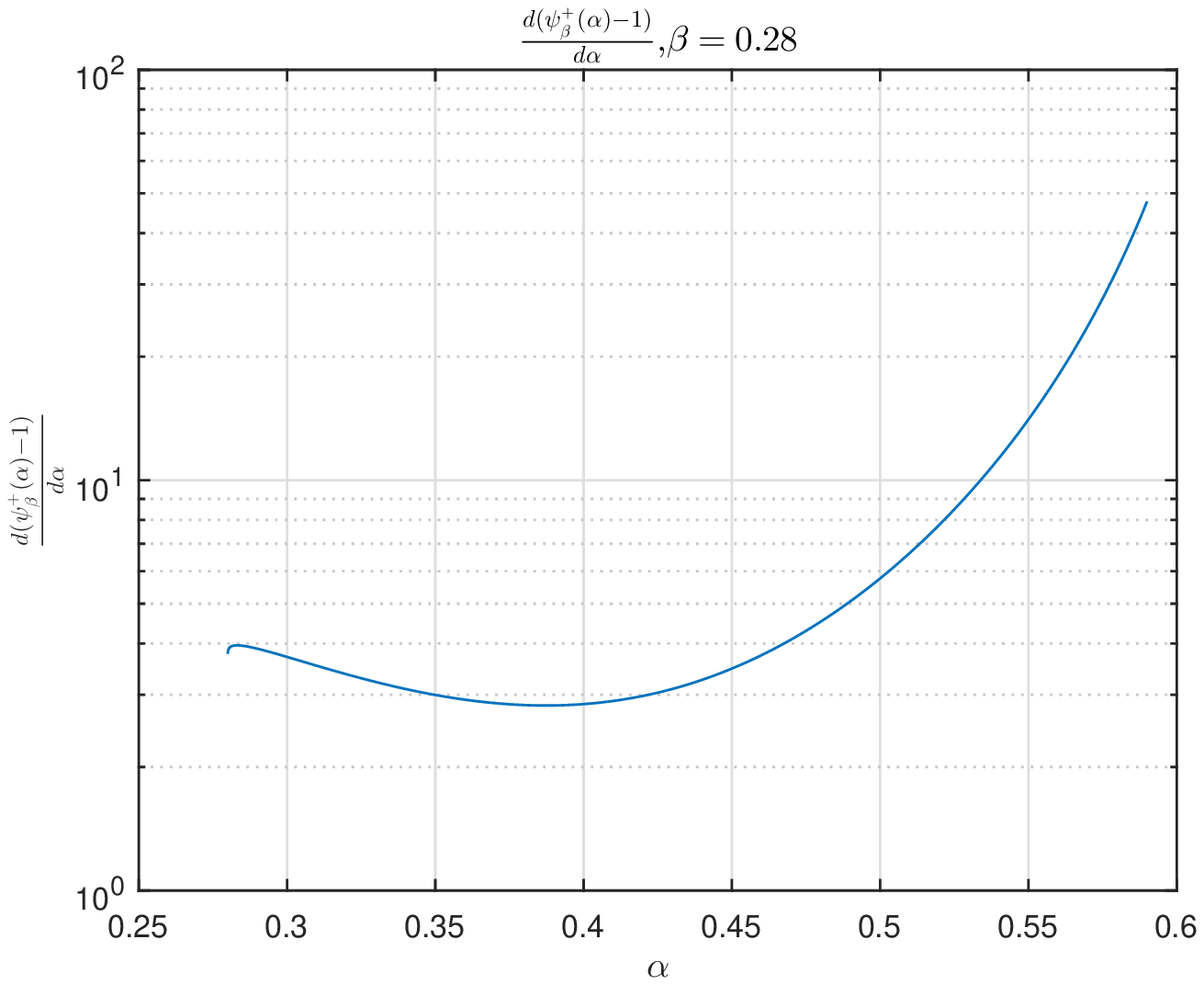,width=9cm,height=7cm}}
\end{minipage}
\caption{Properties of $\psi^+_\beta(\alpha)$: $\psi^+_\beta(\alpha)-1$ as a function of $\alpha$ ($\beta=0.28$) -- left; $\frac{d(\psi^+_\beta(\alpha)-1)}{d\alpha}$ as a function of $\alpha$ ($\beta=0.28$) -- right}
\label{fig:proppsinonn}
\end{figure}

Finally, to give a little bit of a flavor as to what is actually proven in Theorem \ref{thm:thmweakthrnonn} we in Figure \ref{fig:weakl1PTnonn} show the theoretical PT curve that one can obtain based on (\ref{eq:thmweaktheta2nonn}) as well as how it fits the corresponding one obtained through a high-dimensional geometry approach in \cite{DT,DonohoSigned}.
\begin{figure}[htb]
\centering
\centerline{\epsfig{figure=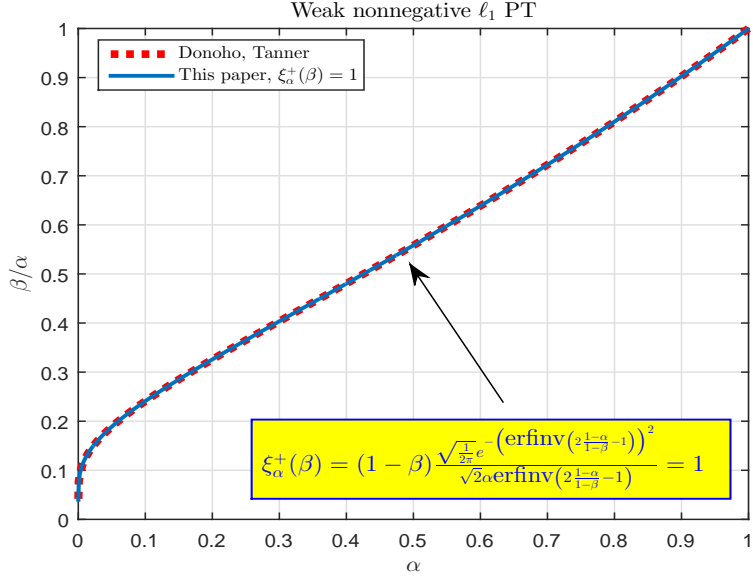,width=11.5cm,height=8cm}}
\caption{Nonnegative $\ell_1$'s weak PT; $\{(\alpha,\beta)|\xi^+_{\alpha}(\beta)=1\}$}
\label{fig:weakl1PTnonn}
\end{figure}


\section{Large deviations -- nonnegative vectors}
\label{sec:ldpnonn}

In the previous section we showed how one can translate the phase transition results obtained for general vectors to the case of nonnegative vectors. In this section we will do the same for the results that relate to the LDP. We will do that through a novel probabilistic concept we presented in Section \ref{sec:ldp}. To insure that there is as little of repetition as possible we will try to skip all the details that are obvious and/or the same as for general vectors and instead will insist only on those that are substantially different.

We will again for the simplicity and without loss of generality assume that the elements $\x_{1},\x_{2},\dots,\x_{n-k}$ of $\x$ are equal to zero and that the elements $\x_{n-k+1},\x_{n-k+2},\dots,\x_n$ all are positive (this time though, this fact is assumed a priori known and as such potentially could be used in the algorithm's design (as actually is in (\ref{eq:l1nonn}))). The following analogue to (\ref{eq:thmeqgenweak}) was proved in \cite{StojnicCSetam09,StojnicICASSP09} and is one of the key features in everything that follows.
\begin{theorem}(\cite{StojnicCSetam09,StojnicICASSP09} Nonzero elements of $\x$ are a priori known to be positive and their location is fixed)
Assume that an $m\times n$ measurement matrix $A$ is given. Let $\x$
be a $k$ sparse vector a priori known to have nonnegative components. Also let $\x_1=\x_2=\dots=\x_{n-k}=0$. Further, assume that $\y=A\x$ and that $\w$ is
a $n\times 1$ vector. Then (\ref{eq:l1nonn}) will
produce the solution of (\ref{eq:system}) if
\begin{equation}
(\forall \w\in \mR^{n} | A\w=0) \quad  -\sum_{i=n-k+1}^n \w_i<\sum_{i=1}^{n-k}\w_{i}, \w_i\geq 0,1\leq i\leq n-k.
\label{eq:thmeqgenweaknonn}
\end{equation}\label{thm:thmgenweaknonn}
\end{theorem}
To facilitate the exposition we set
\begin{equation}
\Sw^+\triangleq\{\w\in S^{n-1}| \quad -\sum_{i=n-k+1}^n \w_i<\sum_{i=1}^{n-k}\w_{i}, \w_i\geq 0,1\leq i\leq n-k.\label{eq:defSwprnonn}
\end{equation}

\subsection{Upper tail -- nonnegative $\x$}
\label{sec:uppertailnonn}

Replacing $\Sw$ by $\Sw^+$ and following what was presented in Section \ref{sec:uppertail} before (\ref{eq:ldpprob3}) we obtain
\begin{equation}
P^+_{err}\leq e^{-\frac{c_3^2}{2}}E\max_{\w\in S_w^+}\min_{\|\y\|_2=1}e^{-c_3(\y^T A\w+g)}\leq
e^{-\frac{c_3^2}{2}}Ee^{-c_3\|\g\|_2}Ee^{c_3w(\h,S_w^+)},
\label{eq:ldpprob3nonn}
\end{equation}
where $P^+_{err}$ is the probability that (\ref{eq:l1nonn}) fails to produce the a priori known to be nonnegative $k$-sparse solution of (\ref{eq:system}),
\begin{equation}
w(\h,\Sw^+)\triangleq\max_{\w\in \Sw^+} (\h^T\w), \label{eq:widthdefSwnonn}
\end{equation}
and the elements of $\h$ are i.i.d. standard normals. Following further what was done in Section \ref{sec:uppertail} (of course, trivially $\h_i$ and $-\h_i$ have the same distribution) we have
\begin{eqnarray}
w(\h,\Sw^+) & = & \min_{\nu\geq0,\gamma\geq 0} \frac{\sum_{i=1}^{n-k}\max(\h_i-\nu,0)^2+\sum_{i=n-k+1}^{n}(\h_i+\nu)^2}{4\gamma}+\gamma\nonumber \\
& = & \min_{\nu\geq0}\sqrt{\sum_{i=1}^{n-k}\max(\h_i-\nu,0)^2+\sum_{i=n-k+1}^{n}(\h_i+\nu)^2}.\label{eq:ldpwhSwnonn}
\end{eqnarray}
The following theorem provides a way to upper bound $P^+_{err}$ (that works for any integers $k$, $m$, and $n$; clearly, $k\leq m\leq n$ to ensure that the results make sense).
\begin{theorem}
Let $A$ be an $m\times n$ matrix in (\ref{eq:system})
with i.i.d. standard normal components. Let
the unknown $\x$ in (\ref{eq:system}) be $k$-sparse and a priori known to be nonnegative and let the location of its nonzero elements be arbitrarily chosen but fixed. Let $P^+_{err}$ be the probability that the solution of (\ref{eq:l1nonn}) is not the a priori known to be nonnegative $k$-sparse solution of (\ref{eq:system}). Then
\begin{equation}
P^+_{err}\leq \min_{c_3\geq 0}e^{-\frac{c_3^2}{2}}e^{-c_3\|\g\|_2}Ee^{c_3w(\h,S_w^+)}
=\min_{c_3\geq 0}\left (e^{-\frac{c_3^2}{2}}\frac{1}{\sqrt{2\pi}^m}\int_{\g}e^{-\sum_{i=1}^{m}\g_i^2/2-c_3\|\g\|_2}d\g \min_{\nu\geq 0,\gamma\geq\frac{c_3}{2}} w_{1,+}^{n-k}w_2^{k}e^{c_3\gamma}\right ),
\label{eq:ldpthm1perrub1nonn}
\end{equation}
where
\begin{eqnarray}
w_{1,+} &=& \frac{1}{\sqrt{2\pi}}\int_{h}e^{-h^2/2}e^{c_3\max(h-\nu,0)^2/4/\gamma}dh
  =\frac{1}{2}\lp\frac{e^{\frac{c_3\nu^2/4/\gamma}{1-c_3/2/\gamma}}}{\sqrt{1-c_3/2/\gamma}}\erfc\left (\frac{\nu}{\sqrt{2}\sqrt{1-c_3/2/\gamma}}\right )+\erf\left (\frac{\nu}{\sqrt{2}}\right )+1\rp\nonumber \\
w_2 &=& \frac{e^{\frac{c_3\nu^2/4/\gamma}{1-c_3/2/\gamma}}}{\sqrt{1-c_3/2/\gamma}}.\label{eq:ldpthm1perrub2nonn}
\end{eqnarray}\label{thm:ldp1nonn}
\end{theorem}
\begin{proof}
Follows by trivially adapting arguments leading up to Theorem \ref{thm:ldp1} and ultimately through the mechanisms developed in \cite{StojnicCSetam09,StojnicCSetamBlock09,StojnicBlockasymldpfinn15,StojnicLiftStrSec13}.
\end{proof}
Moreover, defining
\begin{equation}\label{eq:ldpasymp1nonn}
  I^+_{err}\triangleq\lim_{n\rightarrow\infty}\frac{\log{P^+_{err}}}{n},
\end{equation}
we have the following LDP type of theorem (essentially a nonnegative analogue to Theorem \ref{thm:ldp2}).
\begin{theorem}
Assume the setup of Theorem \ref{thm:ldp1nonn}. Further, let integers $m$, $k$, and $n$ be large ($k\leq m\leq n$) such that $\beta=\frac{k}{n}$ and $\alpha=\frac{m}{n}$ are constants independent of $n$. Assume that a pair $(\alpha,\beta)$  is given. Also, assume the following scaling: $c_3\rightarrow c_3\sqrt{n}$ and $\gamma\rightarrow\gamma\sqrt{n}$. Then
\begin{equation}
I^+_{err}(\alpha,\beta)\triangleq\lim_{n\rightarrow\infty}\frac{\log{P^+_{err}}}{n}
\leq \min_{c_3\geq 0}\left (-\frac{c_3^2}{2}+I_{sph}+\min_{\nu\geq 0,\gamma\geq 0} ((1-\beta)\log{w_{1,+}}+\beta\log{w_2}+c_3\gamma)\right )\triangleq I_{err,u}^{(ub,+)}(\alpha,\beta),
\label{eq:ldpthm2Ierrub1nonn}
\end{equation}
where $I_{sph}$, $\widehat{\gamma}$, $w_1$, and $w_2$ are as in Theorem \ref{thm:ldp2} and
\begin{equation}
w_{1,+} = \frac{1}{2}\lp\frac{e^{\frac{c_3\nu^2/4/\gamma}{1-c_3/2/\gamma}}}{\sqrt{1-c_3/2/\gamma}}\erfc\left (\frac{\nu}{\sqrt{2}\sqrt{1-c_3/2/\gamma}}\right )+\erf\left (\frac{\nu}{\sqrt{2}}\right )+1\rp=\frac{w_1+1}{2}.\label{eq:ldpthm2perrub2nonn}
\end{equation}\label{thm:ldp2nonn}
\end{theorem}
\begin{proof} Follows as a trivial adaptation of Theorem \ref{thm:ldp2}.
\end{proof}

One can of course now numerically estimate the rates of $P^+_{err}$'s decay. In the following sections we will raise the bar a bit higher and provide a closed form solution to the above problem.

\subsection{A detailed analysis of $I_{err,u}^{(ub,+)}$}
\label{sec:analysisIerrnonn}

We start by recalling on
\begin{equation}\label{eq:detanalIerr1nonn}
  A_{0}\triangleq\sqrt{1-\frac{c_3}{2\gamma}}.
\end{equation}
One is then left with the following problem
\begin{equation}\label{eq:detanalIeer2nonn}
I_{err,u}^{(ub,+)}(\alpha,\beta)\triangleq \min_{c_3\geq 0,\nu\geq 0,A_0\leq 1}\zeta^+_{\alpha,\beta}(c_3,\nu,A_0)
\end{equation}
where
\begin{eqnarray}
\zeta^+_{\alpha,\beta}(c_3,\nu,A_0)&=&\left (-\frac{c_3^2}{2}+I_{sph}+(1-\beta)\log\lp{\frac{w_1+1}{2}}\rp+\beta\log{w_2}+\frac{c_3^2}{2(1-A_0^2)}\right ),\label{eq:detanalIeer3nonn}
\end{eqnarray}
and $I_{sph}$, $\widehat{\gamma}$, $w_1$, and $w_2$ are as in Theorems \ref{thm:ldp2} and \ref{thm:ldp2nonn}. We will now proceed by computing the derivatives of $\zeta^+_{\alpha,\beta}(c_3,\nu,A_0)$ with respect to $c_3$, $\nu$, and $A_0$.

\subsubsection{Handling the derivatives of $\zeta^+_{\alpha,\beta}(c_3,\nu,A_0)$}
\label{sec:derivativeszetanonn}

We start with the derivative with respect to $\nu$. Following (\ref{eq:detanalIeer4}) we have
\begin{eqnarray}
\frac{d\zeta^+_{\alpha,\beta}(c_3,\nu,A_0)}{d\nu}&=& \frac{\beta(1-A_0^2)\nu}{A_0^2}+\frac{1-\beta}{w_1+1}\left (\frac{(1-A_0^2)\nu}{A_0^2}\frac{e^{\frac{(1-A_0^2)\nu^2}{2A_0^2}}}{A_0}\erfc\left (\frac{\nu}{\sqrt{2}A_0}\right )-\frac{1-A_0^2}{A_0^2}\frac{2e^{-\frac{\nu^2}{2}}}{\sqrt{2}\sqrt{\pi}}\right)\nonumber \\
&=& \frac{1-A_0^2}{(w_1+1)A_0^2}\left (\beta\nu+\beta\nu\erf\left (\frac{\nu}{\sqrt{2}}\right )+\frac{\nu}{A_0}\frac{\erfc\left (\frac{\nu}{\sqrt{2}A_0}\right )}{e^{-\frac{(1-A_0^2)\nu^2}{2A_0^2}}}-(1-\beta)\sqrt{\frac{2}{\pi}}e^{-\frac{\nu^2}{2}}\right).\nonumber \\
.\label{eq:detanalIeer4nonn}
\end{eqnarray}
For the derivative with respect to $c_3$ one trivially has from (\ref{eq:detanalIeer9})
\begin{equation}
\frac{d\zeta^+_{\alpha,\beta}(c_3,\nu,A_0)}{dc_3}=\frac{d\zeta_{\alpha,\beta}(c_3,\nu,A_0)}{dc_3}
=-c_3+\frac{c_3}{1-A_0^2}+\frac{c_3-\sqrt{(c_3)^2+4\alpha}}{2}.
\label{eq:detanalIeer9nonn}
\end{equation}
Finally, for the derivative with respect to $A_0$ we have
\begin{eqnarray}
\frac{d\zeta^+_{\alpha,\beta}(c_3,\nu,A_0)}{dA_0}&=&\frac{d}{dA_0}\left (-\frac{c_3^2}{2}+I_{sph}+(1-\beta)\log{w_{1,+}}+\beta\log{w_2}+\frac{c_3^2}{2(1-A_0^2)}\right )\nonumber \\
&=& (1-\beta)\frac{d\log\lp\frac{w_1+1}{2}\rp}{dA_0}-\frac{\beta\nu^2}{A_0^3}-\frac{\beta A_0^2}{A_0^3}+\frac{c_3^2A_0}{(1-A_0^2)^2},\nonumber \\
\label{eq:detanalIeer10nonn}
\end{eqnarray}
and
\begin{equation}
\frac{d\log\lp\frac{w_1+1}{2}\rp}{dA_0}=\frac{d\log{(\frac{1}{A_0}e^{\frac{\nu^2}{2A_0^2}}\erfc(\frac{\nu}{\sqrt{2}A_0})+e^{\frac{\nu^2}{2}}(\erf(\frac{\nu}{\sqrt{2}})+1))}}{dA_0}
=
-\frac{e^{\frac{\nu^2}{2A_0^2}}(A_0^2+\nu^2)\erfc(\frac{\nu}{\sqrt{2}A_0})-\sqrt{\frac{2}{\pi}}A_0\nu}
{A_0^3(e^{\frac{\nu^2}{2A_0^2}}\erfc(\frac{\nu}{\sqrt{2}A_0})+A_0e^{\frac{\nu^2}{2}}(\erf(\frac{\nu}{\sqrt{2}})+1))}. \\
\label{eq:detanalIeer11nonn}
\end{equation}
As in Section \ref{sec:uppertail}, we will below select certain values for $c_3$, $\nu$, and $A_0$ and check if one has all of the above derivatives equal to zeros for such values. Before doing that we will recall that setting the derivative with respect to $c_3$ to zero implies that
\begin{equation}\label{eq:detanalIeer12nonn}
  c_3=\frac{(1-A_0^2)\sqrt{\alpha}}{A_0}.
\end{equation}
One then has the following adaptation of (\ref{eq:detanalIeer13})
\begin{equation}
\frac{d\zeta_{\alpha,\beta}(c_3,\nu,A_0)}{dA_0}= (1-\beta)\frac{d\log\lp\frac{w_1+1}{2}\rp}{dA_0}+\frac{(\alpha-\beta) A_0^2-\beta\nu^2}{A_0^3}.
\label{eq:detanalIeer13nonn}
\end{equation}

\subsubsection{Selecting the values for $c_3$, $\nu$, and $A_0$}
\label{sec:valuesc3nuA0nonn}

Now, we will select $\nu$ and $A_0$ in the following way. Let $\beta_w$ be the solution of the fundamental nonnegative $\ell_1$ PT obtained for our given $\alpha$, i.e. let $\beta_w$ be such that
\begin{equation}\label{eq:selvalc3nuA01nonn}
  \xi_{\alpha}(\beta_w)=(1-\beta_w)\frac{\sqrt{\frac{1}{2\pi}}e^{-\lp\erfinv\lp 2\frac{1-\alpha}{1-\beta_w}-1\rp\rp^2}}{\alpha\sqrt{2}\erfinv \lp 2\frac{1-\alpha}{1-\beta_w}-1\rp}=1.
\end{equation}
Then we set
\begin{equation}\label{eq:selvalc3nuA02nonn}
  \nu=\sqrt{2}\erfinv \left (2\frac{1-\alpha}{1-\beta_w}-1\right ).
\end{equation}
Further let $\beta_0$ be such that
\begin{equation}\label{eq:selvalc3nuA03nonn}
  \frac{\alpha-\beta}{\alpha-\beta_0}\xi_{\alpha}(\beta_0)=\frac{\alpha-\beta}{\alpha-\beta_0}(1-\beta_0)
  \frac{\sqrt{\frac{1}{2\pi}}e^{-\lp\erfinv\lp 2\frac{1-\alpha}{1-\beta_0}-1\rp\rp^2}}{\alpha\sqrt{2}\erfinv \lp 2\frac{1-\alpha}{1-\beta_0}-1\rp}=1.
\end{equation}
Then we set
\begin{equation}\label{eq:selvalc3nuA04nonn}
  A_0=\frac{\erfinv \left (2\frac{1-\alpha}{1-\beta_w}-1\right )}{\erfinv \left (2\frac{1-\alpha}{1-\beta_0}-1\right )}=\frac{\nu}{\sqrt{2}\erfinv \left (2\frac{1-\alpha}{1-\beta_0}-1\right )}.
\end{equation}
Finally combining (\ref{eq:detanalIeer12nonn}) and (\ref{eq:selvalc3nuA04nonn}) we obtain the value for $c_3$
\begin{equation}\label{eq:selvalc3nuA05nonn}
  c_3=\frac{(1-A_0^2)\sqrt{\alpha}}{A_0}=\frac{\left (\erfinv \left (2\frac{1-\alpha}{1-\beta_0}-1\right )\right )^2- \left (\erfinv\left (2\frac{1-\alpha}{1-\beta_w}-1\right )\right )^2}{\erfinv \left (2\frac{1-\alpha}{1-\beta_0}-1\right )\erfinv \left (2\frac{1-\alpha}{1-\beta_w}-1\right )}\sqrt{\alpha}.
\end{equation}

\subsubsection{Rechecking the derivatives}
\label{sec:recheckdernonn}

In this subsection we take the above selected values for $c_3$, $\nu$, and $A_0$ and recheck if they indeed ensure that the derivatives are equal to zero. As in Section \ref{sec:recheckder}, that basically amounts to checking if the expressions on the right hand sides of (\ref{eq:detanalIeer4nonn}), (\ref{eq:detanalIeer9nonn}), and (\ref{eq:detanalIeer11nonn}) are equal to zero if $\nu$, $A_0$, and $c_3$ are as in (\ref{eq:selvalc3nuA02nonn}), (\ref{eq:selvalc3nuA04nonn}), and (\ref{eq:selvalc3nuA05nonn}), respectively and $\beta_w$ and $\beta_0$ are as in (\ref{eq:selvalc3nuA01nonn}) and (\ref{eq:selvalc3nuA03nonn}), respectively.

As in Section \ref{sec:recheckder} we observe that (\ref{eq:detanalIeer9nonn}) is trivially satisfied by the way how we chose $c_3$. Combining (\ref{eq:detanalIeer4nonn}), (\ref{eq:selvalc3nuA02nonn}), and (\ref{eq:selvalc3nuA04nonn}) we obtain
\begin{eqnarray}
\frac{d\zeta^+_{\alpha,\beta}(c_3,\nu,A_0)}{d\nu}&=& \frac{1-A_0^2}{(w_1+1)A_0^2}\left (\beta\nu\lp\erf\left (\frac{\nu}{\sqrt{2}}\right )+1\rp+\frac{\nu}{A_0}\frac{\erfc\left (\frac{\nu}{\sqrt{2}A_0}\right )}{e^{-\frac{(1-A_0^2)\nu^2}{2A_0^2}}}-(1-\beta)\sqrt{\frac{2}{\pi}}e^{-\frac{\nu^2}{2}}\right)\nonumber \\
&=&\frac{1-A_0^2}{(w_1+1)A_0^2}\left (\beta\nu\left (2\frac{1-\alpha}{1-\beta_w}-1+1\right )
+\frac{\nu}{A_0}\frac{\erfc\left (\frac{\nu}{\sqrt{2}A_0}\right )}{e^{-\frac{(1-A_0^2)\nu^2}{2A_0^2}}}
-\frac{2(1-\beta)\sqrt{2}\alpha\erfinv \left (2\frac{1-\alpha}{1-\beta_w}-1\right )}{1-\beta_w}\right)\nonumber \\
&=&\frac{1-A_0^2}{(w_1+1)A_0^2}\left (
\frac{\nu}{A_0}\frac{\erfc\left (\frac{\nu}{\sqrt{2}A_0}\right )}{e^{-\frac{(1-A_0^2)\nu^2}{2A_0^2}}}
-\frac{2(\alpha-\beta)\sqrt{2}\erfinv \left (2\frac{1-\alpha}{1-\beta_w}-1\right )}{1-\beta_w}\right)\nonumber \\
&=&\frac{1-A_0^2}{(w_1+1)A_0^2e^{-\frac{(1-A_0^2)\nu^2}{2A_0^2}}}\left (
\frac{\nu}{A_0}\erfc\left (\frac{\nu}{\sqrt{2}A_0}\right )
-\frac{2(\alpha-\beta)\sqrt{2}\erfinv \left (2\frac{1-\alpha}{1-\beta_w}-1\right )}{1-\beta_w}e^{-\frac{(1-A_0^2)\nu^2}{2A_0^2}}\right)\nonumber \\
&=&\frac{1-A_0^2}{(w_1+1)A_0^2e^{-\frac{(1-A_0^2)\nu^2}{2A_0^2}}}\left (
\frac{\nu}{A_0}\erfc\left (\frac{\nu}{\sqrt{2}A_0}\right )
-\frac{(\alpha-\beta)\sqrt{\frac{2}{\pi}}e^{-\frac{\nu^2}{2}}}{\alpha}e^{-\frac{(1-A_0^2)\nu^2}{2A_0^2}}\right)\nonumber \\
&=&\frac{1-A_0^2}{\alpha (w_1+1)A_0^2e^{-\frac{(1-A_0^2)\nu^2}{2A_0^2}}}\left (
\frac{\alpha\nu}{A_0}\erfc\left (\frac{\nu}{\sqrt{2}A_0}\right )
-(\alpha-\beta)\sqrt{\frac{2}{\pi}}e^{-\frac{\nu^2}{2A_0^2}}\right)\nonumber \\
&=&\frac{(1-A_0^2)(w_1+1)^{-1}}{\alpha A_0^2e^{-\frac{(1-A_0^2)\nu^2}{2A_0^2}}}\left (
\alpha\sqrt{2}\erfinv\left (\frac{1-\alpha}{1-\beta_0}\right )\left (2\frac{\alpha-\beta_0}{1-\beta_0}\right )
-(\alpha-\beta)\sqrt{\frac{2}{\pi}}e^{-\left (\erfinv\left (2\frac{1-\alpha}{1-\beta_0}-1\right )\right )^2}\right)\nonumber \\
& = & 0,\label{eq:recheck1nonn}
\end{eqnarray}
where all the equalities follow in exactly the same way as (\ref{eq:recheck1}). The above then confirms that the choice we made in (\ref{eq:selvalc3nuA01nonn})-(\ref{eq:selvalc3nuA05nonn}) indeed ensures that $\frac{d\zeta^+_{\alpha,\beta}(c_3,\nu,A_0)}{d\nu}=0$.

Now, to check the derivative with respect to $A_0$ we look at a combination of (\ref{eq:detanalIeer11}) and (\ref{eq:detanalIeer13}) to find
\begin{eqnarray}
\frac{d\zeta^+_{\alpha,\beta}(c_3,\nu,A_0)}{dA_0} & = &  (1-\beta)\frac{d\log\lp\frac{w_1+1}{2}\rp}{dA_0}+\frac{(\alpha-\beta) A_0^2-\beta\nu^2}{A_0^3}\nonumber \\
& = & -(1-\beta)\frac{e^{\frac{\nu^2}{2A_0^2}}(A_0^2+\nu^2)\erfc(\frac{\nu}{\sqrt{2}A_0})-\sqrt{\frac{2}{\pi}}A_0\nu}
{A_0^3(e^{\frac{\nu^2}{2A_0^2}}\erfc(\frac{\nu}{\sqrt{2}A_0})+A_0e^{\frac{\nu^2}{2}}(\erf(\frac{\nu}{\sqrt{2}})+1))}+\frac{(\alpha-\beta) A_0^2-\beta\nu^2}{A_0^3}\nonumber \\
& = & -(1-\beta)\frac{(A_0^2+\nu^2)\frac{\alpha-\beta}{\alpha}\sqrt{\frac{2}{\pi}}\frac{A_0}{\nu}-\sqrt{\frac{2}{\pi}}A_0\nu}
{A_0^3(\frac{\alpha-\beta}{\alpha}\sqrt{\frac{2}{\pi}}\frac{A_0}{\nu}+A_0e^{\frac{\nu^2}{2}}(\erf(\frac{\nu}{\sqrt{2}})+1))}+\frac{(\alpha-\beta) A_0^2-\beta\nu^2}{A_0^3}\nonumber \\
& = & -(1-\beta)\frac{(A_0^2+\nu^2)\frac{\alpha-\beta}{\alpha}\sqrt{\frac{2}{\pi}}\frac{A_0}{\nu}-\sqrt{\frac{2}{\pi}}A_0\nu}
{A_0^3(\frac{\alpha-\beta}{\alpha}\sqrt{\frac{2}{\pi}}\frac{A_0}{\nu}+\frac{1-\alpha}{\alpha}\sqrt{\frac{2}{\pi}}\frac{A_0}{\nu})}+\frac{(\alpha-\beta) A_0^2-\beta\nu^2}{A_0^3}\nonumber \\
& = & -\frac{(A_0^2+\nu^2)(\alpha-\beta)-\alpha\nu^2}
{A_0^3}
+\frac{(\alpha-\beta) A_0^2-\beta\nu^2}{A_0^3}\nonumber \\
& = & 0,\label{eq:recheck2nonn}
\end{eqnarray}
where we used a reasoning similar to the one employed in the derivation of (\ref{eq:recheck2}). Namely, in (\ref{eq:recheck2nonn}), the third equality holds through a combination of equalities six and eight in (\ref{eq:recheck1nonn}), the fourth equality follows because of (\ref{eq:selvalc3nuA01nonn}) and (\ref{eq:selvalc3nuA02nonn}), and the remaining ones follow through basic algebraic transformations. (\ref{eq:recheck2nonn}) then confirms that the choice we made in (\ref{eq:selvalc3nuA01nonn})-(\ref{eq:selvalc3nuA05nonn}) indeed ensures that $\frac{d\zeta^+_{\alpha,\beta}(c_3,\nu,A_0)}{dA_0}=0$. One can then proceed and check the second derivatives to ensure that the this selection in fact is the global optimum. That can be done both analytically and numerically. Analytical computations are more involved and we refrain from presenting them as they don't bring any novel ideas. As in Section \ref{sec:ldp}, instead of that we will in the following sections prove that the choice (\ref{eq:selvalc3nuA01nonn})-(\ref{eq:selvalc3nuA05nonn}) is not only precisely the one that solves the optimization in (\ref{eq:detanalIeer2nonn}) but also precisely the one that determines $I^+_{err}$. Before doing that we will in the following subsection compute the value of $\zeta^+_{\alpha,\beta}(c_3,\nu,A_0)$ that one gets if $c_3$, $\nu$, and $A_0$ are as in (\ref{eq:selvalc3nuA01nonn})-(\ref{eq:selvalc3nuA05nonn}). As stated above and as will turn out later on, this  value of $\zeta^+_{\alpha,\beta}(c_3,\nu,A_0)$ will be precisely the $I^+_{err}$ from (\ref{eq:ldpasymp1nonn}) and Theorem \ref{thm:ldp2nonn}.

\subsubsection{Computing $\zeta^+_{\alpha,\beta}(c_3,\nu,A_0)$}
\label{sec:computingzetanonn}

As in Section \ref{sec:computingzeta}, before computing $\zeta^+_{\alpha,\beta}(c_3,\nu,A_0)$ we will first compute all other quantities in (\ref{eq:detanalIeer3nonn}), namely, $\widehat{\gamma}$, $I_{sph}$, $w_1$, and $w_2$. From (\ref{eq:computingzeta1}) and (\ref{eq:computingzeta2}) we again have
\begin{eqnarray}\label{eq:computingzeta1nonn}
  \widehat{\gamma} &=& -\frac{A_0\sqrt{\alpha}}{2}\nonumber \\
  I_{sph} & = &-\frac{(1-A_0^2)\alpha}{2}+\alpha\log(A_0).
\end{eqnarray}
For $w_1$ we have
\begin{equation}
w_1 +1=
  \frac{e^{\frac{(1-A_0^2)\nu^2}{2A_0^2}}}{A_0}\erfc\left (\frac{\nu}{\sqrt{2}A_0}\right )+\erf\left (\frac{\nu}{\sqrt{2}}\right )+1
  =\frac{\alpha-\beta}{\alpha\nu}\sqrt{\frac{2}{\pi}}e^{-\frac{\nu^2}{2}}+\erf\left (\frac{\nu}{\sqrt{2}}\right )+1
  =2\frac{\alpha-\beta}{1-\beta_w}+2\frac{1-\alpha}{1-\beta_w}=2\frac{1-\beta}{1-\beta_w},\label{eq:computingzeta3nonn}
\end{equation}
where the second equality follows by a combination of the sixth and the eight equality in (\ref{eq:recheck1nonn}) while the third equality follows by a combination of (\ref{eq:selvalc3nuA01nonn}) and (\ref{eq:selvalc3nuA02nonn}). Finally utilizing (\ref{eq:selvalc3nuA04nonn}) and (\ref{eq:computingzeta3nonn}) we have for $w_2$
\begin{equation}\label{eq:computingzeta4nonn}
  w_2 =
  \frac{e^{\frac{(1-A_0^2)\nu^2}{2A_0^2}}}{A_0}=2\frac{\alpha-\beta}{1-\beta_w}\frac{1}{\erfc\left (\frac{\nu}{\sqrt{2}A_0}\right )}
  =2\frac{\alpha-\beta}{1-\beta_w}\frac{1}{1-\erf\left (\frac{\nu}{\sqrt{2}A_0}\right )}
  =2\frac{\alpha-\beta}{1-\beta_w}\frac{1}{2-2\frac{1-\alpha}{1-\beta_0}}
  =\frac{\alpha-\beta}{\alpha-\beta_0}\frac{1-\beta_0}{1-\beta_w}.
\end{equation}
A combination of (\ref{eq:detanalIeer3nonn}), (\ref{eq:computingzeta1nonn}), (\ref{eq:computingzeta3nonn}), and (\ref{eq:computingzeta4nonn}) then gives
\begin{equation}\label{eq:computingzeta5nonn}
\zeta^+_{\alpha,\beta}(c_3,\nu,A_0)= \alpha\log\left (\frac{\erfinv\left (2\frac{1-\alpha}{1-\beta_w}-1\right )}{\erfinv\left (2\frac{1-\alpha}{1-\beta_0}-1\right )}\right )+(1-\beta)\log{\left (\frac{1-\beta}{1-\beta_w}\right )}+\beta\log{\left (\frac{\alpha-\beta}{\alpha-\beta_0}\frac{1-\beta_0}{1-\beta_w}\right )}.
\end{equation}
We summarize the above results in the following theorem.
\begin{theorem}
Assume the setup of Theorem \ref{thm:ldp2nonn} and assume that a pair $(\alpha,\beta)$  is given. Let $\alpha>\alpha_w$ where $\alpha_w$ is such that $\psi^+_\beta(\alpha_w)=\xi^+_{\alpha_w}(\beta)=1$. Also let $\beta_w$ satisfy the following \textbf{fundamental characterization of the nonnegative $\ell_1$'s PT:}

\begin{equation}\label{eq:thmldp3l1PTnonn}
\xi^+_{\alpha}(\beta_w)\triangleq
(1-\beta_w)\frac{\sqrt{\frac{1}{2\pi}}e^{-\lp\erfinv\lp 2\frac{1-\alpha}{1-\beta_w}-1\rp\rp^2}}{\alpha\sqrt{2}\erfinv \lp 2\frac{1-\alpha}{1-\beta_w}-1\rp}=1.
\end{equation}

\noindent Further let $\beta_0$ satisfy the following \textbf{fundamental characterization of the nonnegative $\ell_1$'s LDP:}

\begin{equation}\label{eq:thmldp3l1LDPnonn}
\frac{\alpha-\beta}{\alpha-\beta_0}\xi^+_{\alpha}(\beta_0)=\frac{\alpha-\beta}{\alpha-\beta_0}
(1-\beta_0)\frac{\sqrt{\frac{1}{2\pi}}e^{-\lp\erfinv\lp 2\frac{1-\alpha}{1-\beta_0}-1\rp\rp^2}}{\alpha\sqrt{2}\erfinv \lp2\frac{1-\alpha}{1-\beta_0}-1\rp}=1.
\end{equation}

\noindent Then
\begin{eqnarray}
I^+_{err}(\alpha,\beta)& \triangleq &\lim_{n\rightarrow\infty}\frac{\log{P^+_{err}}}{n}\nonumber \\
& \leq &
\alpha\log\left (\frac{\erfinv\left (2\frac{1-\alpha}{1-\beta_w}-1\right )}{\erfinv\left (2\frac{1-\alpha}{1-\beta_0}-1\right )}\right )+(1-\beta)\log\left (\frac{1-\beta}{1-\beta_w}\right ) +\beta\log\left (\frac{(\alpha-\beta)(1-\beta_0)}{(\alpha-\beta_0)(1-\beta_w)}\right )\nonumber \\
&\triangleq & I^+_{ldp}(\alpha,\beta).\nonumber \\
\label{eq:ldpthm3Ierrub1nonn}
\end{eqnarray}
Moreover, the following choice for $\nu$, $c_3$, and $\gamma$ in the optimization problem in Theorem \ref{thm:ldp2} achieves the right hand side of (\ref{eq:ldpthm3Ierrub1nonn})
\begin{eqnarray}
  \nu &  = & \sqrt{2}\erfinv \left (2\frac{1-\alpha}{1-\beta_w}-1\right )\nonumber \\
    A_0&=&\frac{\erfinv \left (2\frac{1-\alpha}{1-\beta_w}-1\right )}{\erfinv \left (2\frac{1-\alpha}{1-\beta_0}-1\right )}=\frac{\nu}{\sqrt{2}\erfinv \left (2\frac{1-\alpha}{1-\beta_0}-1\right )}\nonumber \\
      c_3& = &\frac{(1-A_0^2)\sqrt{\alpha}}{A_0}=\frac{\left (\erfinv \left (2\frac{1-\alpha}{1-\beta_0}-1\right )\right )^2- \left (\erfinv\left (2\frac{1-\alpha}{1-\beta_w}-1\right )\right )^2}{\erfinv \left (2\frac{1-\alpha}{1-\beta_0}-1\right )\erfinv \left (2\frac{1-\alpha}{1-\beta_w}-1\right )}\sqrt{\alpha}\nonumber \\
      \gamma&=&\frac{c_3}{2(1-A_0^2)}=\frac{\sqrt{\alpha}}{2A_0}=\frac{\sqrt{\alpha}\erfinv \left (2\frac{1-\alpha}{1-\beta_0}-1\right )}{2\erfinv \left (2\frac{1-\alpha}{1-\beta_w}-1\right )}.
\label{eq:ldpthm3perrub2nonn}
\end{eqnarray}\label{thm:ldp3nonn}
\end{theorem}
\begin{proof} Follows from the above discussion.
\end{proof}
As in Section \ref{sec:ldp}, we will present the results that one can obtain based on the above theorem later on when we complement them with the lower tail estimate in the following subsection and eventually connect them to the high-dimensional geometry. At that time we will also find it useful to look at a couple of properties of the function $\frac{\alpha-\beta}{\alpha-\beta_0}\xi^+_{\alpha}(\beta_0)$.

\subsection{Lower tail -- nonegative vectors}
\label{sec:lowertailnonn}

%

In this section we complement the upper-tail results from the previous section with the corresponding lower tail type of large deviations. We again closely follow what was done in the corresponding part of Section \ref{sec:ldp}. We start by introducing $P^+_{cor}$
\begin{equation}
P^+_{cor} \triangleq P(\min_{A\w=0,\|\w\|_2\leq 1,\w_i\geq 0,1\leq i\leq n-k}\sum_{i=n-k+1}^n \w_i+\sum_{i=1}^{n-k}\w_{i}\geq 0).
\label{eq:ldpproblowernonn}
\end{equation}
Clearly, $P^+_{cor}=1-P^+_{err}$ is the probability that (\ref{eq:l1nonn}) does produce the $k$ sparse a priori known to be nonnegative solution of (\ref{eq:system}), i.e. correctly solves the original linear system. Following (\ref{eq:ldpprob3lower}) we have
\begin{equation}
P^+_{cor}
\leq  \min_{t_1}\min_{c_3\geq 0}
Ee^{c_3\|\g\|_2}Ee^{-c_3w(\h,S_w^+)}e^{-c_3t_1}/P(g\geq t_1).
\label{eq:ldpprob3lowernonn}
\end{equation}
Moreover, setting $I^+_{cor}$
\begin{equation}\label{eq:ldpasymp1lowernonn}
  I^+_{cor}\triangleq\lim_{n\rightarrow\infty}\frac{\log{P^+_{cor}}}{n},
\end{equation}
and following Theorem \ref{thm:ldp2lower} we below establish the lower tail analogue to Theorem \ref{thm:ldp2nonn}.
\begin{theorem}
Assume the setup of Theorem \ref{thm:ldp2nonn}. Then
\begin{equation}
I^+_{cor}\triangleq\lim_{n\rightarrow\infty}\frac{\log{P^+_{cor}}}{n}
\leq \min_{c_3\geq 0}\left (-\frac{c_3^2}{2}+I_{sph}^++\max_{\nu\geq 0,\gamma^{(s)}\geq 0} ((1-\beta_w)\log{w_{1,+}}+\beta_w\log{w_2}+\beta_w\log{w_3}-c_3\gamma)\right )\triangleq I_{cor,l}^{(ub,+)},
\label{eq:ldpthm2Icorub1nonn}
\end{equation}
where $I_{sph}^+$, $\widehat{\gamma_+}$, and $w_2$ are as in Theorem \ref{thm:ldp2lower} and
\begin{equation}
w_{1,+} = \frac{1}{\sqrt{2\pi}}\int_{h}e^{-h^2/2}e^{-c_3\max(h-\nu,0)^2/4/\gamma}dh
  =\frac{1}{2}\lp\frac{e^{\frac{-c_3\nu^2/4/\gamma}{1+c_3/2/\gamma}}}{\sqrt{1+c_3/2/\gamma}}\erfc\left (\frac{\nu}{\sqrt{2}\sqrt{1+c_3/2/\gamma}}\right )+\erf\left (\frac{\nu}{\sqrt{2}}\right )+1\rp. \\\label{eq:ldpthm2perrub2lowernonn}
\end{equation}\label{thm:ldp2lowernonn}
\end{theorem}
\begin{proof} Follows by a line of reasoning similar to the one employed leading up to Theorem \ref{thm:ldp2lower}.
\end{proof}

\subsection{A detailed analysis of $I_{cor,l}^{(ub,+)}$}
\label{sec:analysisIcornonn}

As in Section \ref{sec:analysisIcor}, noting the change $c_3\rightarrow -c_3$ we have as in Section \ref{sec:analysisIerrnonn}
\begin{equation}\label{eq:detanalIcor1nonn}
  A_{0}\triangleq\sqrt{1-\frac{c_3}{2\gamma}}.
\end{equation}
and
\begin{equation}\label{eq:detanalIcor2nonn}
I_{cor,l}^{(ub,+)}(\alpha,\beta)\triangleq \min_{c_3\leq 0}\max_{\nu\geq 0,A_0\leq 1}\zeta^+_{\alpha,\beta}(c_3,\nu,A_0)
\end{equation}
where
\begin{eqnarray}
\zeta^+_{\alpha,\beta}(c_3,\nu,A_0)&=&\left (-\frac{c_3^2}{2}+I_{sph}^++(1-\beta)\log{w_{1,+}}+\beta\log{w_2}+\frac{c_3^2}{2(1-A_0^2)}\right )\nonumber \\
I_{sph}^+ &=& -\widehat{\gamma^+}c_3-\frac{\alpha }{2}\log\left (1+\frac{c_3}{2\widehat{\gamma^+}}\right )\nonumber \\
  \widehat{\gamma^+} &=& \frac{-c_3+\sqrt{c_3^2+4\alpha}}{4}=-\widehat{\gamma}\nonumber \\
w_{1,+} &=&
  \frac{1}{2}\lp\frac{e^{\frac{(1-A_0^2)\nu^2}{2A_0^2}}}{A_0}\erfc\left (\frac{\nu}{\sqrt{2}A_0}\right )+\erf\left (\frac{\nu}{\sqrt{2}}\right )+1\rp\nonumber \\
  w_2 &=&
  \frac{e^{\frac{(1-A_0^2)\nu^2}{2A_0^2}}}{A_0}.\label{eq:detanalIcor3nonn}
\end{eqnarray}
 After observing that $\zeta^+_{\alpha,\beta}(c_3,\nu,A_0)$ defined in (\ref{eq:detanalIcor3nonn}) is exactly the same as the corresponding one in (\ref{eq:detanalIeer3nonn}), one can proceed with computation of all the derivatives as earlier and all the results will match those obtained for the upper tail. Consequently, the selected values for $c_3$, $\nu$, $\gamma$, and $A_0$ will have the same form. Instead of repeating all these calculations we summarize them in the following theorem, essentially a lower tail analogue of Theorem \ref{thm:ldp3nonn}.
\begin{theorem}
Assume the setup of Theorem \ref{thm:ldp3nonn} and assume that a pair $(\alpha,\beta)$  is given. Differently from Theorem \ref{thm:ldp3nonn}, let $\alpha<\alpha_w$ where $\alpha_w$ is such that $\psi^+_\beta(\alpha_w)=\xi^+_{\alpha_w}(\beta)=1$. Also let $\beta_w$ and $\beta_0$ satisfy the \textbf{fundamental nonnegative $\ell_1$'s PT and LDP characterizations}, respectively as in Theorem \ref{thm:ldp3nonn}. Then choosing $\nu$, $c_3$, and $\gamma$ in the optimization problem in (\ref{eq:ldpthm2Icorub1nonn}) as $\nu$, $-c_3$, and $\gamma$  from Theorem \ref{thm:ldp3nonn} (or equivalently, choosing $\nu$, $c_3$, and $A_0$ in the optimization problem in (\ref{eq:detanalIcor2nonn}) as $\nu$, $c_3$, and $A_0$  from Theorem \ref{thm:ldp3nonn}) gives
\begin{equation}
\zeta^+_{\alpha,\beta}(c_3,\nu,A_0)=
\alpha\log\left (\frac{\erfinv\left (2\frac{1-\alpha}{1-\beta_w}-1\right )}{\erfinv\left (2\frac{1-\alpha}{1-\beta_0}-1\right )}\right )+(1-\beta)\log\left (\frac{1-\beta}{1-\beta_w}\right ) +\beta\log\left (\frac{(\alpha-\beta)(1-\beta_0)}{(\alpha-\beta_0)(1-\beta_w)}\right )=I^+_{ldp}(\alpha,\beta).
\label{eq:ldpthm3Icorub1nonn}
\end{equation}
\label{thm:ldp3lowernonn}
\end{theorem}
\begin{proof} Follows from the considerations leading up to Theorem \ref{thm:ldp3nonn}.
\end{proof}
Similarly to what was mentioned after Theorem \ref{thm:ldp3lower}, we add that in the scenario of interest in Theorem \ref{thm:ldp3lowernonn}, i.e. for $\alpha<\alpha_w$, one has $\beta_w<\beta_0$ which implies $A_0>1$ and ultimately $c_3<0$. Exactly opposite happens in the scenario of interest in Theorem \ref{thm:ldp3nonn}.

\subsection{High-dimensional geometry -- nonnegative vectors}
\label{sec:hdgnonn}

In this section we will discuss how the results obtained in Section \ref{sec:hdg} can be translated to the case of nonnegative vectors. We will assume that we are given a pair $(\alpha,\beta)$ and that $\beta_w$ and $\beta_0$ are given by the nonnegative $\ell_1$ fundamental PT and LDP characterizations defined earlier. Following further closely Section \ref{sec:hdgnonn}, we assume the upper tail regimes, i.e. $\alpha>\alpha_w$ (where $\alpha_w$ is such that $\psi^+_\beta(\alpha_w)=\xi^+_{\alpha_w}(\beta)=1$) and start with the following collection of results established in \cite{StojnicEquiv10} (these are of course the nonnegative analogues to the results utilized in Section \ref{sec:hdg}).
\begin{equation}\label{eq:hdg1nonn}
  \Psi^+_{net}(\alpha,\beta)=I^+_{err}(\alpha,\beta)\triangleq\lim_{n\rightarrow\infty}\frac{\log{P^+_{err}}}{n}=\psicomnon+\psiintnon-\psiextnon,
\end{equation}
where
\begin{eqnarray}
\psicomnon & = & -(\alpha-\beta)\log\left (\frac{\alpha-\beta}{1-\beta}\right )-(1-\alpha)\log\left (\frac{1-\alpha}{1-\beta}\right )\nonumber \\
\psiintnon & = & \min_{y\geq 0} (\alpha y^2 +(\alpha-\beta)\log(\erfc(y)))-(\alpha-\beta)\log(2)\nonumber\\
\psiextnon & = & \min_{y\geq 0} (\alpha y^2 -(1-\alpha)\log(1+\erf(y))+(1-\alpha)\log(2)). \label{eq:hdg2nonn}
\end{eqnarray}
Let $y_{int}^+$ be the solution of the optimization associated with $\psiintnon$ and let $y_{ext}^+$ be the solution of the optimization associated with $\psiextnon$. As in (\ref{eq:hdg3}) we have
\begin{equation}
\frac{d(\alpha y^2 +(\alpha-\beta)\log(\erfc(y)))}{dy}  =  2\alpha y+\frac{\alpha-\beta}{\erfc(y)}\frac{d\erfc(y)}{dy}= 2\alpha y-\frac{\alpha-\beta}{1-\erf(y)}\frac{2e^{-y^2}}{\sqrt{\pi}}. \label{eq:hdg3nonn}
\end{equation}
Choosing
\begin{equation}\label{eq:hdg4nonn}
y_{int}^+  =  \erfinv\left (2\frac{1-\alpha}{1-\beta_0}-1\right ),
\end{equation}
the derivative in (\ref{eq:hdg3nonn}) becomes
\begin{eqnarray}
\frac{d(\alpha y^2 +(\alpha-\beta)\log(\erfc(y)))}{dy} | _{y=y_{int}^+}  & = &   2\alpha y_{int}^+-\frac{\alpha-\beta}{1-\erf(y_{int}^+)}\frac{2e^{-(y^+_{int})^2}}{\sqrt{\pi}}\nonumber \\
& = &2\alpha \erfinv\left (2\frac{1-\alpha}{1-\beta_0}-1\right )-\frac{\alpha-\beta}{\alpha-\beta_0}(1-\beta_0)\frac{e^{-\left (\erfinv\left (2\frac{1-\alpha}{1-\beta_0}-1\right )\right )^2}}{\sqrt{\pi}}\nonumber \\
& = &2\alpha \erfinv\left (2\frac{1-\alpha}{1-\beta_0}-1\right )\left (1-\frac{\alpha-\beta}{\alpha-\beta_0}\xi^+_\alpha(\beta_0)\right )\nonumber \\
& = &0,\nonumber \\ \label{eq:hdg5nonn}
\end{eqnarray}
where the last equality follows by the fundamental characterization of nonnegative $\ell_1$'s LDP.  Moreover, as was shown in (\ref{eq:hdg11}) $(\alpha y^2 +(\alpha-\beta)\log(\erfc(y)))$ is convex and $y_{int}^+$ is not only its a local but also its a global optimum (minimum) as well. A combination of (\ref{eq:hdg2nonn}) and (\ref{eq:hdg4nonn}) together with the nonnegative $\ell_1$'s fundamental LDP then finally gives
\begin{eqnarray}\label{eq:hdg12nonn}
\psiintnon & = &  \alpha (y^+_{int})^2 +(\alpha-\beta)\log(\erfc(y_{int}^+))-(\alpha-\beta)\log(2) \nonumber \\
 & = &   \alpha \left (\erfinv\left (2\frac{1-\alpha}{1-\beta_0}-1\right )
\right )^2 +(\alpha-\beta)\log\left (\frac{\alpha-\beta_0}{1-\beta_0}\right )\nonumber \\
& = & \alpha \log\left (e^{\left (\erfinv\left (2\frac{1-\alpha}{1-\beta_0}-1\right )
\right )^2}\right ) +(\alpha-\beta)\log\left (\frac{\alpha-\beta_0}{1-\beta_0}\right )\nonumber \\
& = & \alpha \log\left (\frac{\alpha-\beta}{\alpha-\beta_0}\frac{1-\beta_0}{\alpha\sqrt{2}\erfinv\left (2\frac{1-\alpha}{1-\beta_0}-1\right )}\right ) +(\alpha-\beta)\log\left (\frac{\alpha-\beta_0}{1-\beta_0}\right )-\alpha \log(\sqrt{2\pi})\nonumber \\
& = & -\alpha \log\left (\sqrt{2}\erfinv\left (2\frac{1-\alpha}{1-\beta_0}-1\right )\right )+\alpha\log\left (\frac{\alpha-\beta}{\alpha}\right )
-\beta\log\left (\frac{\alpha-\beta_0}{1-\beta_0}\right )-\alpha \log(\sqrt{2\pi})\nonumber \\
\end{eqnarray}
Using the nonnegative $\ell_1$'s fundamental PT and the definition of $\beta_w$ one can then further utilize the results of \cite{StojnicEquiv10} to obtain
\begin{eqnarray}\label{eq:hdg13nonn}
\psiextnon  & = &   \min_{y\geq 0} (\alpha y^2 -(1-\alpha)\log(1+\erf(y))+(1-\alpha)\log(2))\nonumber \\
& = &\alpha\left (\erfinv\left (2\frac{1-\alpha}{1-\beta_w}-1\right )\right )^2-(1-\alpha)\log\left (2\frac{1-\alpha}{1-\beta_w}\right )+(1-\alpha)\log(2)\nonumber \\
& = &\alpha\log\left (e^{\left (\erfinv\left (2\frac{1-\alpha}{1-\beta_w}-1\right )\right )^2}\right )-(1-\alpha)\log\left (\frac{1-\alpha}{1-\beta_w}\right )\nonumber \\
& = &\alpha \log\left (\frac{1-\beta_w}{\alpha\sqrt{2}\erfinv\left (2\frac{1-\alpha}{1-\beta_w}-1\right )}\right )-(1-\alpha)\log\left (\frac{1-\alpha}{1-\beta_w}\right )-\alpha \log(\sqrt{2\pi})\nonumber \\
& = &-\alpha \log\left (\sqrt{2}\erfinv\left (2\frac{1-\alpha}{1-\beta_w}-1\right )\right )+\alpha \log\left (\frac{1-\beta_w}{\alpha}\right )-(1-\alpha)\log\left (\frac{1-\alpha}{1-\beta_w}\right )-\alpha \log(\sqrt{2\pi})\nonumber \\
& = &-\alpha \log\left (\sqrt{2}\erfinv\left (2\frac{1-\alpha}{1-\beta_w}-1\right )\right )-\alpha \log(\alpha)-(1-\alpha)\log (1-\alpha)+\log(1-\beta_w)-\alpha \log(\sqrt{2\pi}).\nonumber \\
\end{eqnarray}
Finally one can combine (\ref{eq:hdg1nonn}), (\ref{eq:hdg12nonn}), and (\ref{eq:hdg13nonn}) to obtain
\begin{eqnarray}
  \Psi_{net}^+(\alpha,\beta)&=&I^+_{err}(\alpha,\beta)=\psicomnon+\psiintnon-\psiextnon\nonumber \\
& = & -(\alpha-\beta)\log\left (\frac{\alpha-\beta}{1-\beta}\right )-(1-\alpha)\log\left (\frac{1-\alpha}{1-\beta}\right )
-\alpha \log\left (\sqrt{2}\erfinv\left (\frac{1-\alpha}{1-\beta_0}\right )\right )\nonumber \\
&&+\alpha\log\left (\frac{\alpha-\beta}{\alpha}\right )
-\beta\log\left (\frac{\alpha-\beta_0}{1-\beta_0}\right )\nonumber\\
&&+\alpha \log\left (\sqrt{2}\erfinv\left (2\frac{1-\alpha}{1-\beta_w}-1\right )\right )+\alpha \log(\alpha)+(1-\alpha)\log (1-\alpha)-\log(1-\beta_w)\nonumber \\
& = & \alpha\log\left (\frac{1-\beta}{\alpha-\beta}\frac{1-\alpha}{1-\beta}\frac{\alpha-\beta}{\alpha}\frac{\alpha}{1-\alpha}\right )+\beta\log\left (\frac{\alpha-\beta}{1-\beta}\frac{1-\beta_0}{\alpha-\beta_0}\right )
+\alpha \log\left (\frac{\erfinv\left (2\frac{1-\alpha}{1-\beta_w}-1\right )}{\erfinv\left (2\frac{1-\alpha}{1-\beta_0}-1\right )}\right )\nonumber \\
&&+\log(1-\beta)-\log(1-\beta_w)\nonumber \\
& = & \beta\log\left (\frac{\alpha-\beta}{\alpha-\beta_0}\frac{1-\beta_0}{1-\beta}\frac{1-\beta_w}{1-\beta_w}\right )
+\alpha \log\left (\frac{\erfinv\left (2\frac{1-\alpha}{1-\beta_w}-1\right )}{\erfinv\left (2\frac{1-\alpha}{1-\beta_0}-1\right )}\right )+
\log\left (\frac{1-\beta}{1-\beta_w}\right )\nonumber \\
& = & \alpha \log\left (\frac{\erfinv\left (2\frac{1-\alpha}{1-\beta_w}-1\right )}{\erfinv\left (2\frac{1-\alpha}{1-\beta_0}-1\right )}\right )+
(1-\beta)\log\left (\frac{1-\beta}{1-\beta_w}\right )+\beta\log\left (\frac{\alpha-\beta}{\alpha-\beta_0}\frac{1-\beta_0}{1-\beta_w}\right )
=I^+_{ldp}(\alpha,\beta).\nonumber \\
\label{eq:hdg14nonn}
\end{eqnarray}
A combination of (\ref{eq:ldpthm3Ierrub1nonn}), (\ref{eq:hdg1nonn}), and (\ref{eq:hdg14nonn}) then gives
\begin{equation}\label{eq:hdg15nonn}
  I^+_{err}=I^+_{ldp}(\alpha,\beta),
\end{equation}
and ensures that the choice for $\nu$, $A_0$, $c_3$, and $\gamma$ made in (\ref{eq:ldpthm3perrub2nonn}) is indeed optimal. Moreover, in the lower tail regime ($\alpha<\alpha_w$, where $\alpha_w$ is such that $\psi_\beta(\alpha_w)=\xi^+_{\alpha_w}(\beta)=1$) considerations from \cite{StojnicEquiv10} ensure that one also has
\begin{equation}\label{eq:hdg1anonn}
  \Psi_{net}^+(\alpha,\beta)=I^+_{cor}(\alpha,\beta)\triangleq\lim_{n\rightarrow\infty}\frac{\log{P^+_{cor}}}{n}=\psicomnon+\psiintnon-\psiextnon,
\end{equation}
where $\psicomnon$, $\psiintnon$, and $\psiextnon$ are as in (\ref{eq:hdg2nonn}). Finally, we are in position to fully characterize the nonnegative $\ell_1$'s LDP. The following theorem does so.
\begin{theorem}[nonnegative $\ell_1$'s LDP]
Assume the setup of Theorem \ref{thm:thmweakthrnonn} and assume that a pair $(\alpha,\beta)$ is given. Let $P^+_{err}$ be the probability that the solutions of (\ref{eq:system}) and (\ref{eq:l1nonn}) coincide and let $P^+_{cor}$ be the probability that the solutions of (\ref{eq:system}) and (\ref{eq:l1nonn}) do \emph{not} coincide. Let $\alpha_w$ and $\beta_w$ satisfy the \textbf{nonnegative $\ell_1$'s fundamental PT} characterizations in the following way
\begin{equation}
\psi^+_{\beta}(\alpha_w)\triangleq
(1-\beta)\frac{\sqrt{\frac{2}{\pi}}e^{-\lp\erfinv\lp 2\frac{1-\alpha_w}{1-\beta}-1\rp\rp^2}}{\alpha_w\sqrt{2}\erfinv \lp 2\frac{1-\alpha_w}{1-\beta}-1\rp}=1\quad \mbox{and} \quad
\xi^+_{\alpha}(\beta_w)\triangleq
(1-\beta_w)\frac{\sqrt{\frac{1}{2\pi}}e^{-\lp\erfinv\lp 2\frac{1-\alpha}{1-\beta_w}-1\rp\rp^2}}{\alpha\sqrt{2}\erfinv \lp 2\frac{1-\alpha}{1-\beta_w}-1\rp}=1.\label{eq:thmfinalldpl11nonn}
\end{equation}
Further let $\beta_0$ satisfy the following \textbf{nonnegative $\ell_1$'s fundamental LDP} characterization
\begin{equation}\label{eq:thmfinalldpl12nonn}
\frac{\alpha-\beta}{\alpha-\beta_0}\xi^+_{\alpha}(\beta_0)=\frac{\alpha-\beta}{\alpha-\beta_0}
(1-\beta_0)(1-\beta_w)\frac{\sqrt{\frac{1}{2\pi}}e^{-\lp\erfinv\lp 2\frac{1-\alpha}{1-\beta_w}-1\rp\rp^2}}{\alpha\sqrt{2}\erfinv \lp 2\frac{1-\alpha}{1-\beta_w}-1\rp}=1.
\end{equation}
Finally, let $I^+_{ldp}(\alpha,\beta)$ be defined through the following \textbf{nonnegative $\ell_1$'s fundamental LDP rate function} characterization
\begin{equation}
I^+_{ldp}(\alpha,\beta)\triangleq
\alpha\log\left (\frac{\erfinv\left (2\frac{1-\alpha}{1-\beta_w}-1\right )}{\erfinv\left (2\frac{1-\alpha}{1-\beta_0}-1\right )}\right )+(1-\beta)\log\left (\frac{1-\beta}{1-\beta_w}\right ) +\beta\log\left (\frac{(\alpha-\beta)(1-\beta_0)}{(\alpha-\beta_0)(1-\beta_w)}\right ).
\label{eq:thmfinalldpl13nonn}
\end{equation}
Then if $\alpha>\alpha_w$
\begin{equation}
I^+_{err}(\alpha,\beta)\triangleq\lim_{n\rightarrow\infty}\frac{\log{P^+_{err}}}{n}=I^+_{ldp}(\alpha,\beta).\label{eq:thmfinalldpl14nonn}
\end{equation}
Moreover, if $\alpha<\alpha_w$
\begin{equation}
I^+_{cor}(\alpha,\beta)\triangleq\lim_{n\rightarrow\infty}\frac{\log{P^+_{cor}}}{n}=I^+_{ldp}(\alpha,\beta).\label{eq:thmfinalldpl15nonn}
\end{equation}\label{thm:finalldpl1nonn}
\end{theorem}
\begin{proof} Follows from the above discussion.
\end{proof}
Similarly to what was done in Section \ref{sec:hdg}, before we present the results that can be obtained based on the above theorem we will establish a few additional properties of function $\frac{\alpha-\beta}{\alpha-\beta_0}\xi^+_{\alpha}(\beta_0)$ to ensure that everything is on a right mathematical track.

\subsubsection{Properties of $\frac{\alpha-\beta}{\alpha-\beta_0}\xi^+_{\alpha}(\beta_0)$}
\label{sec:proppsixinonn}

As we did in Section \ref{sec:proppsixinonn}, in this subsection we will try to complement some of the key properties of functions $\xi^+_\alpha(\beta)$ and $\psi^+_\beta(\alpha_w)$ from Theorem \ref{thm:thmweakthrnonn} that we introduced in Section \ref{sec:propxinonn}.

Similarly to our earlier observations, the key observation regarding $\frac{\alpha-\beta}{\alpha-\beta_0}\xi^+_{\alpha}(\beta_0)$ is that for any fixed $(\alpha,\beta)\in (0,1)\times(0,\alpha)$ there is a unique $\beta_0$ such that $\frac{\alpha-\beta}{\alpha-\beta_0}\xi^+_{\alpha}(\beta_0)=1$. This essentially ensures that (\ref{eq:thmfinalldpl12nonn}) is an unambiguous LDP characterization. To confirm that this is indeed true we proceed in a fashion similar to the one showcased in Section \ref{sec:propxinonn}. Setting as in (\ref{eq:proppsixi2nonn})
\begin{equation}\label{eq:propxi02nonn}
  q^+=\erfinv\left (2\frac{1-\alpha}{1-\beta_0}-1\right ),
\end{equation}
we have
\begin{equation}\label{eq:propxi03}
  \frac{\alpha-\beta}{\alpha-\beta_0}\xi^+_{\alpha}(\beta_0)=1 \Leftrightarrow \frac{\alpha-\beta}{\alpha}\frac{1}{\erfc(q^+)}\frac{\sqrt{\frac{1}{\pi}}e^{-(q^+)^2}}{q^+}=1
  \Leftrightarrow \frac{\sqrt{\frac{1}{\pi}}e^{-(q^+)^2}}{q^+}-\erfc(q^+)c_{\alpha,\beta}=0, c_{\alpha,\beta}> 1.
\end{equation}
Following the reasoning of (\ref{eq:propxi04})-(\ref{eq:propxi06})
one has that $\left ( \frac{\sqrt{\frac{1}{\pi}}e^{-q^2}}{q}-\erfc(q)c_{\alpha,\beta}\right )$ has a unique solution (moreover, it is in the interval $(0, \frac{1}{\sqrt{2c_{\alpha,\beta}(c_{\alpha,\beta}-1)}})$). This then implies that $\frac{\alpha-\beta}{\alpha-\beta_0}\xi^+_{\alpha}(\beta_0)=1$ also has a unique solution, i.e. that for any fixed $(\alpha,\beta)\in (0,1)\times (0,\alpha)$ there is a unique $\beta_0$ such that $\frac{\alpha-\beta}{\alpha-\beta_0}\xi^+_\alpha(\beta_0)=1$, which as mentioned above essentially means that (\ref{eq:thmfinalldpl12}) is an unambiguous LDP characterization. We recall that a few numerical results related to the behavior of $\left ( \frac{\sqrt{\frac{1}{\pi}}e^{-q^2}}{q}-\erfc(q)c_{\alpha,\beta}\right )$ (and ultimately of $\left (\frac{\alpha-\beta}{\alpha-\beta_0}\xi_{\alpha}(\beta_0)-1\right )$) can be found in Figure \ref{fig:propxi0}.



\subsection{Theoretical and numerical LDP results --  nonnegative vectors}
\label{sec:thnumresutsnonn}

In this section we finally give a little bit of a flavor to what is actually proven in Theorem \ref{thm:finalldpl1nonn}. These results are essentially nonnegative analogues to the results we presented in Section \ref{sec:thnumresuts}. Consequently, in presentation of the results, we try to maintain as much of a parallelism with Section \ref{sec:thnumresuts} as possible. In Figure \ref{fig:l1regldpIerrubnonn} we show the theoretical LDP rate function curve that one can obtain based on Theorem \ref{thm:finalldpl1nonn}. This figure is complemented by Table \ref{tab:Ildptab1nonn} where we show the numerical values for all quantities of interest in Theorems \ref{thm:ldp3nonn} and \ref{thm:finalldpl1nonn} for several $\alpha$'s from the transition zone (i.e. for several $\alpha$'s around the breaking point; here $\beta=0.27911$ is chosen such that the breaking point/threshold for $\alpha=0.5$). Finally, in Figure \ref{fig:weakl1LDPthrsimnonn} and Table \ref{tab:Ildptab2nonn} we show the comparison between the simulated values and the theoretical ones. As was the case for the general vectors in Section \ref{sec:thnumresuts}, here we again observe that even for fairly small dimensions (of order $100$) one already approaches the theoretical curves (derived of course assuming an infinite dimensional asymptotic regime). One should note though that for the nonnegative vectors the transition zone is noticeable more narrow which implies better concentration properties.


\begin{figure}[htb]
\begin{minipage}[b]{.5\linewidth}
\centering
\centerline{\epsfig{figure=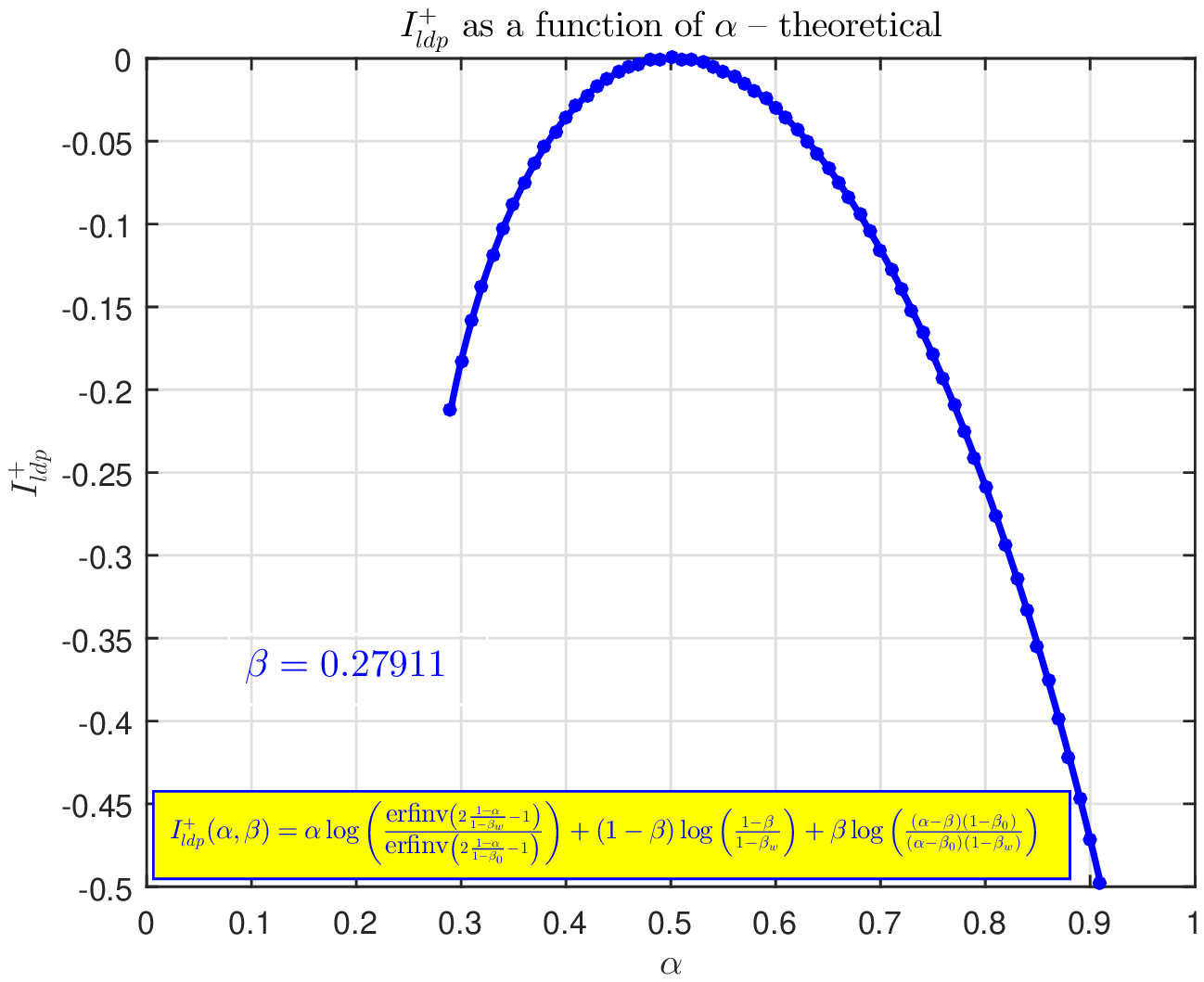,width=9cm,height=7cm}}
\end{minipage}
\begin{minipage}[b]{.5\linewidth}
\centering
\centerline{\epsfig{figure=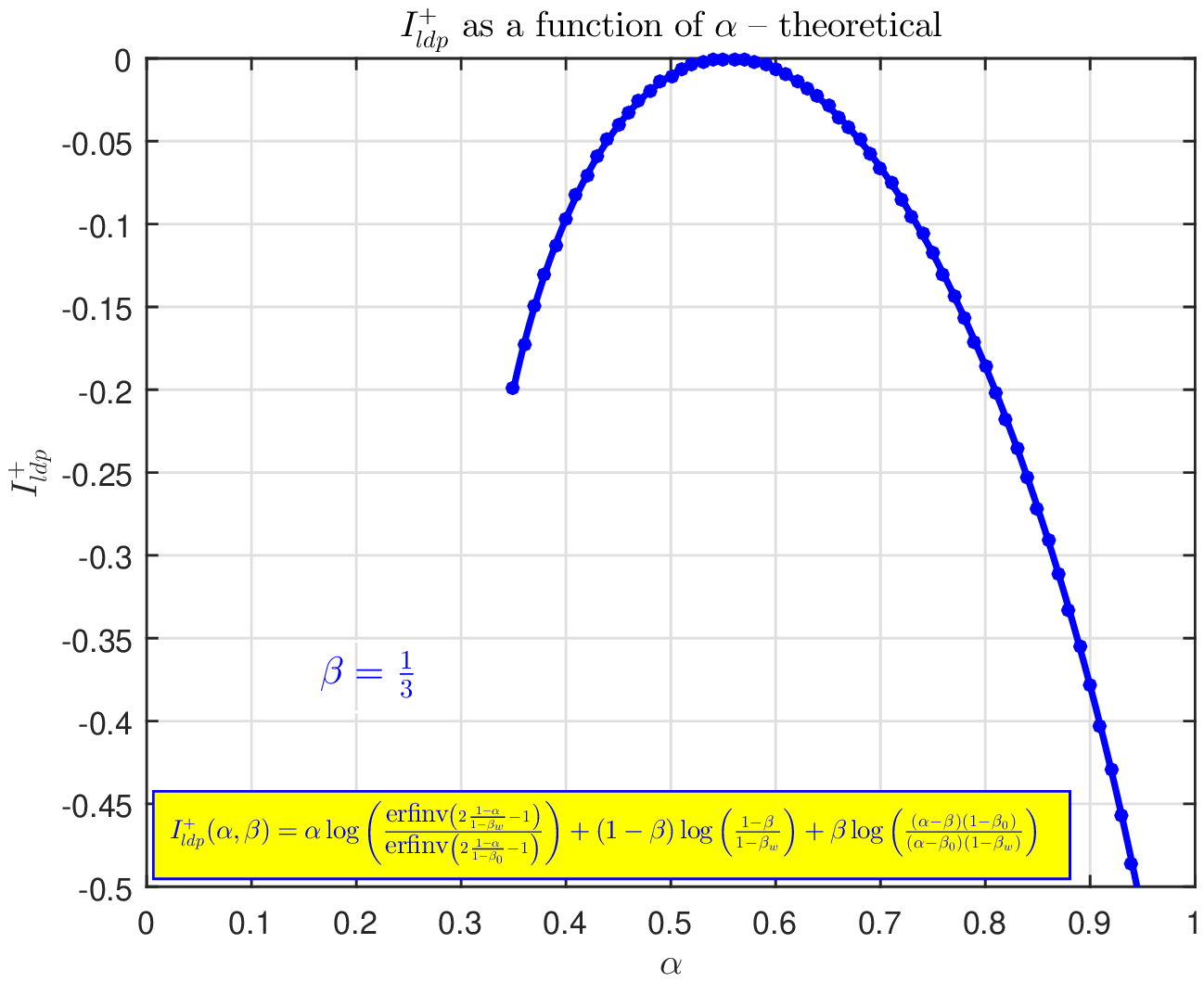,width=9cm,height=7cm}}
\end{minipage}
\caption{$I^+_{ldp}$ as a function of $\alpha$; left -- $\beta=0.27911$; right -- $\beta=\frac{1}{3}$}
\label{fig:l1regldpIerrubnonn}
\end{figure}

\begin{table}[h]
\caption{A collection of values for $\beta_w$, $\beta_0$, $\nu$, $A_0$, $c_3$, $\gamma$, and $I_{ldp}$ in Theorem \ref{thm:ldp3}; $\beta=0.27911$}\vspace{.1in}
\hspace{-0in}\centering
\begin{tabular}{||c||c|c|c|c|c||}\hline\hline
$\alpha$ & $ 0.40 $ & $ 0.45 $ & $ 0.50 $ & $ 0.55 $ & $ 0.60 $ \\ \hline\hline
$\beta_w$& $ 0.1921 $ & $ 0.2336 $ & $ 0.2791 $ & $ 0.3289 $ & $ 0.3832 $ \\ \hline
$\beta_0$& $ 0.0302 $ & $ 0.1646 $ & $ 0.2791 $ & $ 0.3792 $ & $ 0.4687 $ \\ \hline\hline
$\nu$    & $ 0.6516 $ & $ 0.5757 $ & $ 0.5061 $ & $ 0.4415 $ & $ 0.3813 $ \\ \hline
$A_0$    & $ 2.1569 $ & $ 1.4109 $ & $ 1.0000 $ & $ 0.7390 $ & $ 0.5580 $ \\ \hline
$c_3$    & $ -1.0709 $ & $ -0.4710 $ & $ -0.0000 $ & $ 0.4556 $ & $ 0.9560 $ \\ \hline
$\gamma$ & $ 0.1466 $ & $ 0.2377 $ & $ 0.3536 $ & $ 0.5018 $ & $ 0.6941 $ \\ \hline\hline
$I^+_{ldp}$& $ \mathbf{-0.0357} $ & $ \mathbf{-0.0084} $ & $ \mathbf{0.0000} $ & $ \mathbf{-0.0077} $ & $ \mathbf{-0.0299} $ \\ \hline\hline
\end{tabular}
\label{tab:Ildptab1nonn}
\end{table}

\begin{figure}[htb]
\centering
\centerline{\epsfig{figure=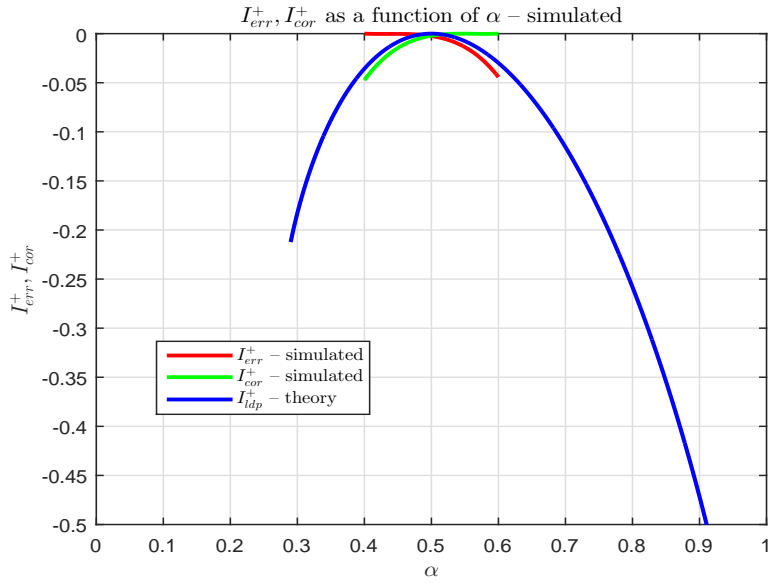,width=11.5cm,height=8cm}}
\caption{Nonnegative $\ell_1$'s weak LDP rate function -- theory and simulation; $\beta=0.27911$}
\label{fig:weakl1LDPthrsimnonn}
\end{figure}

\begin{table}[h]
\caption{$I^+_{err}$, $I^+_{err}$ -- simulated; $I^+_{ldp}$ calculated for $\beta=0.27911$}\vspace{.1in}
\hspace{-0in}\centering
\begin{tabular}{||c||c|c|c|c|c||}\hline\hline
$\alpha$ & $ 0.40 $ & $ 0.45 $ & $ 0.50 $ & $ 0.55 $ & $ 0.60 $ \\ \hline\hline
$k$      & $ 35 $ & $ 56 $ & $ 84 $ & $ 84 $ & $ 42 $ \\ \hline
$m$      & $ 50 $ & $ 90 $ & $ 150 $ & $ 165 $ & $ 90 $ \\ \hline
$n$      & $ 125 $ & $ 200 $ & $ 300 $ & $ 300 $ & $ 150 $ \\ \hline\hline
$I^+_{err}$& $ -0.0000 $ & $ -0.0003 $ & \red{$ \mathbf{-0.0025} $} & \red{$ \mathbf{-0.0138} $} & \red{$ \mathbf{-0.0443} $} \\ \hline\hline
$I^+_{cor}$& \gr{$ \mathbf{-0.0474} $} & \gr{$ \mathbf{-0.0146} $} & \gr{$ \mathbf{-0.0022} $} & $ -0.0001 $ & $ -0.0000 $ \\ \hline\hline
$I^+_{ldp}$& \bl{$ \mathbf{-0.0357} $} & \bl{$ \mathbf{-0.0084} $} & \bl{$ \mathbf{0.0000} $} & \bl{$ \mathbf{-0.0077} $} & \bl{$ \mathbf{-0.0299} $} \\ \hline\hline
\end{tabular}
\label{tab:Ildptab2nonn}
\end{table}

\subsection{High-dimensional geometry approach of \cite{DT,DonohoSigned} -- nonnegative vectors}
\label{sec:donhdgnonn}

In this section we again look at an alternative high-dimensional approach. The approach is popularized in \cite{DT,DonohoSigned} and in its essence is adapting the one from \cite{DonohoPol,DonohoUnsigned} so that it can handle the nonnegative vectors and their phase transitions. Moreover, it is now well known through our own work \cite{StojnicEquiv10} that the nonnegative $\ell_1$'s phase transitions from Theorem \ref{thm:thmweakthrnonn} and the results from \cite{DT,DonohoSigned} are in a perfect mathematical agreement.

Additionally, the results of \cite{DT,DonohoSigned} can follow into the footsteps of \cite{DonohoPol,DonohoUnsigned} and can also be used for the LDP characterizations. Below is a quick sketch how one can align such an analysis with what we presented earlier. As usual, we will first assume that we are given a pair $(\alpha,\beta)$ and, as in Section \ref{sec:donhdg}, we will immediately write the results for both, the upper and the lower LDP regimes, i.e. for $\alpha>\alpha_w$ and for $\alpha<\alpha_w$ (where $\alpha_w$ is such that $\psi^+_\beta(\alpha_w)=\xi^+_{\alpha_w}(\beta)=1$). When put in the LDP frame of Section \ref{sec:hdg} and ultimately \cite{StojnicEquiv10} results of \cite{DT,DonohoSigned} give
\begin{eqnarray}\label{eq:donhdg1nonn}
  \Psi_{net}^{(D,+)}(\alpha,\beta) & =&I^+_{err}(\alpha,\beta)\triangleq\lim_{n\rightarrow\infty}\frac{\log{P^+_{err}}}{n}=\psicom^{(D,+)}-\psiint^{(D,+)}-\psiext^{(D,+)},\alpha>\alpha_w\nonumber \\
  \Psi_{net}^{(D,+)}(\alpha,\beta) & =&I^+_{cor}(\alpha,\beta)\triangleq\lim_{n\rightarrow\infty}\frac{\log{P^+_{err}}}{n}=\psicom^{(D,+)}-\psiint^{(D,+)}-\psiext^{(D,+)},\alpha<\alpha_w,
\end{eqnarray}
where
\begin{eqnarray}
\psicom^{(D,+)} & = & -(\alpha-\beta)\log\left (\frac{\alpha-\beta}{1-\beta}\right )-(1-\alpha)\log\left (\frac{1-\alpha}{1-\beta}\right )\nonumber \\
\psiext^{(D,+)} & = & \min_{y\geq 0} (\alpha y^2 -(1-\alpha)\log(1+\erf(y))+(1-\alpha)\log(2)), \label{eq:donhdg2nonn}
\end{eqnarray}
and
\begin{equation}
\psiint^{(D,+)}=(\alpha-\beta)\left (-\frac{1}{2}\frac{\beta}{\alpha-\beta}s_{\alpha,\beta}^2-\frac{1}{2}\log\left (\frac{2}{\pi}\right )+\log\left (\frac{\alpha s_{\alpha,\beta}}{\alpha-\beta}\right )+\log(2)\right ),\label{eq:intang1nonn}
\end{equation}
where $s_{\alpha,\beta}\geq 0$ is the solution of (\ref{eq:intang3}).
Now if we can show that $\psiint^{(D,+)}=-\psiintnon$ then $\psinet^{(D,+)}=\psinetnon$ and the approach of \cite{DT,DonohoSigned} indeed matches the approach of Section \ref{sec:hdg}. To that end, we follow into the footsteps of \cite{StojnicEquiv10} and set
\begin{equation}
s^+_{\alpha,\beta}=\sqrt{2}\erfinv\left (2\frac{1-\alpha}{1-\beta_0}-1\right ),\label{eq:sgammanonn}
\end{equation}
where as earlier $\beta_0$ is such that $\frac{\alpha-\beta}{\alpha-\beta_0}\xi^+_{\alpha}(\beta_0)=1$. For such a $s^+_{\alpha,\beta}$ (\ref{eq:intang3}) becomes
\begin{equation}
\left (\frac{\alpha-\beta_0}{1-\beta_0}\right )=\frac{\alpha-\beta}{\alpha}\frac{e^{-\left (\erfinv\left (2\frac{1-\alpha}{1-\beta_0}-1\right )\right )^2}}{\sqrt{2\pi}\sqrt{2}\erfinv\left (2\frac{1-\alpha}{1-\beta_0}-1\right )},\label{eq:intang3anonn}
\end{equation}
which is true because of $\frac{\alpha-\beta}{\alpha-\beta_0}\xi^+_{\alpha}(\beta_0)=1$. Now replacing $s^+_{\alpha,\beta}$ from (\ref{eq:sgamma}) in (\ref{eq:intang1nonn}) we obtain
\begin{eqnarray}
\psiint^{(D,+)}&=&(\alpha-\beta)\left (-\frac{1}{2}\frac{\beta}{\alpha-\beta}(s^+_{\alpha,\beta})^2-\frac{1}{2}\log\left (\frac{2}{\pi}\right )+\log\left (\frac{\alpha s^+_{\alpha,\beta}}{\alpha-\beta}\right )+\log(2)\right )
\nonumber \\
&=&(\alpha-\beta)\left (-\frac{\beta}{\alpha-\beta}\left (\erfinv\left (2\frac{1-\alpha}{1-\beta_0}-1\right )\right )^2-\frac{1}{2}\log\left (\frac{2}{\pi}\right )+\log\left (\frac{\sqrt{\frac{2}{\pi}}(1-\beta_0) e^{-\left (\erfinv\left (2\frac{1-\alpha}{1-\beta_0}-1\right )\right )^2}}{\alpha-\beta_0}\right )\right )\nonumber \\
& = & -\alpha \left (\erfinv\left (2\frac{1-\alpha}{1-\beta_0}-1\right )
\right )^2 -(\alpha-\beta)\log\left (\frac{\alpha-\beta_0}{1-\beta_0}\right ).
\label{eq:intang1anonn}
\end{eqnarray}
Connecting (\ref{eq:intang1anonn}) and the second equality in (\ref{eq:hdg12nonn}) then confirms that indeed $\psiint^{(D,+)}=-\psiintnon$ and finally $\psinet^{(D,+)}=\psinetnon$.

\section{Conclusion}
\label{sec:conc}

In this paper we revisited the random under-determined systems of linear equations with sparse solutions. We looked at the classical phase transitions phenomena that appear in these systems if a standard optimization algorithm/thenique called $\ell_1$ minimization is utilized for their solving. We substantially widened the scope of studying and understanding these phenomena by connecting them further to the large deviations properties/principles from the classical probability theory. We first introduced and explained what could be a way of thinking about large deviations when it comes to random linear systems and their dimension. We then continued by providing a series of novel probabilistic mechanisms that turned out to be fairly powerful and enabled us to fully exactly characterize the introduced large deviations concepts. Moreover, the final results turned out to be fairly elegant and in our view match the elegance we achieved in phase transitions characterizations in our earlier works.

We started the presentation by introducing the main ideas while considering the case of the general unknown sparse vectors and then proceeded by adapting them so that they fit the case of the a priori known to be nonnegative vectors. For that case one first modifies the $\ell_1$ optimization to the so-called nonnegative $\ell_1$ and then proceeds with the adaptation of the general methodology that works for the $\ell_1$. For both cases we connected the probabilistic analysis to the one that can be conducted through a high-dimensional geometry approach and showed that one obtains exactly the same results pursuing both of these substantially different mathematical paths. Finally, we presented quite a few numerical results that are in a very solid agreement with all of our theoretical results (we in fact observed a pretty good level of agreement between the theoretical results that are derived assuming an infinite dimensional asymptotic regime and the simulated ones obtained for systems of rather small dimensions of order of few hundreds).

As expected, the design of a theory as powerful and widely applicable as the one that we presented here then leaves a tone of opportunity to continue further and consider various other aspects/extensions of the algorithms/problems at hand. That typically assumes a bit of adjustment of the techniques introduced here and in a few of our earlier works so that they fit those problems as well. However, we view these adjustments as fairly routine tasks and for a few problems that we consider of particular interest we will in a few companion papers present how they can be done and what kind of results they eventually produce.

\begin{singlespace}
\bibliographystyle{plain}
\bibliography{l1regposldpasym1Refs}
\end{singlespace}

\end{document}